\documentclass[10pt]{article}

\usepackage{amsmath, amsfonts, amssymb, amsthm}
\usepackage{enumerate}

\allowdisplaybreaks[1]
\usepackage{microtype}

\usepackage{nicefrac}
\usepackage{enumitem}

\allowdisplaybreaks

\usepackage{calc}

\usepackage{pgf}

\newcommand{\R}{\mathbb{R}}

\newcommand{\N}{\mathbb{N}}
\newcommand{\Q}{\mathbb{Q}}

\newcommand{\ee}{\mathrm{e}}
\newcommand{\dx}{\mathrm{d}x}
\newcommand{\dy}{\mathrm{d}y}
\newcommand{\dz}{\mathrm{d}z}
\newcommand{\dX}{\mathrm{d}X}
\newcommand{\dY}{\mathrm{d}Y}
\newcommand{\dZ}{\mathrm{d}Z}

\newcommand{\ds}{\mathrm{d}s}
\newcommand{\dt}{\mathrm{d}t}
\newcommand{\dtau}{\mathrm{d}\tau}
\newcommand{\dd}{\mathrm{d}}
\newcommand{\del}{\partial}
\newcommand{\eps}{\varepsilon}

\newcommand{\esp}{r}

\newcommand{\Y}{\mathcal{Y}}

\newcommand{\const}{P}

\newcommand{\Mfin}{\mathcal{M}^{\mathrm{fin}}}

\newcommand{\T}[1]{{\mathcal{T}\left[#1\right]}}

\newcommand{\abs}[1]{\left| #1 \right|}
\newcommand{\norm}[1]{\left\| #1 \right\|}

\DeclareMathOperator*{\supp}{supp} 

\theoremstyle{plain}
\newtheorem{theorem}{Theorem}[section]
\newtheorem{lemma}[theorem]{Lemma}
\newtheorem{proposition}[theorem]{Proposition}
\newtheorem{corollary}[theorem]{Corollary}

\theoremstyle{definition}
\newtheorem{definition}[theorem]{Definition}

\theoremstyle{remark}
\newtheorem{remark}[theorem]{Remark}

\usepackage{footmisc}

\begin{document}

\title{Self-similar solutions with fat tails for Smoluchowski's coagulation equation with singular kernels}
\author{B. Niethammer \and S. Throm \and J.~J.~L.~Vel\'azquez}

%
%


\date{}
\maketitle
Institute of Applied Mathematics, University of Bonn, Endenicher Allee 60, 53115 Bonn, Germany

E-mail: \texttt{niethammer@iam.uni-bonn.de}; \texttt{throm@iam.uni-bonn.de}; \texttt{velazquez@iam.uni-bonn.de}

\begin{abstract}

 We show the existence of self-similar solutions with fat tails for Smoluchowski's coagulation equation for homogeneous kernels satisfying $C_1 \left(x^{-a}y^{b}+x^{b}y^{-a}\right)\leq K\left(x,y\right)\leq C_2\left(x^{-a}y^{b}+x^{b}y^{-a}\right)$ with $a>0$ and $b<1$. This covers especially the case of Smoluchowski's classical kernel $K(x,y)=(x^{1/3} + y^{1/3})(x^{-1/3} + y^{-1/3})$.

 For the proof of existence we first consider some regularized kernel $K_{\eps}$ for which we construct a sequence of solutions $h_{\eps}$. In a second step we pass to the limit $\eps\to 0$ to obtain a solution for the original kernel $K$. The main difficulty is to establish a uniform lower bound on $h_{\eps}$. The basic idea for this is to consider the time-dependent problem and choosing a special test function that solves the dual problem.

%
%

\end{abstract}

\section{Introduction}

\subsection{Smoluchowski's equation and self-similarity}

Smoluchowski's coagulation equation \cite{Smolu16} describes irreversible
aggregation of clusters through binary collisions by a mean-field model for
the density $f(\xi,t)$ of clusters of mass $\xi$. It is assumed that the rate
of coagulation of clusters of size $\xi$ and $\eta$ is given by a rate kernel
$K=K(\xi,\eta)$, such that the evolution of $f$ is determined by
\begin{equation}
\partial_{t}f(\xi,t)=\frac{1}{2}\int_{0}^{\xi}K(\xi-\eta,\eta)f(\xi-\eta
,t)f(\eta,t)\dd\eta-f(\xi,t)\int_{0}^{\infty}K(\xi,\eta)f(\eta,t)\dd\eta\,.
\label{smolu1}%
\end{equation}
Applications in which this model has been used are numerous and include, for example,
aerosol physics, polymerization, astrophysics and mathematical biology (see
e.g. \cite{Aldous99,Drake72}).

A topic of particular interest in the theory of coagulation is the scaling hypothesis
on the long-time behaviour of solutions to \eqref{smolu1}.
Indeed, for  homogeneous kernels one expects that solutions converge to a uniquely determined self-similar profile.
 This issue is however only well-understood for the
solvable kernels $K(x,y)=2$, $K(x,y)=x+y$ and $K(x,y)=xy$. In these cases it
is known that \eqref{smolu1} has one fast-decaying self-similar solution with
finite mass and a family of so-called fat-tail self-similar solutions with
power-law decay. Furthermore, their domains of attraction under the evolution
\eqref{smolu1} have been completely characterized in \cite{MePe04}. For
non-solvable kernels much less is known and it is exclusively for the case $\gamma<1$. In
\cite{EMR05,FouLau05} existence of self-similar solutions with finite mass has
been established for a large range of kernels and some properties of those
solutions have been investigated in \cite{CanMisch11,EsMisch06,FouLau06a}.
More recently, the first existence results of self-similar solutions with fat
tails have been proved, first for the diagonal kernel \cite{NV11a}, then for
kernels that are bounded by $C(x^{\gamma}+y^{\gamma})$ for $\gamma\in
[0,1)$ \cite{NV12a}. It is the goal of this paper to extend the results
in \cite{NV12a} to singular kernels, such as Smoluchowski's classical kernel $K(x,y)=
(x^{1/3} + y^{1/3})(x^{-1/3} + y^{-1/3})$. Uniqueness of solutions with both, finite and infinite mass is still one of the main problems for non-solvable kernels and in most cases an open question. Only recently uniqueness has been shown in the finite mass case for kernels that are in some sense close to the constant kernel \cite{NV14}.

In order to describe our results in more detail, we first derive the equation
for self-similar solutions. Such solutions to \eqref{smolu1} for kernels of
homogeneity $\gamma<1$ are of the form
\begin{equation}
f(\xi,t)=\frac{\beta}{t^{\alpha}}g\left(x\right)\,,\qquad
\alpha=1+(1{+}\gamma)\beta\,, \qquad x=\frac{\xi}{t^{\beta}}\label{ss1}%
\end{equation}
where the self-similar profile $g$ solves
\begin{equation}
-\frac{\alpha}{\beta} g- xg^{\prime}(x)=\frac{1}{2}\int_{0}%
^{x}K(x-y,y)g(x-y)g(y)\dy-g(x)\int_{0}^{\infty}K(x,y)g(y)\dy\,.
\label{ss2}%
\end{equation}
It is known that for some kernels the self-similar profiles are singular at
the origin, so that the integrals on the right-hand side are not finite and it
is necessary to rewrite the equation in a weaker form. Multiplying the equation by $x$ and
rearranging we obtain that a weak self-similar solution $g$ solves
\begin{equation}
\partial_{x}(x^{2}g(x))=\partial_{x}\left[\int_{0}^{x}\int_{x-y}%
^{\infty}yK(y,z)g(z)g(y)\,\dz\,\dy\right]+\left ((1-\gamma)-\frac{1}{\beta}\right)xg(x)\, \label{ss3}%
\end{equation}
in a distributional sense. If one in addition requires that the solution has
finite first moment, then this also fixes $\beta=1/(1-\gamma)$ and in this
case the second term on the right hand side of \eqref{ss3} vanishes.

For the following it is convenient to go over to the monomer density function
$h\left(  x,t\right)  =xg\left(  x,t\right)  $ and to introduce the parameter
$\rho=\gamma+\frac{1}{\beta}$. Then equation \eqref{ss3} becomes
\begin{equation}
\partial_{x}\left[  \int_{0}^{x}\int
_{x-y}^{\infty}\frac{K\left(  y,z\right)  }{z}h\left(  z\right)  h\left(
y\right)  \,\dz\,\dy\right]  -\left[  \partial_{x}\left(  xh\right)
+\left(  \rho-1\right)  h\right]  \left(  x\right)  =0\,. \label{A2a}%
\end{equation}
Our approach to find a solution to \eqref{A2a} requires to work with the corresponding evolution equation. 
Using as new time variable $\log\left(
t\right)  $ which will be denoted as $t$ from now on, the time dependent
version of equation \eqref{A2a} becomes
\begin{equation}
\partial_{t}h\left(  x,t\right)  +\partial_{x}\left[  \int_{0}^{x}\int
_{x-y}^{\infty}\frac{K\left(  y,z\right)  }{z}h\left(  z,t\right)  h\left(
y,t\right)  \,\dz\,\dy\right]  -\left[  \partial_{x}\left(  xh\right)
+\left(  \rho-1\right)  h\right]  \left(  x,t\right)  =0\,, \label{A2}%
\end{equation}
with initial data
\begin{equation}
h(x,0)=h_{0}(x)\,. \label{A3}%
\end{equation}

\subsection{Assumptions on the kernel and main result}

We now formulate our assumptions on the kernel $K$. We assume that
\begin{equation}
K\in C^{1}((0,\infty)\times(0,\infty))\,,\qquad K(x,y)=K(y,x)\geq
0\qquad\mbox{ for all }x,y\in(0,\infty)\,, \label{Ass1a}%
\end{equation}
$K$ is homogeneous of degree $\gamma\in(-\infty,1)$, that is
\begin{equation}
K(\lambda x,\lambda y)=\lambda^{\gamma}K(x,y)\qquad\mbox{ for all }x,y\in
(0,\infty)\,, \label{Ass1b}%
\end{equation}
and satisfies the growth condition
\begin{equation}
C_{1}\left (x^{-a}y^{b}+x^{b}y^{-a}\right)\leq K(x,y)\leq C_{2}\left (x^{-a}%
y^{b}+x^{b}y^{-a}\right)\,\qquad\mbox{ for all }x,y\in(0,\infty)\,, \label{Ass1}%
\end{equation}
where $a>0$, $b<1$, $\gamma=b-a$, and $C_{1},C_{2}$ are positive constants. Furthermore we assume the following locally uniform bound on the partial derivative: for each interval $\left[d,D\right]\subset \left(0,\infty\right)$ there exists a constant $C_{3}=C_{3}\left(d,D\right)>0$ such that 
\begin{equation}\label{eq:Ass2}
 \abs{\del_{x}K\left(x,y\right)}\leq C_{3}\left(y^{-a}+y^{b}\right) \quad \text{for all } x\in \left[d,D\right] \text{ and } y\in\left(0,\infty\right).
\end{equation}

Let us first discuss what we can expect on the possible decay behaviours of self-similar
solutions. If $h(x) \sim C x^{-\rho}$ as $x
\to\infty$, then in order for $\int_{1}^{\infty} \frac{K(x,y)}{y} h(y)\,dy <
\infty$ we need
\begin{equation}
\label{Ass1c}\rho>b=\gamma+a \qquad\mbox{ and } \qquad\rho+ a>0\,.
\end{equation}

Note that since $\gamma$ can be negative, $-a$ can be larger than $b$.
Furthermore we need to assume that $b<1$ since for $b>1$  we could have instantaneous gelation and
$b=1$ is a borderline case that can also not be treated with our methods. The same assumption has also been made
in related work, where, for example in \cite{CanMisch11}, regularity of self-similar solutions with finite mass have been investigated. In addition it will turn out later that we have to assume $\rho>0$ (see: Lemma~\ref{Lem:non:solv:limit}).

Our main result can now be formulated as follows

\begin{theorem}
\label{T.main} Let $K$ be a kernel that satisfies assumptions
\eqref{Ass1a}-\eqref{eq:Ass2} for some $b \in(-\infty,1)$ and $a>0$. Then for any
$\rho \in (\max(-a,b,0),1) = (\max(b,0),1)$ there exists a non-negative measure  $h\in \mathcal{M}([0,\infty))$ that
solves \eqref{A2a} in the sense of distributions. This solution decays in the expected manner in an averaged sense, i.e. it satisfies $\int_{\left[0,R\right]}h\dx\leq R^{1-\rho}$ for all $R>0$ and for each $\delta>0$ there exists $R_{\delta}>0$ such that 
\begin{align*}
 \left(1-\delta\right)R^{1-\rho}\leq \int_{\left[0,R\right]}h\dx \quad \text{for all } R\geq R_{\delta},
\end{align*}
which together implies $\lim_{R\to \infty}\frac{1}{R^{1-\rho}}\int_{\left[0,R\right]}h\dx=1$.
\end{theorem}

\begin{remark} One can in fact show that under the assumptions \eqref{Ass1a}-\eqref{eq:Ass2} the measure $h$ has a continuous density
and satisfies $h(r) \sim (1-\rho) r^{-\rho}$ as $r \to \infty$. This has been proved in the case of locally bounded kernels
in \cite{NV12a} and the proof in the present case proceeds similarly. Furthermore, if $K$ is more regular, one can also establish
higher regularity of $h$ in $(0,\infty)$. In order to keep the present paper within a reasonable length we will give
the corresponding proofs in a subsequent separate paper.
\end{remark}

\subsection{Strategy of the proof}

The proof of Theorem~\ref{T.main} consists of two main parts. In the first one which is contained in Section~\ref{sec:sol:heps} we shift the singularities of the kernel by some $\eps>0$ to get a kernel $K_{\eps}$ that is bounded at the origin. The idea is then to prove Theorem \ref{T.main} with this modified kernel to get a solution $h_{\eps}$. The proof follows the one in \cite{NV12a}, i.e. the existence of a stationary solution to \eqref{A2} is shown by using the following variant of Tikhonov's fixed point theorem.

\begin{theorem}[Theorem~1.2 in \cite{EMR05,GPV04}]\label{T.fixedpoint}
 Let $X$ be a Banach space and $(S_t)_{t \geq 0}$ be a continuous semi-group on $X$. Assume that $S_t$ is weakly sequentially 
continuous for any $t>0$ and that there exists a subset ${\cal Y}$ of $X$  that is nonempty, convex, weakly sequentially
compact and invariant under the action of $S_t$. Then there exists $z_0 \in {\cal Y}$ which is stationary under the action of $S_t$.
\end{theorem}

As most of the estimates from \cite{NV12a} remain valid for the shifted kernel $K_{\eps}$ we only state the main definitions and results and refer to \cite{NV12a} for the proofs. The only exception is the invariance of some lower bound defining the set $\Y$ from Theorem~\ref{T.fixedpoint}. As this step cannot just be transferred to the present situation we will give the full proof of this. The main idea here is to construct a special test function that solves the dual problem, for which one can derive some lower bounds that are sufficient to obtain the invariance (Section~\ref{Sec:dual:prob}). 

In the second part which is contained in Section~\ref{sec:eps:to:zero} we have to remove the shift in $K_{\eps}$, i.e we have to take the limit $\eps\to 0$. The strategy here is similar to what is done in the first part as one of the main difficulties consists in showing a suitable lower bound (uniform in $\eps$) for $h_{\eps}$ (Section~\ref{subsec:iterate}). This will again be done by constructing a suitable test function by solving the dual problem for which we get adequate estimates from below (Section~\ref{subsec:test}). One difficulty then is to show that the functions $h_{\eps}$ obtained before decay sufficiently rapidly at the origin as $\eps\to 0$ (Section~\ref{subsec:exp:decay}). In fact we will get some exponential decay that will be enough to pass to the limit $\eps\to 0$ (Section~\ref{sec:limit:eps}).

The proofs of the existence of the solutions to the dual problems as well as some basic properties and estimates frequently used are contained in the appendix.

\section{Stationary solutions for the kernel $K_{\eps}$}\label{sec:sol:heps}

In this section we let $\eps>0$ be fixed and consider the kernel 
\[
 K_{\eps}(y,z):=K(y+\eps,z+\eps).
\]
We prove the following Proposition:

\begin{proposition}\label{P.hepsexistence}
For any
$\rho\in(\max(b,0),1)$ there exists a continuous function $h_{\eps} \colon (0,\infty) \to [0,\infty)$
that is a weak  solution to \eqref{A2a} with $K$ replaced by $K_{\eps}$. This solution satisfies 
\begin{align*}
 \int_{0}^r h_{\eps}(x)\,dx&\leq r^{1-\rho} \qquad \mbox{ and } \qquad \lim_{r \to \infty} \frac{\int_0^r h_{\eps}(x)\,dx}{r^{1-\rho}}=1\,.
\end{align*}
\end{proposition}

\subsection{Plan of the construction for $h_{\eps}$}

The proof of Proposition~\ref{P.hepsexistence} follows closely the proof of Theorem~1.1 in \cite{NV12a}. As the estimates remain in principle the same here we just recall the strategy of the proof and state the main definitions and results while for proofs we refer to \cite{NV12a}. The only modification we have to establish, compared to \cite{NV12a}, is the proof of the invariance of some lower bound that cannot just easily be adapted and we will show this in Section~\ref{S.le}.

The strategy to find a solution to~\eqref{A2a} (with $K$ replaced by $K_{\eps}$) will be to show that the evolution given by \eqref{A2} satisfies the assumptions of Theorem \ref{T.fixedpoint} (notice that it suffices that the respective properties hold in a possibly small time interval $[0,T]$). One key point in the application of this theorem is obviously an appropriate choice of $X$ and ${\cal Y}$. Here we use for $X$ the set of measures on $[0,\infty)$ and as ${\cal Y}$ the set of non-negative measures which satisfy the expected decay behaviour in an averaged sense (cf. Definition 
\ref{Def:InvSet}). As well-posedness of~\eqref{A2} is not so easy to show for $K_{\eps}$ directly we introduce a regularized problem where we cut the kernel $K_{\eps}$ (in a smooth way) for small and large cluster sizes in the following way: for $\lambda>0$ we consider
\begin{equation}\label{eq:A1}
 \begin{aligned}
  K_{\eps}^{\lambda}\left(x,y\right)&=K_{\eps}\left(x,y\right), &&\text{ if } \lambda \leq \min\left\{x,y\right\} \text{ and } \max\left\{x,y\right\} \leq \frac{1}{\lambda},\\
  K_{\eps}^{\lambda}\left(x,y\right)&=0, &&\text{ if } \min\left\{x,y\right\}\leq\frac{\lambda}{2} \text{ and } \max\left\{x,y\right\}\geq \frac{3}{2\lambda},\\
  K_{\eps}^{\lambda}&\leq K_{\eps}.
 \end{aligned}
\end{equation}
\begin{remark}
 Note that in \cite{NV12a} a slightly different cutoff was used but this does not cause any problem.
\end{remark}
In the rest of this section we assume that in all equations the kernel $K$ is replaced by $K_{\eps}^{\lambda}$ (or later by $K_{\eps}$ when we take the limit $\lambda\to 0$). 

We are now going to prove the well-posedness of~\eqref{A2} for the kernel $K_{\eps}^{\lambda}$ with $\lambda>0$. We will consider the set of non-negative Radon measures that we will denote with some abuse of notation by $h(x)\,\mbox{d}x$ and such that the norm defined in \eqref{eq:S1E3} is finite. We notice that this implies that $h\mbox{d}x$ does not contain a Dirac at the origin. Since, however, $h\mbox{d}x$ might contain
Dirac measures away from the origin, we use the convention that integrals such as  $\int_a^b h(x)\mbox{d}x$ are always understood in the sense $\int_{[a,b]}h(x)\mbox{d}x$.

\begin{definition}
 Given $\rho\in \left( \max\left\{0,b\right\},1 \right)$ with $b$ as in Assumption~\eqref{Ass1}, we will denote as $\mathcal{X}_{\rho}$ the set of measures $h\in \mathcal{M}_{+}\left( \left[0,\infty\right) \right)$ such that
\begin{align}\label{eq:S1E3}
	\norm{h}:=\sup_{R\geq 0}\frac{\int_{\left[0,R\right]}h\left( x \right)\dx}{R^{1-\rho}}<\infty\,.
\end{align}

\end{definition}

We introduce a suitable topology in $\mathcal{X}_{\rho}$. We define the neighbourhoods of $h_{*}\in \mathcal{X}_{\rho}$ by means of the intersections of sets of the form
\begin{align}\label{eq:S2E6}
 \mathcal{N}_{\phi,\epsilon}:=\left\{ h\in \mathcal{X}_{\rho}\colon \abs{\int_{\left[ 0,\infty \right)}\left( h-h_{*} \right)\phi \dx}<\epsilon \right\}, \quad \phi\in C_{c}\left( \left[ 0, \infty \right) \right), \quad \epsilon>0.
\end{align}

We now define the subset $\Y$, for which we will show that it remains invariant under the evolution defined by \eqref{A2}.
\begin{definition}\label{Def:InvSet}
Given $R_0>0$ and  $\delta>0$ we will denote by $\Y$ the family of measures $h\in \mathcal{X}_{\rho}$ satisfying the following inequalities
\begin{align}
\int_{\left[0,r\right]}h\dx&\leq r^{1-\rho}, \quad \text{for all }r\geq 0 \label{eq:F1}\\
\int_{\left[0,r\right]}h\dx&\geq r^{1-\rho}\left( 1-\frac{R_0^{\delta}}{r^{\delta}} \right)_{+} \quad \text{for all }r> 0. \label{eq:F2}
\end{align}
\end{definition}

\begin{remark}
We will not make the dependence of $\Y$ on the variables $R_0$ and $\delta$ explicit.
\end{remark}

We then easily see that

\begin{lemma}\label{Prop:weakcontin}
The sets $\Y\subset \mathcal{X}_{\rho}$ defined in Definition~\ref{Def:InvSet} are weakly sequentially compact.
\end{lemma}

\subsection{Well-posedness of the evolution equation }\label{S.wp}

We first have to make sure that the evolution equation \eqref{A2}-\eqref{A3}
with $K$ replaced by $K_{\eps}^{\lambda}$ as in \eqref{eq:A1} is well-posed. We are going to construct first a mild solution of \eqref{A2}. For that purpose we introduce the rescaling  
\begin{align}\label{eq:S1E7}
 X=x\ee^t,\quad h\left( x,t \right)=H\left(X,t \right)
\end{align}
and  get
\begin{align}
 \del_t H\left( X,t \right)-\rho H\left(X,t\right)+\del_{X}\left[ \int_{0}^{X}\int_{X-Y}^{\infty}\frac{K_{\lambda}^{\eps}\left( Y\ee^{-t},Z\ee^{-t} \right)}{Z}H\left( Z,t \right)H\left( Y,t \right)\dZ\dY \right]&=0\label{eq:S1E8}\\
H\left( X,0 \right)&=h_{0}\left( X \right).\label{eq:S1E8a}
\end{align}
Expanding the derivative $\del_{X}$ we find that \eqref{eq:S1E8} is equivalent to 
\begin{align}\label{eq:S1E9}
 \del_{t} H\left( X,t \right)+\left( \mathcal{A}\left[H\right]\left( X,t \right) \right)H\left( X,t \right)-\mathcal{Q}\left[ H \right] \left( X,t \right)=0
\end{align}
with
\begin{align*}
 \mathcal{A}\left[ H \right]\left( X,t \right)&:=\int_{0}^{\infty}\frac{K_{\eps}^{\lambda}\left( X\ee^{-t},Y\ee^{-t} \right)}{Y}H\left( Y,t \right)\dY-\rho\\
\mathcal{Q}\left[ H \right] \left( X,t \right)&:= \int_{0}^{X} \frac{K_{\eps}^{\lambda}\left( Y\ee^{-t},\left(X-Y\right)\ee^{-t}\right)}{X-Y}H\left(X-Y,t\right)H\left(Y,t\right)\dY\,.
\end{align*}

\begin{definition}\label{Def:mild}
 We will say that a function $H\in C\left( \left[ 0, T\right], \mathcal{X}_{\rho} \right)$ is a mild solution of equation~\eqref{eq:S1E9} if the following identity holds in the sense of measures
\begin{align}\label{eq:S2E4}
 H\left( \cdot, t \right)=\T{H}\qquad \text{for } 0\leq t\leq T,
\end{align}
where
\begin{equation}\label{eq:S2E5}
 \begin{split}
  \T{H}\left( X,t \right)&=\exp\left( -\int_{0}^{t}\mathcal{A}\left[H\right]\left( X,s \right)\ds \right)h_{0}\left( X \right)\\
&\quad +\int_{0}^{t}\exp\left( -\int_{s}^{t}\mathcal{A}\left[H\right]\left( X,\tau \right)\dtau \right)\mathcal{Q}{H}\left( X,s \right)\ds.
 \end{split}
\end{equation}
\end{definition}

For this notion of solutions we have the following existence result that follows by the contraction mapping principle.

\begin{theorem}\label{Thm:calexistence}
 Let $K$ satisfy Assumptions~\eqref{Ass1a}-\eqref{Ass1} and let $K_{\eps}^{\lambda}$ be as in \eqref{eq:A1} for $\lambda>0$. 
Then there exists $T>0$ such that there exists a unique mild solution of \eqref{eq:S1E9} in $\left( 0, T \right)$ in the sense of Definition~\ref{Def:mild}.         
\end{theorem}

We furthermore introduce weak solutions in the following sense:

\begin{definition}\label{Def:weak}
 We say that $h\in C \left( \left[ 0,T \right],\mathcal{X}_{\rho} \right)$ is a weak solution of \eqref{A2}, \eqref{A3} if for any $t\in \left[0,T\right]$ and any test function $\psi\in C_{c}^{1}\left( \left[0,\infty\right)\times \left[0,t\right] \right)$ we have
\begin{equation}\label{eq:S1E5}
 \begin{split}
  &\int_{0}^{\infty}h \left( x,t \right)\psi \left( x,t \right)\dx-\int_{0}^{\infty}h_0 \left( x \right)\psi \left( x,0 \right)\dx-\int_{0}^{t}\left[\int_{0}^{\infty} \del_s\psi \left( x,s \right)h \left( x,s \right)\dx \right]\ds\\
&+\int_{0}^{t}\left[ \int_{0}^{\infty}\psi \left( x,s \right)\int_{0}^{\infty}\frac{K_{\eps}^{\lambda} \left( x,z \right)}{z}h \left( z,s \right)\dz h \left( x,s \right)\dx \right]\ds\\
&-\int_{0}^{t}\left[ \int_{0}^{\infty}\psi \left( x,s \right)\int_{0}^{x}\frac{K_{\eps}^{\lambda} \left( y,x-y \right)}{x-y}h \left( x-y,s \right)h \left( y,s \right)\dy\dx \right]\ds\\
&+\int_{0}^{t}\int_{0}^{\infty}xh \left( x,s \right)\del_x \psi \left( x,s \right)\dx\ds-\left(\rho-1\right)\int_{0}^{t}\int_{0}^{\infty}h \left( x,s \right)\psi \left( x,s \right)\dx\ds=0.
 \end{split}
\end{equation}
\end{definition}

By changing variables one obtains the following lemma.

\begin{lemma}\label{Lem:weakSol}
 Suppose $H\in C\left( \left[0,T\right],\mathcal{X}_{\rho} \right)$ is a mild solution of \eqref{A2},\eqref{A3} in the sense of Definition~\ref{Def:mild}. Then $h$ is a weak solution in the sense of Definition~\ref{Def:weak}.
\end{lemma}

\subsection{Weak continuity of the evolution semi-group}\label{S.wc}

We denote the value of the mild solution $h$ of \eqref{A2}-\eqref{A3} obtained in Theorem~\ref{Thm:calexistence} with initial data $h_0$ as
\begin{align}\label{eq:S3E8}
h\left( x,t \right)=S_{\eps}^{\lambda}\left(t\right)h_0\left(x\right),\quad t>0.
\end{align}

Note that $S_{\eps}^{\lambda}\left(t\right)$ defines a mapping from $\mathcal{X}_{\rho}$ to itself. 
We also define the transformation that brings $h_0\left(x\right)$ to $H\left(X,t\right)$ 
which solves \eqref{eq:S1E8}, \eqref{eq:S1E8a} and whose existence is given by Theorem~\ref{Thm:calexistence}. We will write
\begin{align}\label{eq:S3E8a}
H\left(X,t\right)=T_{\eps}^{\lambda}\left(t\right)h_0,\quad t>0.
\end{align} 

With this notation we can state the following Proposition giving the weak continuity of the evolution semi-group:

\begin{proposition}\label{Prop:ContSemi}
The transformation $S_{\eps}^{\lambda}\left(t\right)$ defined by means of \eqref{eq:S3E8} for any $t\in\left[0,T\right]$ is a continuous map from $\mathcal{X}_{\rho}$ into itself if $\mathcal{X}_{\rho}$ is endowed with the topology defined by means of the functionals \eqref{eq:S2E6}. 
\end{proposition}

\begin{remark}
The continuity that we obtain is not uniform in $\lambda$.
\end{remark}

The idea to prove this is to use a special test function in the definition of weak solutions. As the transformation \eqref{eq:S1E7} is continuous in the weak topology it suffices to show that $T_{\eps}^{\lambda}$ is continuous. Since it is not linear it is not enough to check continuity at $h_{0}=0$. More precisely fix $t\in\left[0,T\right]$ and consider a test function $\bar{\Psi}\left(X\right)$ with $\bar{\Psi}\in C_{c}\left(\left[0,\infty\right)\right)$. Suppose that we have $H_{1},H_{2}$ such that $T_{\eps}^{\lambda}h_{0,i}=H_{i}\left(\cdot,t\right)$ for $i=1,2$. Using the definition of weak solutions and taking the difference of the corresponding equations we obtain after some manipulations:
\begin{equation*}
 \begin{split}
    \int_{0}^{\infty}&\left( H_1\left(X,t\right)-H_2\left(X,t\right) \right)\Psi\left(X,t\right)\dX-\int_{0}^{\infty}\left( h_{0,1}\left(X\right)-h_{0,2}\left(X\right) \right)\Psi\left(X,0\right)\dX\\
     &=\int_{0}^{t}\int_{0}^{\infty}\left( H_1\left(X,s\right)-H_2\left(X,s\right) \right)\left[\del_{s}\Psi-\mathcal{T}\left[\Psi\right]\right]\left(X,s\right)\dX\ds
 \end{split}
\end{equation*}
Thus choosing the test function $\Psi$ in some suitable function space such that $\del_{s}\Psi-\mathcal{T}\left[\Psi\right]=0$ and $\Psi\left(\cdot,t\right)=\bar{\Psi}$ the claim follows (for more details see~\cite[Proposition~2.8]{NV12a}).

\subsection{Recovering the upper estimate - Conservation of \eqref{eq:F1}}\label{S.ue}

\begin{proposition}\label{Prop:upper:bound:large}
Suppose that $h_0\in\mathcal{X}_{\rho}$  satisfies \eqref{eq:F1}. Let $h\left(x,t\right)$ be as in \eqref{eq:S3E8}. Then $h\left(\cdot, t\right)$ satisfies \eqref{eq:F1} as well.
\end{proposition}

This follows in the same way as in \cite[Proposition~3.1]{NV12a} by integrating \eqref{eq:S1E8} and changing variables.

\subsection{Recovering the lower estimate - Conservation of \eqref{eq:F2}} \label{S.le}

The invariance of the lower bound will be shown in several steps. First we choose a special test function in the definition of weak solutions, more precisely
the solution of the corresponding dual problem.
 Next we derive suitable lower bounds and integral estimates for this function. Finally we show the invariance of \eqref{eq:F2}.

\subsubsection{The dual problem}\label{Sec:dual:prob}

In Definition~\ref{Def:weak} of weak solutions
we choose  $\psi$ such that
\begin{equation}\label{eq:sptest1}
 \begin{split}
 &\quad -\int_{0}^{t}\int_{\left[0,\infty\right)}\del_s\psi\left(x,s\right)h\left(x,s\right)\dx\ds\\
 &-\int_{0}^{t}\int_{\left[0,\infty\right)}\psi\left(x,s\right)\int_{0}^{\infty}\frac{K_{\eps}^{\lambda}\left(x,z\right)}{z}h\left(z,s\right)\dz h\left(x,s\right)\dx\ds\\
 &-\int_{0}^{t}\int_{\left[0,\infty\right)}\psi\left(x,s\right)\int_{0}^{x}\frac{K_{\eps}^{\lambda}\left(y,x-y\right)}{x-y}h\left(x-y,s\right)h\left(y,s\right)\dy\dx\ds\\
 &+\int_{0}^{t}\int_{\left[0,\infty\right)}x h\left(x,s\right)\del_x \psi\left(x,s\right)\dx\ds - \left(\rho-1\right)\int_{0}^{t}\int_{\left[0,\infty\right)}h\left(x,s\right)\psi\left(x,s\right)\dx\ds\leq0,
\end{split}
\end{equation}
and thus obtain
\begin{align}\label{eq:sptest2}
\int_{\left[0,\infty\right)}h\left(x,t\right)\psi\left(x,t\right)\dx\ds-\int_{\left[0,\infty\right)}h_{0}\left(x\right)\psi \left(x,0\right)\dx\geq0.
\end{align}
 After some rearrangement we find that \eqref{eq:sptest1} is satisfied if $\psi$ solves the dual problem
\begin{align}\label{eq:sptest3}
 \del_{s}\psi\left(x,s\right)+\int_{0}^{\infty}\frac{K_{\eps}^{\lambda}\left(x,z\right)}{z}h\left(z,s\right)\left[ \psi\left(x+z,s\right)-\psi\left(x,s\right) \right]\dz-x\del_{x}\psi\left(x,s\right)-\left(1-\rho\right)\psi\left(x,s\right)\geq 0
\end{align}
together with some suitable initial condition $\psi\left(x,t\right)=\psi_{0}\left(x\right)$.

With regard to \eqref{eq:sptest2}, the idea to estimate $\int_{0}^{R}h\left(x,t\right)\dx$ is to take $\psi_{0}$ as a smoothed version of $\chi_{\left(-\infty,R\right]}$ and to estimate $\psi\left(\cdot,0\right)$ from below. Using then that \eqref{eq:F2} holds for $h_{0}$ this will be enough to show that this estimate is conserved under the action of $S_{\eps}^{\lambda}$. Rescaling $X:=x\ee^{s-t}$ and $\psi\left(x,s\right)=\ee^{-\left(1-\rho\right)\left(t-s\right)}\Phi\left(x\ee^{s-t},s\right)$ we find after
some elementary computations that
equation \eqref{eq:sptest3} together with the initial condition is equivalent to 
\begin{equation}\label{eq:sptest4}
 \begin{split}
  \del_{s}\Phi\left(X,s\right)+\int_{0}^{\infty}\frac{K_{\eps}^{\lambda}\left(X\ee^{t-s},Z\ee^{t-s}\right)}{Z}h\left(Z\ee^{t-s}\right)\left[ \Phi\left(X+Z,s\right)-\Phi\left(X,s\right) \right]\dZ&\geq0\\
  \Phi\left(X,t\right)&=\psi_{0}\left(X\right).
 \end{split}
\end{equation}
This is (the rescaled version of) the dual problem.

\begin{remark}
 Note that for applying Theorem~\ref{T.fixedpoint} we only need the assumptions on a small interval $\left[0,T\right]$. So we may assume in the following that $T<1$ is sufficiently small.
\end{remark}

\subsubsection{Construction of a solution to the dual problem}\label{Sec:constr:dual:prop}

In the following we always assume without loss of generality that $\eps\leq 1$ and $R_{0}\geq 1$ (and thus we may also assume $R\geq 1$). For $0<\kappa<1$ we furthermore denote by $\varphi_{\kappa}$ a non-negative, symmetric standard mollifier such that $\supp\varphi_{\kappa}\subset \left[-\kappa,\kappa\right]$.

The idea to construct a solution $\Phi$ to the dual problem is to replace the solution $h$ and the integral kernel $K_{\eps}^{\lambda}$ in \eqref{eq:sptest4} by corresponding power laws (using \eqref{Ass1}) multiplied by a sufficiently large constant and estimating the powers of $X$ using $X\in\left[0,R\right]$. We therefore choose $\Phi$ as a solution to
\begin{equation}\label{eq:subsolution:2}
 \begin{split}
\del_{s}\Phi\left(X,s\right)&+C_{0}\eps^{-a}\max\left\{\eps^{b},1\right\}\int_{0}^{\infty}\frac{\tilde{v}_{1}\left(Z\right)}{Z}\left[\Phi\left(X+Z\right)-\Phi\left(X\right)\right]\dZ\\
&+C_{0}\eps^{-a}\left[\max\left\{\eps^{b},1\right\}+\max\left\{\eps^{b},R^{b}\right\}\right]\int_{0}^{\infty}\frac{\tilde{v}_{2}\left(Z\right)}{Z}\left[\Phi\left(X+Z\right)-\Phi\left(X\right)\right]\dZ=0
 \end{split}
\end{equation}
with initial condition $\Phi\left(X,t\right)=\chi_{\left(-\infty,R-\kappa\right]}\ast^{2}\varphi_{\kappa/2}\left(X\right)$ and $C_{0}>0$ to be fixed later and $\tilde{v}_{i}\left(Z\right):=Z^{-\omega_{i}}$ for $i=1,2$, with $\omega_{1}=\min\left\{\rho-b,\rho\right\}$ and $\omega_{2}=\rho$. Here $\ast^{n}$ denotes the $n$-fold convolution.

\begin{lemma}\label{Lem:ex:Phi:test}  
 There exists a solution $\Phi$ of \eqref{eq:subsolution:2} in $C^{1}\left(\left[0,t\right],C^{\infty}\left(\R\right)\right)$.
\end{lemma}

\begin{proof}
This is shown in Proposition~\ref{Prop:ex:dual:sum}.
\end{proof}
We furthermore define $G\left(X,s\right):=-\del_{X}\Phi\left(X,s\right)$ and $\tilde{G}$ by $G\left(X,s\right)=\frac{1}{R}\tilde{G}\left(\frac{X}{R}-1+\frac{\kappa}{R},s\right)$. Then $\tilde{G}$ solves
\begin{equation}\label{eq:subsolution:3}
 \begin{split}
  \del_{s}\tilde{G}\left(\xi,s\right)&+C_{0}\eps^{-a}\max\left\{\eps^{b},1\right\}\int_{0}^{\infty}\frac{\tilde{v}_{1}\left(R\eta\right)}{\eta}\left[\tilde{G}\left(\xi+\eta\right)-\tilde{G}\left(\xi\right)\right]\dd\eta\\
   &+C_{0}\eps^{-a}\left[\max\left\{\eps^{b},1\right\}+\max\left\{\eps^{b},R^{b}\right\}\right]\int_{0}^{\infty}\frac{\tilde{v}_{2}\left(R\eta\right)}{\eta}\left[\tilde{G}\left(\xi+\eta\right)-\tilde{G}\left(\xi\right)\right]\dd\eta=0
 \end{split}
\end{equation}
with initial datum $\tilde{G}\left(\cdot,t\right)=\delta\left(\cdot\right)\ast^{2}\varphi_{\frac{\kappa}{2R}}$. We also summarize the following properties for $\Phi$ and $\tilde{G}$ given by Remark~\ref{Rem:properies} and the choice of the initial condition:

\begin{remark}\label{Rem:properties:der}
 The function $\Phi\left(\cdot,s\right)$ given by Lemma~\ref{Lem:ex:Phi:test} is non-increasing for all $s\in\left[0,t\right]$ and satisfies:
 \begin{equation*}
  \begin{split}
   0\leq \Phi\left(\cdot,s\right)\leq 1, \quad \supp\Phi\left(\cdot,s\right)\subset \left(-\infty,R\right] \quad \text{for all } s\in \left[0,t\right] \quad \text{and} \quad  \Phi\left(X,t\right)=1 \text{ for all } X\in \left(-\infty,R-2\kappa\right].
  \end{split}
 \end{equation*}
 Furthermore $\tilde{G}$ is non-negative and satisfies
 \begin{equation*}
  \begin{split}
   \supp \tilde{G}\left(\cdot,s\right)\subset \left(-\infty,\kappa/R\right]\quad \text{and} \quad \int_{\R}\tilde{G}\left(\xi,s\right)\dd\xi=1 \quad  \text{for all } s\in\left[0,t\right].
  \end{split}
 \end{equation*}
\end{remark}

The following Lemma states the two integral bounds that will be the key in proving the invariance of~\eqref{eq:F2}:

\begin{lemma}\label{Lem:subsolution:2:integral:estimate}
There exist $\omega,\theta\in\left(0,1\right)$ such that for every $\mu\in\left(0,1\right)$ and $D>0$ we have
\begin{enumerate}
 \item  $\int_{-\infty}^{-D}\tilde{G}\left(\xi,s\right)\dd\xi\leq C\left(\frac{\kappa}{R D}\right)^{\mu}+\frac{Ct}{D^{\omega}}R^{-\theta}$,
 \item $\int_{-1}^{0}\abs{\xi}\tilde{G}\left(\xi,s\right)\dd\xi\leq C\left(\frac{\kappa}{R}\right)^{\mu}+CtR^{-\theta}$.
\end{enumerate}
\end{lemma}

\begin{proof}[Proof of Lemma~\ref{Lem:subsolution:2:integral:estimate}]
From the proof of Proposition~\ref{Prop:ex:dual:sum} we have that $\tilde{G}$ is given as $\tilde{G}=\tilde{G}_{1}\ast\tilde{G}_{2}$ where $\tilde{G}_{1}$ and $\tilde{G}_{2}$ solve
 \begin{align*}
     \del_{s}\tilde{G}_{1}\left(\xi,s\right)+C_{0}\eps^{-a}\max\left\{\eps^{b},1\right\}\int_{0}^{\infty}\frac{\tilde{v}_{1}\left(R\eta\right)}{\eta}\left[\tilde{G}_{1}\left(\xi+\eta\right)-\tilde{G}_{1}\left(\xi\right)\right]\dd\eta&=0\\
     \del_{s}\tilde{G}_{2}\left(\xi,s\right)+C_{0}\eps^{-a}\left[\max\left\{\eps^{b},1\right\}+\max\left\{\eps^{b},R^{b}\right\}\right]\int_{0}^{\infty}\frac{\tilde{v}_{2}\left(R\eta\right)}{\eta}\left[\tilde{G}_{2}\left(\xi+\eta\right)-\tilde{G}_{2}\left(\xi\right)\right]\dd\eta&=0 
 \end{align*}
with initial data $\tilde{G}_{1}\left(\cdot,t\right)=\tilde{G}_{2}\left(\cdot,t\right)=\varphi_{\frac{\kappa}{2R}}$. Then one has from Lemma~\ref{Lem:int:est:conolution}:
\begin{equation*}
 \begin{split}
  \int_{-\infty}^{-D}\tilde{G}\left(\xi,s\right)\dd\xi\leq \int_{-\infty}^{-D/2}\tilde{G}_{1}\left(\xi,s\right)\dd\xi+\int_{-\infty}^{-D/2}\tilde{G}_{2}\left(\xi,s\right)\dd\xi
 \end{split}
\end{equation*}
and 
\begin{equation*}
 \begin{split}
  \int_{-1}^{0}\abs{\xi}\tilde{G}\left(\xi,s\right)\dd\xi&\leq \int_{-1-\frac{\kappa}{2R}}^{\frac{\kappa}{2R}}\abs{\xi}\tilde{G}_{1}\left(\xi,s\right)\dd\xi+\int_{-1-\frac{\kappa}{2R}}^{\frac{\kappa}{2R}}\abs{\xi}\tilde{G}_{2}\left(\xi,s\right)\dd\xi\\
  &\leq \frac{\kappa}{R}+\int_{-2}^{0}\abs{\xi}\tilde{G}_{1}\left(\xi,s\right)\dd\xi+\int_{-2}^{0}\abs{\xi}\tilde{G}_{2}\left(\xi,s\right)\dd\xi.
 \end{split}
\end{equation*}
From Lemma~\ref{Lem:der:int:est} we obtain with 
\begin{equation*}
 \begin{split}
  N_{1}\left(\eta\right)= C\left(\eps\right)R^{-\omega_{1}}\eta^{-1-\omega_{1}},\quad N_{2}\left(\eta\right)= C\left(\eps\right)\left[\max\left\{\eps^{b},1\right\}+\max\left\{\eps^{b},R^{b}\right\}\right]R^{-\omega_{2}}\eta^{-1-\omega_{2}}
 \end{split}
\end{equation*}
the appropriate estimates for $\tilde{G}_{i}$, $i=1,2$ with $\omega_{i}$ for $i=1,2$ and $\theta_{1}=\omega_{1}$, $\theta_{2}=\min\left\{\omega_{2},b-\omega_{2}\right\}$. Thus the claim follows by setting $\theta:=\min\left\{\omega_{1},\omega_{2},b-\omega_{2}\right\}$ and $\omega:=\min\left\{\omega_{1},\omega_{2}\right\}$ and using also that $\frac{\kappa}{R}\leq \left(\frac{\kappa}{R}\right)^{\mu}$ for $\kappa\leq 1$ and $R\geq 1$ for the second statement.
\end{proof}

\begin{lemma}\label{Lem:subsolution:2}
 For sufficiently large $C_{0}$ the function $\Phi$ satisfies \eqref{eq:sptest4}.
\end{lemma}

\begin{proof}
We have to show that for $X\in\left[0,R\right]$ we have
\begin{equation*}
 \begin{split}
  \del_{s}\Phi\left(X,s\right)+\int_{0}^{\infty}\frac{K_{\eps}^{\lambda}\left(X\ee^{t-s},Z\ee^{t-s}\right)}{Z}\ee^{\frac{s}{\beta}}h\left(Z\ee^{t-s},s\right)\left[\Phi\left(X+Z,s\right)-\Phi\left(X,s\right)\right]\dZ\geq 0.
 \end{split}
\end{equation*}
By construction of $\Phi$ this is equivalent to
\begin{equation*}
 \begin{split}
  &-C_{0}\eps^{-a}\max\left\{\eps^{b},1\right\}\int_{0}^{\infty}\frac{\tilde{v}_{1}\left(Z\right)}{Z}\left[\Phi\left(X+Z\right)-\Phi\left(X\right)\right]\dZ\\
  &-C_{0}\eps^{-a}\left[\max\left\{\eps^{b},1\right\}+\max\left\{\eps^{b},R^{b}\right\}\right]\int_{0}^{\infty}\frac{\tilde{v}_{2}\left(Z\right)}{Z}\left[\Phi\left(X+Z\right)-\Phi\left(X\right)\right]\dZ\\
  &+\int_{0}^{\infty}\frac{K_{\eps}^{\lambda}\left(X\ee^{t-s},Z\ee^{t-s}\right)}{Z}\ee^{\frac{s}{\beta}}h\left(Z\ee^{t-s}\right)\left[\Phi\left(X+Z\right)-\Phi\left(X\right)\right]dZ\geq 0
 \end{split}
\end{equation*}
Estimating $K_{\eps}^{\lambda}$ for $X\in\left[0,R\right]$ one has
\begin{equation*}
 \begin{split}
  &\quad \frac{K_{\eps}^{\lambda}\left(X\ee^{t-s},Z\ee^{t-s}\right)}{Z}\ee^{\frac{s}{\beta}}\leq C\frac{\left(X\ee^{t-s}+\eps\right)^{-a}\left(Z\ee^{t-s}+\eps\right)^{b}+\left(X\ee^{t-s}+\eps\right)^{b}\left(Z\ee^{t-s}+\eps\right)^{-a}}{Z}\\
  &\leq C\eps^{-a}\frac{\left(Z\ee^{t-s}+\eps\right)^{b}}{Z}+C\max\left\{\eps^{b},R^{b}\right\}\frac{\left(Z\ee^{t-s}+\eps\right)^{-a}}{Z}\\
  &\leq C\eps^{-a}\max\left\{1,\eps^{b}\right\}\left(Z^{-1}\chi_{\left[0,1\right]}\left(Z\right)+Z^{b-1}\chi_{\left[1,\infty\right)}\left(Z\right)\right)+C\max\left\{\eps^{b},R^{b}\right\}\eps^{-a}Z^{-1}\\
  &\leq C\eps^{-a}\left[\max\left\{\eps^{b},1\right\}+\max\left\{\eps^{b},R^{b}\right\}\right]Z^{-1}+C\eps^{-a}\max\left\{1,\eps^{b}\right\}Z^{\max\left\{0,b\right\}-1}.
 \end{split}
\end{equation*}
Defining 
\begin{equation*}
w_{1}\left(Z\right):=\frac{h\left(Z\ee^{t-s}\right)}{Z^{1-\max\left\{0,b\right\}}} \quad \text{and} \quad w_{2}\left(Z\right):=\frac{h\left(Z\ee^{t-s}\right)}{Z}
\end{equation*}
and using that $\Phi$ is non-increasing it thus suffices to show
\begin{equation*}
 \begin{split}
  &\quad C_{0}\eps^{-a}\max\left\{\eps^{b},1\right\}\int_{0}^{\infty}\frac{\tilde{v}_{1}\left(Z\right)}{Z}\left[\Phi\left(X\right)-\Phi\left(X+Z\right)\right]\dZ\\
  &+C_{0}\eps^{-a}\left[\max\left\{\eps^{b},1\right\}+\max\left\{\eps^{b},R^{b}\right\}\right]\int_{0}^{\infty}\frac{\tilde{v}_{2}\left(Z\right)}{Z}\left[\Phi\left(X\right)-\Phi\left(X+Z\right)\right]\dZ\\
  &-C\eps^{-a}\max\left\{1,\eps^{b}\right\}\int_{0}^{\infty}\frac{w_{1}\left(Z\right)}{Z}\left[\Phi\left(X\right)-\Phi\left(X+Z\right)\right]\dZ\\
  &-C\eps^{-a}\left[\max\left\{\eps^{b},1\right\}+\max\left\{\eps^{b},R^{b}\right\}\right]\int_{0}^{\infty}\frac{w_{2}\left(Z\right)}{Z}\left[\Phi\left(X\right)-\Phi\left(X+Z\right)\right]\dZ\geq 0.
 \end{split}
\end{equation*}
Defining
\begin{equation*}
 V_{i}\left(Z\right):=\int_{Z}^{\infty}\frac{v_{i}\left(Y\right)}{Y}\dY\quad \text{and} \quad W_{i}\left(Z\right):=\int_{Z}^{\infty}\frac{w_{i}\left(Y\right)}{Y}\dY
\end{equation*}
we can rewrite this as
\begin{equation*}
 \begin{split}
  &\quad\eps^{-a}\max\left\{\eps^{b},1\right\}\int_{0}^{\infty}-\del_{Z}\left(C_{0}V_{1}\left(Z\right)-CW_{1}\left(Z\right)\right)\left[\Phi\left(X\right)-\Phi\left(X+Z\right)\right]\dZ\\
  &+\eps^{-a}\left[\max\left\{\eps^{b},1\right\}+\max\left\{\eps^{b},R^{b}\right\}\right]\int_{0}^{\infty}-\del_{Z}\left(C_{0}V_{2}\left(Z\right)-CW_{2}\left(Z\right)\right)\left[\Phi\left(X\right)-\Phi\left(X+Z\right)\right]\dZ\geq 0.
 \end{split}
\end{equation*}
Integrating by parts this is equivalent to
\begin{equation*}
 \begin{split}
  &\quad\eps^{-a}\max\left\{\eps^{b},1\right\}\int_{0}^{\infty}-\del_{Z}\Phi\left(X+Z\right)\left(C_{0}V_{1}\left(Z\right)-CW_{1}\left(Z\right)\right)\dZ\\
  &+\eps^{-a}\left[\max\left\{\eps^{b},1\right\}+\max\left\{\eps^{b},R^{b}\right\}\right]\int_{0}^{\infty}-\del_{Z}\Phi\left(X+Z\right)\left(C_{0}V_{2}\left(Z\right)-CW_{2}\left(Z\right)\right)\dZ\geq 0.
 \end{split}
\end{equation*}
Using that $\Phi$ is non-increasing it thus suffices to show $C_{0}V_{i}\left(Z\right)-CW_{i}\left(Z\right)\geq 0$ for $i=1,2$. To see this note first that $V_{i}$ is explicitly given by $V_{i}\left(Z\right)=\frac{1}{\omega_{i}}Z^{-\omega_{i}}$ for $i=1,2$. Furthermore using Lemma~\ref{Lem:moment:est:stand} one has
\begin{equation*}
 W_{1}\left(Z\right)\leq CZ^{\max\left\{0,b\right\}-\rho}=CZ^{-\omega_{1}}\quad \text{and} \quad W_{2}\left(Z\right)\leq CZ^{-\omega_{2}}.
\end{equation*}
Thus choosing $C_{0}$ sufficiently large the claim follows.
\end{proof}

We finally prove a technical Lemma that will be needed in the following.

\begin{lemma}\label{Lem:Taylor:estimate}
 Let $\delta\in\left(0,1\right)$ and $\rho\in\left(\max\left\{0,b\right\},1\right)$. Then
 \begin{equation*}
  \left(1-\frac{\kappa}{R}+\xi\right)^{1-\rho}\left(1-\left(\frac{R_{0}}{R}\right)^{\delta}\frac{\ee^{-\delta t}}{\left(1-\frac{\kappa}{R}+\xi\right)^{\delta}}\right)\geq \left(1-\left(\frac{R_{0}}{R}\right)^{\delta}\ee^{-\delta t}\right)-\abs{\xi-\frac{\kappa}{R}}
 \end{equation*}
 holds for every $\xi\in\left[\frac{R_{0}}{R}\ee^{-t}-1+\frac{\kappa}{R},\frac{\kappa}{R}\right]$.
\end{lemma}

\begin{proof}
 Let $\frac{R_{0}}{R}\ee^{-t}-1+\frac{\kappa}{R}\leq \xi\leq \frac{\kappa}{R}$, then $\left(1-\frac{\kappa}{R}+\xi\right)\in\left[\frac{R_{0}}{R}\ee^{-t},1\right]$. Thus for $0\leq \rho<1$ we have $\left(1-\frac{\kappa}{R}+\xi\right)^{1-\rho}\geq \left(1-\abs{\xi-\frac{\kappa}{R}}\right)$ and therefore we can estimate (noting that the second term in brackets is non-negative)
 \begin{equation*}
  \begin{split}
   &\quad \left(1-\frac{\kappa}{R}+\xi\right)^{1-\rho}\left(1-\left(\frac{R_{0}}{R}\right)^{\delta}\frac{\ee^{-\delta t}}{\left(1-\frac{\kappa}{R}+\xi\right)^{\delta}}\right)\geq \left(1-\frac{\kappa}{R}+\xi\right)\left(1-\left(\frac{R_{0}}{R}\right)^{\delta}\frac{\ee^{-\delta t}}{\left(1-\frac{\kappa}{R}+\xi\right)^{\delta}}\right)\\
   &=1-\frac{\kappa}{R}+\xi-\left(\frac{R_{0}}{R}\right)^{\delta}\ee^{-\delta t}\frac{1-\frac{\kappa}{R}+\xi}{\left(1-\frac{\kappa}{R}+\xi\right)^{\delta}}\geq 1-\left(\frac{R_{0}}{R}\right)^{\delta}\ee^{-\delta t}-\abs{\xi-\frac{\kappa}{R}}
  \end{split}
 \end{equation*}
as $\frac{1-\frac{\kappa}{R}+\xi}{\left(1-\frac{\kappa}{R}+\xi\right)^{\delta}}\leq 1$ for $\xi$ as above and $\delta<1$.
\end{proof}

\subsubsection{Invariance of the lower bound}
 We are now prepared to finish the proof of the invariance of the lower bound~\eqref{eq:F2}.
 
 \begin{proposition}
  For sufficiently small $\delta>0$ and sufficiently large $R_{0}$ (maybe also depending on $\delta$),  we have
  \begin{equation*}
   \int_{0}^{R}h\left(x,t\right)\dx\geq R^{1-\rho}\left(1-\left(\frac{R_0}{R}\right)^{\delta}\right)_{+}.
  \end{equation*}
 \end{proposition}

\begin{proof}
 From the special choice of the test function $\psi$ we have according to \eqref{eq:sptest2} that
 \begin{equation*}
  \begin{split}
   \int_{0}^{R}h\left(x,t\right)\dx\geq\int_{0}^{\infty}h\left(x,t\right)\psi\left(x,t\right)\dx\geq\int_{0}^{\infty}h_{0}\left(x\right)\psi\left(x,0\right)\dx=\ee^{-\left(1-\rho\right)t}\int_{0}^{\infty}h_{0}\left(X\ee^{t}\right)\Phi\left(X,0\right)\ee^{t}\dX.
  \end{split}
 \end{equation*}
Defining $H_{0}\left(X\right):=\int_{0}^{X}h_{0}\left(Y\right)\dY$ we obtain 
\begin{equation*}
 \begin{split}
  \int_{0}^{R}h\left(x,t\right)&\geq \ee^{-\left(1-\rho\right)t}\int_{0}^{\infty}h_{0}\left(X\ee^{t}\right)\Phi\left(X,0\right)\ee^{t}\dX=\ee^{-\left(1-\rho\right)t}\int_{0}^{\infty}H_{0}'\left(X\ee^{t}\right)\Phi\left(X,0\right)\dX\\
  &=\left.\ee^{-\left(1-\rho\right)t}H_{0}\left(X\ee^{t}\right)\Phi\left(X,0\right)\right|_{0}^{\infty}+\ee^{-\left(1-\rho\right)t}\int_{0}^{\infty}H_{0}\left(X\ee^{t}\right)\left(-\del_{X}\Phi\left(X,0\right)\right)\dX\\
  &= \ee^{-\left(1-\rho\right)t}\int_{0}^{\infty}H_{0}\left(X\ee^{t}\right)\left(-\del_{X}\Phi\left(X,0\right)\right)\dX
 \end{split}
\end{equation*}
where we integrated by parts and used that the boundary terms are zero as $H\left(0\right)=0$ and the support of $\Phi$ is bounded from the right. Using now that by assumption we have $H_{0}\left(X\ee^{t}\right)\geq \left(X\ee^{t}\right)^{1-\rho}\left(1-\left(\frac{R_{0}}{X\ee^{t}}\right)^{\delta}\right)_{+}$ we obtain
\begin{equation*}
 \begin{split}
  \int_{0}^{R}h\left(x,t\right)\dx&\geq \ee^{-\left(1-\rho\right)t}\int_{0}^{\infty}\left(X\ee^{t}\right)^{1-\rho}\left(1-\left(\frac{R_{0}}{X\ee^{t}}\right)^{\delta}\right)_{+}G\left(X,0\right)\dX\\
  &=\int_{0}^{\infty}X^{1-\rho}\left(1-\left(\frac{R_{0}}{X\ee^{t}}\right)^{\delta}\right)_{+}\frac{1}{R}\tilde{G}\left(\frac{X}{R}-1+\frac{\kappa}{R},0\right)\dX\\
  &=R^{1-\rho}\int_{-D}^{\kappa/R}\left(1-\frac{\kappa}{R}+\xi\right)^{1-\rho}\left(1-\left(\frac{R_{0}}{R}\right)^{\delta}\frac{\ee^{-\delta t}}{\left(1-\frac{\kappa}{R}+\xi\right)^{\delta}}\right)\tilde{G}\left(\xi,0\right)\dd\xi
 \end{split}
\end{equation*}
where we used the change of variables $\xi=\frac{X}{R}-1+\frac{\kappa}{R}$ and defined $D:=1-\frac{R_{0}}{R}\ee^{-t}-\frac{\kappa}{R}$. Note that for $\kappa\leq \left(1-\ee^{-t}\right)$ we have $0\leq D\leq 1$ (because we assume $R\geq 1$). Using Lemma~\ref{Lem:Taylor:estimate} we can estimate the integrand on the right hand side to obtain
\begin{equation*}
 \begin{split}
  \int_{0}^{R}h\left(x,t\right)\dx&\geq R^{1-\rho}\int_{-D}^{\kappa/R}\left(1-\left(\frac{R_{0}}{R}\right)^{\delta}\ee^{-\delta t}\right)\tilde{G}\left(\xi,0\right)\dd\xi-R^{1-\rho}\int_{-D}^{\kappa/R}\abs{\xi-\frac{\kappa}{R}}\tilde{G}\left(\xi,0\right)\dd\xi\\
  &\geq R^{1-\rho}\left(1-\left(\frac{R_{0}}{R}\right)^{\delta}\ee^{-\delta t}\right)\left(\int_{-\infty}^{\kappa/R}\tilde{G}\left(\xi,0\right)\dd\xi-\int_{-\infty}^{-D}\tilde{G}\left(\xi,0\right)\dd\xi\right)\\
  &\quad-R^{1-\rho}\left(\int_{-1}^{0}\abs{\xi-\frac{\kappa}{R}}\tilde{G}\left(\xi,0\right)\dd\xi+\frac{\kappa}{R}\int_{0}^{\kappa/R}\tilde{G}\left(\xi,0\right)\dd\xi\right)
 \end{split}
\end{equation*}
Applying furthermore Remark~\ref{Rem:properties:der} and Lemma~\ref{Lem:subsolution:2:integral:estimate} we get

\begin{equation*}
\begin{split}
&\quad\int_{0}^{R}h\left(x,t\right)\dx\\
  &\geq R^{1-\rho}\left(1-\left(\frac{R_{0}}{R}\right)^{\delta}\ee^{-\delta t}\right)\left(1-C\left(\frac{\kappa}{R D}\right)^{\mu}-\frac{C}{D^{\omega}}R^{-\theta}\right)\\
  &\quad-R^{1-\rho}\left(\int_{-1}^{0}\abs{\xi}\tilde{G}\left(\xi,0\right)\dd\xi+\frac{\kappa}{R}\int_{-1}^{0}\tilde{G}\left(\xi,0\right)\dd\xi+\frac{\kappa}{R}\right)\\
  &\geq R^{1-\rho}\left(1-\left(\frac{R_{0}}{R}\right)^{\delta}\ee^{-\delta t}\right)\left(1-C\left(\frac{\kappa}{R D}\right)^{\mu}-\frac{Ct}{D^{\omega}}R^{-\theta}\right)-R^{1-\rho}\left(C\left(\frac{\kappa}{R}\right)^{\mu}+CtR^{-\theta}+\frac{2\kappa}{R}\right).
 \end{split}
\end{equation*}
Inserting the definition of $D$ and rearranging we thus obtain 
\begin{equation*}
 \begin{split}
  \int_{0}^{R}h\left(x,t\right)\dx&\geq R^{1-\rho}\left(1-\left(\frac{R_{0}}{R}\right)^{\delta}\ee^{-\delta t}\right)-CR^{1-\rho}\frac{\left(1-\left(\frac{R_{0}}{R}\right)^{\delta}\ee^{-\delta t}\right)}{\left(1-\left(\frac{R_{0}}{R}\right)^{\delta}\ee^{-\delta t}\right)^{\mu}}\left(\frac{\kappa}{R}\right)^{\mu}\\
  &\quad-\frac{Ct}{R^{\theta}}R^{1-\rho}\left(\frac{\left(1-\left(\frac{R_{0}}{R}\right)^{\delta}\ee^{-\delta t}\right)}{\left(1-\left(\frac{R_{0}}{R}\right)\ee^{-t}\right)^{\omega}}+1\right)-CR^{1-\rho}\left(\frac{\kappa}{R}+\left(\frac{\kappa}{R}\right)^{\mu}\right).
 \end{split}
\end{equation*}
Using that for $\delta,\omega\in\left(0,1\right)$ and $R>R_{0}$ we have $\left(1-
\left(\frac{R_{0}}{R}\right)^{\delta}\ee^{-\delta t}\right)\leq 1-\frac{R_{0}}{R}\ee^{-t}\leq\left(1-\frac{R_{0}}{R}\ee^{-t}\right)^{\omega}$ 
and thus $\left(\frac{1-\left(\frac{R_{0}}{R}\right)^{\delta}\ee^{-\delta t}}{\left(1-\frac{R_{0}}{R}\ee^{-t}\right)^{\omega}}+1\right)\leq 2$ and $\frac{\left(1-\left(\frac{R_{0}}{R}\right)^{\delta}\ee^{-\delta t}\right)}{\left(1-\left(\frac{R_{0}}{R}\right)^{\delta}\ee^{-\delta t}\right)^{\mu}}\leq 1$, we therefore get
\begin{equation*}
 \begin{split}
  \int_{0}^{R}h\left(x,t\right)\dx\geq R^{1-\rho}\left(1-\left(\frac{R_{0}}{R}\right)^{\delta}\ee^{-\delta t}\right)-\frac{Ct}{R^{\theta}}R^{1-\rho}-CR^{1-\rho}\left(\frac{\kappa}{R}+\left(\frac{\kappa}{R}\right)^{\mu}\right).  
 \end{split}
\end{equation*}
As for $\delta t\leq 1$ (note that we assume $\delta, t\leq 1$) we can estimate $1-\left(\frac{R_{0}}{R}\right)^{\delta}\ee^{-\delta t}\geq 1-\left(\frac{R_{0}}{R}\right)^{\delta}+\frac{1}{\ee}\left(\frac{R_{0}}{R}\right)^{\delta}\delta t$, we obtain
\begin{equation*}
 \begin{split}
  \int_{0}^{R}h\left(x,t\right)\dx&\geq R^{1-\rho}\left(1-\left(\frac{R_{0}}{R}\right)^{\delta}\right)+R^{1-\rho}\left(\frac{R_{0}}{R}\right)^{\delta}\frac{\delta t}{\ee}-\frac{Ct}{R^{\theta}}R^{1-\rho}-CR^{1-\rho}\left(\frac{\kappa}{R}+\left(\frac{\kappa}{R}\right)^{\mu}\right).
 \end{split}
\end{equation*}
We now choose $\mu=\theta$ and $\kappa<1$ sufficiently small. Then as we assume $R\geq R_{0}\geq 1$ we have $\frac{\kappa}{R}\leq \left(\frac{\kappa}{R}\right)^{\mu}=\left(\frac{\kappa}{R}\right)^{\theta}$. Using this we can further estimate
\begin{equation*}
 \begin{split}
   \int_{0}^{R}h\left(x,t\right)\dx&\geq R^{1-\rho}\left(1-\left(\frac{R_{0}}{R}\right)^{\delta}\right)+R^{1-\rho}\left(\frac{R_{0}}{R}\right)^{\delta}\frac{\delta t}{\ee}-\frac{Ct}{R^{\theta}}R^{1-\rho}-\frac{C\kappa^{\theta}}{R^{\theta}}R^{1-\rho}\\
   &\geq R^{1-\rho}\left(1-\left(\frac{R_{0}}{R}\right)^{\delta}\right)+R^{1-\rho}\left(\left(\frac{R_{0}}{R}\right)^{\delta}\frac{\delta t}{\ee}-\frac{C}{R^{\theta}}\left(t+\kappa^{\theta}\right)\right).
 \end{split}
\end{equation*}
 Thus it suffices to show $\left(\frac{R_{0}}{R}\right)^{\delta}\frac{\delta t}{\ee}-\frac{C}{R^{\theta}}\left(t+\kappa^{\theta}\right)\geq 0$ while this is equivalent to $R^{\theta-\delta}\geq C\ee\frac{t+\kappa^{\theta}}{R_{0}^{\delta}\delta t}$, but this is true at least if we choose $\kappa$ sufficiently small such that also $\kappa^{\theta}<t$, $0<\delta<\theta$ and then $R_{0}$ sufficiently large (note that we have only to prove this for $R\geq R_{0}$).
\end{proof}

\subsection{Existence of self-similar solutions}\label{S.ex}

\begin{proposition}\label{Prop:existence:stat}
 Let $K$ satisfy Assumptions~\eqref{Ass1a}-\eqref{Ass1}. Then for any $\rho\in\left(\max\left\{0,b\right\},1\right)$ there exists a weak stationary solution $h_{\eps}$ to \eqref{A2} (with $K$ replaced by $K_{\eps}$).
\end{proposition}

This proposition is proved in the following way: according to Theorem~\ref{T.fixedpoint} we obtain $h_{\eps}^{\lambda}\in\Y$ that is stationary under the action of $S_{\eps}^{\lambda}\left(t\right)$, i.e. $H_{\eps}^{\lambda}$ is a stationary mild solution of \eqref{eq:S1E9}. Then taking a subsequence of $h_{\eps}^{\lambda}$ converging weakly to some $h_{\eps}$ and passing to the limit $\lambda\to 0$ in the equation shows the claim. The last step is quite similar to passing to the limit $\eps\to 0$ in Section~\ref{sec:limit:eps} and thus we do not give details here.

We furthermore have the following result about the regularity and asymptotic behaviour of $h_{\eps}$ completing the proof of Proposition~\ref{P.hepsexistence}:

\begin{proposition}\label{Prop:reg:decay:stat:eps}
 The solution $h_{\eps}\in\mathcal{Y}$ from Proposition~\ref{Prop:existence:stat} is continuous on $\left(0,\infty\right)$ and satisfies $h_{\eps}\left(x\right)\sim \left(1-\rho\right)x^{-\rho}$ as $x\to\infty$. 
\end{proposition}

This can be proved in a similar way as the corresponding result in \cite[Lemma~4.2 \& Lemma~4.3]{NV12a}.

\section{Passing to the limit $\eps \to 0$}\label{sec:eps:to:zero}

\subsection{Strategy of the proof}

Our starting point is that we have a continuous positive function  $h_{\eps}$ that is a weak solution to 
\begin{equation}\label{eq:self:sim:eps}
 \partial_x I_{\eps}[h_{\eps}] = \partial_x \left( x h_{\eps}\right) + (\rho-1) h_{\eps}\,, \qquad \mbox{ with } \quad I_{\eps}[h_{\eps}] = \int_0^x \int_{x-y}^{\infty} \frac{K_{\eps}(y,z)}{z} h_{\eps}(y) h_{\eps}(z)\,dz\,dy\,.
\end{equation}
Furthermore we have the estimates 
\begin{equation}\label{hepsestimates}
 \int_0^rh_{\eps}(x)\,dx \leq r^{1-\rho} \qquad \mbox{ and } \qquad \lim_{r \to \infty} \int_0^r h_{\eps}(x)\,dx / r^{1-\rho} =1\,.
 \end{equation}
We introduce the following quantities:
\begin{equation}\label{lepsdef}
 \mu_{\eps}=\int_0^1 h_{\eps}(x) \left(x+\eps\right)^{-a} \,dx \, ,\qquad \lambda_{\eps} =\int_0^1 h_{\eps}(x) \left(x+\eps\right)^b\,dx \,, \qquad L_{\eps}:=\max\left( \lambda_{\eps}^{\frac{1}{1+a}},\mu_{\eps}^{\frac{1}{1-b}}\right)\,.
\end{equation}
Up to passing to a subsequence we can in the following assume that either $L_{\eps}$ converges or $L_{\eps}\to\infty$ for $\eps\to 0$. Furthermore as the case $L_{\eps}\to 0$ behaves slightly different, we use from now on the following notation: we define $L:=L_{\eps}$ if $L_{\eps}\not\to 0$ and $L:=1$ if $L_{\eps}\to 0$ and thus (up to passing maybe to another subsequence) we may assume $L>0$.
For the following let 
\[
 X=\frac{x}{L} \,, \qquad h_{\eps}(x) = H_{\eps}(X) L^{-\rho}\,.
\]

The strategy of the proof is the following: First we derive a uniform lower integral bound for $H_{\eps}$. This will be done by constructing a special test function that provides us with some estimate from below that is sufficient to conclude on a lower bound by some iteration argument. In order to obtain the same uniform lower bound for $h_{\eps}$ we need to exclude the case $L_{\eps}\to \infty$. For this reason we will show that in this case $H_{\eps}$ converges to some limit $H$ solving some differential equation that has no solution satisfying the growth condition $\int_{0}^{R}H\dX\leq R^{1-\rho}$. Note that at this point it is crucial to assume $\rho>0$. Using the lower bound on $h_{\eps}$ we can show some exponential decay of $h_{\eps}$ near zero which is then enough to pass to the limit $\eps\to 0$.

For $0<\kappa<1$ we again denote in the following by $\varphi_{\kappa}$ a non-negative, symmetric standard mollifier with $\supp \varphi_{\kappa}\subset \left[-\kappa,\kappa\right]$.

\subsection{Uniform lower bound for $H_{\eps}$}

In this subsection we will show a uniform lower bound on $H_{\eps}$, i.e. we will prove:
\begin{proposition}\label{Prop:Hepslowerbound}
 For any $\delta>0$ there exists $R_{\delta}>0$ such that 
 \begin{equation}\label{Hepslowerbound}
  \int_0^R H_{\eps}(X)\dX \geq (1-\delta) R^{1-\rho} \qquad \mbox{ for all } R \geq R_{\delta}.
 \end{equation}
\end{proposition}

\subsubsection{Construction of a suitable test function}\label{subsec:test}

We start by constructing a special test function and therefore notice that for $\psi=\psi\left(x,t\right)$ with $\psi\in C^{1}$ and compact support in $\left[0,T\right]\times \left[0,\infty\right)$ we obtain from the equation on $h_{\eps}$:
\begin{equation*}
 \begin{split} 
 0&=\int_{0}^{T}\int_{0}^{\infty}\del_{x}\psi I_{\eps}\left[h_{\eps}\right]\dx\dt-\int_{0}^{T}\int_{0}^{\infty}x\del_{x}\psi h_{\eps}\dx\dt+\left(\rho-1\right)\int_{0}^{T}\int_{0}^{\infty}\psi h_{\eps}\dx\dt\\
 &=\int_{0}^{T}\int_{0}^{\infty}\del_{t}\psi h_{\eps}\dx\dt+\int_{0}^{\infty}\psi\left(\cdot,0\right)h_{\eps}\dx-\int_{0}^{\infty}\psi\left(\cdot,T\right)h_{\eps}\dx.
 \end{split}
\end{equation*}
Choosing $\psi$ such that 
\begin{equation}\label{eq:special:test:1}
 \begin{split}
  \int_{0}^{T}\int_{0}^{\infty}\del_{x}\psi I_{\eps}\left[h_{\eps}\right]\dx\dt-\int_{0}^{T}\int_{0}^{\infty}x\del_{x}\psi h_{\eps}\dx\dt+\left(\rho-1\right)\int_{0}^{T}\int_{0}^{\infty}\psi h_{\eps}\dx\dt-\int_{0}^{T}\int_{0}^{\infty}\del_{t}\psi h_{\eps}\dx\dt\geq 0
 \end{split}
\end{equation}
we obtain
\begin{equation*}
 \begin{split}
  \int_{0}^{\infty}\psi\left(\cdot,0\right)h_{\eps}\dx\geq \int_{0}^{\infty}\psi\left(\cdot,T\right)h_{\eps}\dx.  
 \end{split}
\end{equation*}
Rewriting \eqref{eq:special:test:1} we obtain
\begin{equation}\label{eq:special:test:weak:psi}
 \begin{split}
  \int_{0}^{T}\int_{0}^{\infty}h_{\eps}\left(x\right)\left\{\int_{0}^{\infty}\frac{K_{\eps}\left(x,y\right)}{y}h_{\eps}\left(y\right)\left[\psi\left(x+y\right)-\psi\left(x\right)\right]\dy-x\del_{x}\psi\left(x\right)+\left(\rho-1\right)\psi\left(x\right)-\del_{t}\psi\left(x\right)\right\}\dx\dt\geq 0.
 \end{split}
\end{equation}
Defining $W$ by $\psi\left(x,t\right)=\ee^{-\left(1-\rho\right)t}W\left(\xi,t\right)$ with $\xi=\frac{x}{L\ee^{t}}$ we can rewrite the term in brackets and obtain that it suffices to construct $W$ such that
\begin{equation}\label{eq:special:test:2}
 \begin{split}
  \del_{t}W\left(\frac{x}{L\ee^{t}},t\right)\leq \int_{0}^{\infty}\frac{K_{\eps}\left(x,y\right)}{y}h_{\eps}\left(y\right)\left[W\left(\frac{x+y}{L\ee^{t}},t\right)-W\left(\frac{x}{L\ee^{t}},t\right)\right]\dy.  
 \end{split}
\end{equation}
For further use we also note that we only need this in weak form, i.e. we need that 
\begin{equation}\label{eq:special:test:weak}
 \begin{split}
  \int_{0}^{T}\int_{0}^{\infty}\ee^{-\left(1-\rho\right)t}h_{\eps}\left(x\right)\left\{\del_{t}W\left(\frac{x}{L\ee^{t}},t\right)-\int_{0}^{\infty}\frac{K_{\eps}\left(x,y\right)}{y}h_{\eps}\left(y\right)\left[W\left(\frac{x+y}{L\ee^{t}},t\right)-W\left(\frac{x}{L\ee^{t}},t\right)\right]\dy\right\}\dx\dt\leq 0,
 \end{split}
\end{equation}
provided that we can justify the change from $\psi$ to $W$.

We furthermore list here some parameters that are frequently used in the following. For given $\nu\in\left(0,1\right)$ that will be fixed later we define 
\begin{equation*}
 \begin{split}
  \beta:=\begin{cases}
          b & b\geq  0\\
          \nu b & b <0
         \end{cases}, \qquad 
  \omega_{1}:=\min\left\{\rho-b,\rho\right\}, \qquad
  \omega_{2}:=\rho,\qquad \tilde{b}:=\max\left\{0,b\right\}.
 \end{split}
\end{equation*}
The idea to construct the test function $W$ is similar to the approach in Section~\ref{Sec:constr:dual:prop}, i.e. we replace the integral kernel $K_{\eps}$ and $h_{\eps}$ by corresponding power laws. Due to the singular behaviour of $K_{\eps}$ (for $\eps\to 0$) the resulting integral is not defined near the origin and thus we have to consider the region near the origin separately. We have the following existence result.
\begin{lemma}\label{Lem:ex:special:test:1}
 For any constant $\tilde{C}>0$ there exists a function $\tilde{W}\in C^{1}\left(\left[0,T\right],C^{\infty}\left(\R\right)\right)$ solving
 \begin{equation*}
  \begin{split}
   &\quad\del_{t}\tilde{W}\left(\xi,t\right)-C\int_{0}^{1}\frac{h_{\eps}\left(z\right)}{z}\left(L^{b}A^{\beta}\left(z+\eps\right)^{-a}+L^{-a}A^{-\nu a}\left(z+\eps\right)^{b}\right)\left[\tilde{W}\left(\xi+\frac{z}{L}\right)-\tilde{W}\left(\xi\right)\right]\dz\\
   &-\frac{C}{L^{\rho+a-\max\left\{0,b\right\}}}\int_{0}^{\infty}\frac{A^{-\nu a}}{\eta^{1+\omega_{1}}}\left[\tilde{W}\left(\xi+\eta\right)-\tilde{W}\left(\xi\right)\right]\dd\eta-\frac{C}{L^{\rho-b}}\int_{0}^{\infty}\frac{A^{\beta}}{\eta^{1+\omega_{2}}}\left[\tilde{W}\left(\xi+\eta\right)-\tilde{W}\left(\xi\right)\right]\dd\eta=0
  \end{split}
 \end{equation*}
 with $\tilde{W}\left(\cdot,0\right)=\chi_{\left(-\infty,A-\kappa\right]}\ast^{3}\varphi_{\kappa/3}\left(\cdot\right)$.
\end{lemma}

\begin{proof}
 This is shown in Proposition~\ref{Prop:ex:dual:sum}.
\end{proof}

\begin{remark}
 As shown in the appendix $\tilde{W}$ is non-increasing, has support in $\left(-\infty,A\right]$, is non-negative and bounded by $1$.
\end{remark}

As $K_{\eps}$ might get quite singular at the origin for $\eps\to 0$, we define now $W$ as the function $W\left(\xi,t\right):=\tilde{W}\left(\xi,t\right)\chi_{\left[A^{\nu},\infty\right)}\left(\xi\right)$, i.e. we cut $\tilde{W}$ at $\xi=A^{\nu}$ in order to avoid integrating near the origin. Obviously $W$ is not in $C^{1}$ and thus the corresponding $\psi$ is also not differentiable. But as already mentioned it is enough to show that \eqref{eq:special:test:weak} holds, provided we can justify the change from $\psi$ to $W$ (and reverse). This will be done next, i.e. we will first show that \eqref{eq:special:test:2} holds for all $\xi\neq A^{\nu}$. Then by convolution in $\xi$ with $\varphi_{\delta}$ it is possible to change from $\psi$ to $W$ (and reverse). Finally taking the limit $\delta\to 0$ this then shows that \eqref{eq:special:test:weak} holds.

\begin{lemma}\label{Lem:special:test:ineq}
 For sufficiently large $\tilde{C}$, inequality \eqref{eq:special:test:2} holds pointwise for all $\xi\neq A^{\nu}$.
\end{lemma}

\begin{proof}
 From the non-negativity of $W$ the claim follows immediately for $\xi<A^{\nu}$ (where $W$ is identically zero). Thus it suffices to consider $\xi>A^{\nu}$. Using furthermore that $\supp W\subset \left(-\infty,A\right]$ it suffices to consider $\xi\in\left(A^{\nu},A\right]$. As $W$ is non-increasing on $\left(A^{\nu},A\right]$ we can estimate $-\left[W\left(\xi+\frac{y}{L\ee^{t}}\right)-W\left(\xi\right)\right]\leq-\left[W\left(\xi+\frac{y}{L}\right)-W\left(\xi\right)\right]$. On the other hand using the estimates on the kernel $K$ we obtain
 \begin{equation}\label{eq:special:test:estimate:1}
  \begin{split}
   &\quad-\int_{0}^{\infty}\frac{K_{\eps}\left(L\ee^{t}\xi,y\right)}{y}h_{\eps}\left(y\right)\left[W\left(\xi+\frac{y}{L\ee^{t}}\right)-W\left(\xi\right)\right]\dy\\
   &\leq -C_{2}\int_{0}^{\infty}\frac{\left(L\ee^{t}\xi+\eps\right)^{-a}\left(y+\eps\right)^{b}+\left(L\ee^{t}\xi+\eps\right)^{b}\left(y+\eps\right)^{-a}}{y}h_{\eps}\left(y\right)\left[W\left(\xi+\frac{y}{L}\right)-W\left(\xi\right)\right]\dy\\
   &\leq -C\int_{0}^{\infty}\frac{L^{-a}A^{-\nu a}\left(y+\eps\right)^{b}+L^{b}A^{\beta}\left(y+\eps\right)^{-a}}{y}h_{\eps}\left(y\right)\left[W\left(\xi+\frac{y}{L}\right)-W\left(\xi\right)\right]\dy\\
   &\leq -C\int_{0}^{1}\frac{h_{\eps}\left(y\right)}{y}\left[L^{b}A^{\beta}\left(y+\eps\right)^{-a}+L^{-a}A^{-\nu a}\left(y+\eps\right)^{b}\right]\left[W\left(\xi+\frac{y}{L}\right)-W\left(\xi\right)\right]\dy\\
   &\quad -\frac{C}{L^{\rho}}\int_{1/L}^{\infty}\frac{H_{\eps}\left(\eta\right)}{\eta}\left[L^{b}A^{\beta}+L^{-a}A^{-\nu a}\left(L\eta +\eps\right)^{b}\right]\left[W\left(\xi+\eta\right)-W\left(\xi\right)\right]\dd\eta\\
   &\leq -C\int_{0}^{1}\frac{h_{\eps}\left(y\right)}{y}\left[L^{b}A^{\beta}\left(y+\eps\right)^{-a}+L^{-a}A^{-\nu a}\left(y+\eps\right)^{b}\right]\left[W\left(\xi+\frac{y}{L}\right)-W\left(\xi\right)\right]\dy\\
   &\quad -\frac{C A^{\beta}}{L^{\rho-b}}\int_{0}^{\infty}\frac{H_{\eps}\left(\eta\right)}{\eta}\left[W\left(\xi+\eta\right)-W\left(\xi\right)\right]\dd\eta-\frac{C A^{-\nu a}}{L^{\rho+a-\max\left\{0,b\right\}}}\int_{0}^{\infty}\frac{H_{\eps}\left(\eta\right)}{\eta^{1-\max\left\{0,b\right\}}}\left[W\left(\xi+\eta\right)-W\left(\xi\right)\right]\dd\eta.
  \end{split}
 \end{equation}
 As $\xi>A^{\nu}$ we have by construction
 \begin{equation}\label{eq:special:test:estimate:2}
  \begin{split}
   &\del_{t}W\left(\xi,t\right)=\del_{t}\tilde{W}\left(\xi,t\right)=\tilde{C}\int_{0}^{1}\frac{h_{\eps}\left(z\right)}{z}\left[L^{b}A^{\beta}\left(z+\eps\right)^{-a}+L^{-a}A^{-\nu a}\left(z+\eps\right)^{b}\right]\left[\tilde{W}\left(\xi+\frac{z}{L}\right)-\tilde{W}\left(\xi\right)\right]\dz\\
   &\quad+\tilde{C}\frac{A^{\beta}}{L^{\rho-b}}\int_{0}^{\infty}\frac{1}{\eta^{1+\omega_{2}}}\left[\tilde{W}\left(\xi+\eta\right)-\tilde{W}\left(\xi\right)\right]\dd\eta+\tilde{C}\frac{A^{-\nu a}}{L^{\rho+a-\max\left\{0,b\right\}}}\int_{0}^{\infty}\frac{1}{\eta^{1+\omega_{1}}}\left[\tilde{W}\left(\xi+\eta\right)-\tilde{W}\left(\xi\right)\right]\dd\eta.
  \end{split}
 \end{equation}
 Thus in order to show \eqref{eq:special:test:2}, i.e.
 \begin{equation*}
  \begin{split}
   \del_{t}W\left(\xi,t\right)-\int_{0}^{\infty}\frac{K_{\eps}\left(L\ee^{t}\xi,y\right)}{y}h_{\eps}\left(y\right)\left[W\left(\xi+\frac{y}{L\ee^{t}},t\right)-W\left(\xi,t\right)\right]\dy\leq 0
  \end{split}
 \end{equation*}
 it is sufficient to compare the expressions in \eqref{eq:special:test:estimate:1} and \eqref{eq:special:test:estimate:2} term by term. Proceeding in the same way as in Lemma~\ref{Lem:subsolution:2} the claim follows, noting that due to the rescaling the estimates from Lemma~\ref{Lem:moment:est:stand} also hold for $H_{\eps}$.
\end{proof}

\begin{lemma}
 The change of variables from $\psi$ to $W$ (and reverse) is justified and inequality~\eqref{eq:special:test:weak} holds.
\end{lemma}

\begin{proof}
 We define $W_{\delta}:=W\left(\cdot,t\right)\ast \varphi_{\delta}\left(\cdot\right)$. Then $W_{\delta}\left(\cdot,t\right)$ is smooth for all $t$ and the change of variables from the corresponding $\psi_{\delta}$ to $W_{\delta}$ (and reverse) is justified. We next show that we can pass to the limit $\delta\to 0$ in the left hand side of \eqref{eq:special:test:weak}. Then from Lemma~\ref{Lem:special:test:ineq} the claim follows.

 As $\del_{t}W_{\delta}\left(\frac{x}{L\ee^{t}},t\right)$ is compactly supported and uniformly bounded (in $\delta$) it suffices to consider
 \begin{equation}\label{eq:special:test:weak:2}
  \begin{split}
   \int_{0}^{\infty}h_{\eps}\left(x\right)\int_{0}^{\infty}\frac{K_{\eps}\left(x,y\right)}{y}h_{\eps}\left(y\right)\left[W_{\delta}\left(\frac{x+y}{L\ee^{t}},t\right)-W_{\delta}\left(\frac{x}{L\ee^{t}},t\right)\right]\dy\dx.
  \end{split}
 \end{equation}

 Changing variables, interchanging the order of integration and splitting the integral we have to consider
 \begin{equation*}
  \begin{split}
   &\quad\int_{0}^{\infty}\int_{0}^{\infty}\frac{1}{L\ee^{t}}\frac{K_{\eps}\left(L\ee^{t}\xi,y\right)}{y}h_{\eps}\left(L\ee^{t}\xi\right)h_{\eps}\left(y\right)\left[W_{\delta}\left(\xi+\frac{y}{L\ee^{t}}\right)-W_{\delta}\left(\xi\right)\right]\dd\xi\dy\\
   &=\int_{0}^{\infty}\int_{0}^{A^{\nu}-\esp}\left(\cdots\right)\dd\xi\dy+\int_{0}^{\infty}\int_{A^{\nu}+\esp}^{\infty}\left(\cdots\right)\dd\xi\dy+\int_{0}^{\infty}\int_{A^{\nu}-\esp}^{A^{\nu}+\esp}\left(\cdots\right)\dd\xi\dy\\
   &=:\left(I\right)+\left(II\right)+\left(III\right),
  \end{split}
 \end{equation*}
 where $2\delta<\esp$ is a fixed and sufficiently small constant. As $W$ and $W_{\delta}$ are compactly supported and bounded (uniformly in $\delta$) it is straightforward to pass to the limit $\delta\to 0$ in the $\xi$-integrals (for fixed $y>0$). It thus remains to show that it is also possible to pass to the limit $\delta\to 0$ in the $y$-integral while this will be done by using Lebesgue's Theorem. We therefore estimate the three integrands separately. First we have using that $W$ is non-negative, bounded and has support in $\left[A^{\nu}-\delta,A+\delta\right]$ as well as the estimate for $K$ and Lemma~\ref{Lem:moment:est:stand} that
 \begin{equation*}
  \begin{split}
   &\quad\abs{\int_{0}^{A^{\nu}-\esp}\frac{1}{L\ee^{t}}\frac{K_{\eps}\left(L\ee^{t}\xi,y\right)}{y}h_{\eps}\left(L\ee^{t}\xi\right)h_{\eps}\left(y\right)\left[W_{\delta}\left(\xi+\frac{y}{L\ee^{t}}\right)-W_{\delta}\left(\xi\right)\right]\dd\xi}\\
   &\leq \frac{C}{L\ee^{t}}\chi_{\left[L\ee^{t}\left(\esp-\delta\right),\infty\right)}\left(y\right)\int_{0}^{A^{\nu}-\esp}\frac{\left(L\ee^{t}\xi+\eps\right)^{-a}\left(y+\eps\right)^{b}}{y}h_{\eps}\left(L\ee^{t}\xi\right)h_{\eps}\left(y\right)W_{\delta}\left(\xi+\frac{y}{L\ee^{t}},t\right)\dd\xi\\
   &\quad +\frac{C}{L\ee^{t}}\chi_{\left[L\ee^{t}\left(\esp-\delta\right),\infty\right)}\left(y\right)\int_{0}^{A^{\nu}-\esp}\frac{\left(L\ee^{t}\xi+\eps\right)^{b}\left(y+\eps\right)^{-a}}{y}h_{\eps}\left(L\ee^{t}\xi\right)h_{\eps}\left(y\right)W_{\delta}\left(\xi+\frac{y}{L\ee^{t}},t\right)\dd\xi\\
   &\leq C\left(\eps, L, A\right)\frac{h_{\eps}\left(y\right)}{y}\chi_{\left[L\ee^{t}\frac{\esp}{2},\infty\right)}\left(y\right)\int_{0}^{A^{\nu}-\esp}h_{\eps}\left(L\ee^{t}\xi\right)\dd\xi\leq C\left(\eps, L, A\right) \left(L\ee^{t}\left(A^{\nu}-\esp\right)\right)^{1-\rho}\frac{h_{\eps}\left(y\right)}{y}\chi_{\left[L\ee^{t}\frac{\esp}{2},\infty\right)}\left(y\right).
  \end{split}
 \end{equation*} 
As the right hand side is independent of $\delta$ and integrable due to Lemma~\ref{Lem:moment:est:stand} we can pass to the limit $\delta\to 0$ in $(I)$.

To estimate the integrand in $(II)$ note that we can bound $h_{\eps}\left(L\ee^{t}\xi\right)$ for $\xi\in\left[A^{\nu},A+\delta\right]$ uniformly in $t$ for $t\in\left[0,T\right]$ as $h_{\eps}$ is continuous. Furthermore as $W$ is differentiable in $\xi$ on $\left[A^{\nu}+\esp,\infty\right)$ with bounded derivative (for $\esp>0$ fixed and depending on $\kappa$) we can bound the $L^{\infty}$-norm of $W_{\delta}$ and $\del_{\xi}W_{\delta}$ by the corresponding expression of $W$. Thus we obtain 
 \begin{equation*}
  \begin{split}
   &\quad\abs{\int_{A^{\nu}+\esp}^{A+\delta}\frac{1}{L\ee^{t}}\frac{K_{\eps}\left(L\ee^{t}\xi,y\right)}{y}h_{\eps}\left(L\ee^{t}\xi\right)h_{\eps}\left(y\right)\left[W_{\delta}\left(\xi+\frac{y}{L\ee^{t}}\right)-W_{\delta}\left(\xi\right)\right]\dd\xi}\\
   &\leq \frac{C}{L\ee^{t}}\int_{A^{\nu}+\esp}^{A+\delta}\frac{\left(L\ee^{t}\xi+\eps\right)^{-a}\left(y+\eps\right)^{b}+\left(L\ee^{t}\xi+\eps\right)^{b}\left(y+\eps\right)^{-a}}{y}h_{\eps}\left(L\ee^{t}\xi\right)h_{\eps}\left(y\right)\abs{W_{\delta}\left(\xi+\frac{y}{L\ee^{t}}\right)-W_{\delta}\left(\xi\right)}\dd\xi\\
   &\leq C\left(L,A,\esp,\eps,\kappa \right)\left(\norm{W}_{L^{\infty}}+\norm{\del_{\xi}W}_{L^{\infty}\left(\left[A^{\nu}+\esp,\infty\right)\right)}\right)\frac{\left(y+\eps\right)^{b}+\left(y+\eps\right)^{-a}}{y}h_{\eps}\left(y\right)\min\left\{\frac{y}{L\ee^{t}},A-A^{\nu}+1-\esp\right\}.
  \end{split}
 \end{equation*}
 Again the right hand side is independent of $\delta$ and integrable due to Lemma~\ref{Lem:moment:est:stand}. Thus we also can pass to the limit $\delta\to 0$ in the $y$-integral in $(II)$. Before estimating the integrand of $(III)$ we first derive an estimate for the expression $\int_{A^{\nu}-\esp}^{A^{\nu}+\esp}\abs{W_{\delta}\left(\xi+\frac{y}{L\ee^{t}}\right)-W_{\delta}\left(\xi\right)}\dd\xi$, i.e. for $y\in \left[0,L\ee^{t}\esp\right]$ we have 
 \begin{equation*}
  \begin{split}
   &\quad\int_{A^{\nu}-\esp}^{A^{\nu}+\esp}\abs{W_{\delta}\left(\xi+\frac{y}{L\ee^{t}}\right)-W_{\delta}\left(\xi\right)}\dd\xi\leq \int_{-\delta}^{\delta}\varphi_{\delta}\left(\eta\right)\int_{A^{\nu}-\esp}^{A^{\nu}+\esp}\abs{W\left(\xi-\eta+\frac{y}{L\ee^{t}}\right)-W\left(\xi-\eta\right)} \dd\xi\dd\eta.
  \end{split}
 \end{equation*}
We thus consider
\begin{equation*}
 \begin{split}
  &\quad \int_{A^{\nu}-\esp}^{A^{\nu}+\esp}\abs{W\left(\xi-\eta+\frac{y}{L\ee^{t}}\right)-W\left(\xi-\eta\right)} \dd\xi\\
  &= \int_{A^{\nu}-\esp}^{A^{\nu}+\eta}W\left(\xi-\eta+\frac{y}{L\ee^{t}}\right)\dd\xi+\int_{A^{\nu}+\eta}^{A^{\nu}+\esp}W\left(\xi-\eta\right)-W\left(\xi-\eta+\frac{y}{L\ee^{t}}\right)\dd\xi\\
  &= \int_{A^{\nu}-\esp+\frac{y}{L\ee^{t}}-\eta}^{A^{\nu}+\frac{y}{L\ee^{t}}}W\left(\xi\right)\dd\xi+\int_{A^{\nu}}^{A^{\nu}+\esp-\eta}W\left(\xi\right)\dd\xi-\int_{A^{\nu}+\frac{y}{L\ee^{t}}}^{A^{\nu}+\esp-\eta}W\left(\xi\right)\dd\xi-\int_{A^{\nu}+\esp-\eta}^{A^{\nu}+\esp-\eta+\frac{y}{L\ee^{t}}}W\left(\xi\right)\dd\xi\\
  &\leq 2\int_{A^{\nu}}^{A^{\nu}+\frac{y}{L\ee^{t}}}W\left(\xi\right)\dd\xi-\int_{A^{\nu}+\esp-\eta}^{A^{\nu}+\esp-\eta+\frac{y}{L\ee^{t}}}W\left(\xi\right)\dd\xi\leq 3\norm{W}_{L^{\infty}}\frac{y}{L\ee^{t}}.
 \end{split}
\end{equation*}
This shows $\int_{A^{\nu}-\esp}^{A^{\nu}+\esp}\abs{W_{\delta}\left(\xi+\frac{y}{L\ee^{t}}\right)-W_{\delta}\left(\xi\right)}\dd\xi\leq 3\norm{W}_{L^{\infty}}\frac{y}{L\ee^{t}}$, i.e. for $y\in \left[0,L\ee^{t}\esp\right]$.

 On the other hand for $y>L\ee^{t}\esp$ we have the trivial estimate 
 \begin{equation*}
  \int_{A^{\nu}-\esp}^{A^{\nu}+\esp}\abs{W_{\delta}\left(\xi+\frac{y}{L\ee^{t}}\right)-W_{\delta}\left(\xi\right)}\dd\xi\leq 2\norm{W_{\delta}}_{L^{\infty}}\esp\leq 2\norm{W}_{L^{\infty}}\esp
 \end{equation*}
 and thus altogether
 \begin{equation*}
  \int_{A^{\nu}-\esp}^{A^{\nu}+\esp}\abs{W_{\delta}\left(\xi+\frac{y}{L\ee^{t}}\right)-W_{\delta}\left(\xi\right)}\dd\xi\leq C \min\left\{\frac{y}{L\ee^{t}},\esp\right\}.
 \end{equation*}
 Using this and again that $h_{\eps}$ is continuous we can also estimate the integrand in $(III)$, i.e.
 \begin{equation*}
  \begin{split}
   &\quad \int_{A^{\nu}-\esp}^{A^{\nu}+\esp}\frac{1}{L\ee^{t}}\frac{K_{\eps}\left(L\ee^{t}\xi,y\right)}{y}h_{\eps}\left(L\ee^{t}\xi\right)h_{\eps}\left(y\right)\left[W_{\delta}\left(\xi+\frac{y}{L\ee^{t}}\right)-W_{\delta}\left(\xi\right)\right]\dd\xi\\
   &\leq C\left(L\right)\int_{A^{\nu}-\esp}^{A^{\nu}+\esp}\frac{\left(L\ee^{t}\xi+\eps\right)^{-a}\left(y+\eps\right)^{b}+\left(L\ee^{t}\xi+\eps\right)^{b}\left(y+\eps\right)^{-a}}{y}h_{\eps}\left(L\ee^{t}\xi\right)h_{\eps}\left(y\right)\abs{W_{\delta}\left(\xi+\frac{y}{L\ee^{t}}\right)-W_{\delta}\left(\xi\right)}\dd\xi\\
   &\leq C\left(L,A,\eps\right) \sup_{\xi\in \left[A^{\nu}-\esp,A^{\nu}+\esp\right]}\abs{h_{\eps}\left(L\ee^{t}\xi\right)}\frac{\left(y+\eps\right)^{b}+\left(y+\eps\right)^{-a}}{y}h_{\eps}\left(y\right)\min\left\{\frac{y}{L\ee^{t}},\esp\right\}.
  \end{split}
 \end{equation*}
 According to Lemma~\ref{Lem:moment:est:stand} the right hand side is integrable and thus we can also bound the integrand in $(III)$ independently of $\delta$ by some integrable function. Thus due to Lebesgue's Theorem we can pass to the limit $\delta\to 0$ in \eqref{eq:special:test:weak:2} and the claim then follows using Lemma~\ref{Lem:special:test:ineq}.
\end{proof}
 
 In the following we will now derive an estimate from below on $W$.

\subsubsection{Lower bound on $W$}

\begin{lemma}\label{L.westimate}
There exists $\sigma \in (\max\left\{b,\nu\right\},1)$ and  $\theta>0$ such that
 \begin{equation*}
  \begin{split}
   1-W(A-A^{\sigma}) &\leq C t A^{-\theta}
  \end{split}
 \end{equation*}
 for sufficiently large $A$.
 \end{lemma}

\begin{proof}
 From the construction in Section~\ref{Sec:existence:results} we know that $\tilde{W}$ can be written as $\tilde{W}\left(\xi,t\right)=\int_{\xi}^{\infty}G\left(\eta,t\right)\dd\eta$ with $G=-\del_{\xi}\tilde{W}=G_{1,1}\ast G_{1,2}\ast G_{2}$, where $G_{1,1},G_{1,2}$ and $G_{2}$ solve
 \begin{equation}\label{eq:subsolepsder}
  \begin{split}
   \del_{t}G_{1,1}&=\frac{C A^{-\nu a}}{L^{\rho+a-\max\left\{0,b\right\}}}\int_{0}^{\infty}\frac{1}{\eta^{1+\omega_{1}}}\left[G_{1,1}\left(\xi+\eta\right)-G_{1,1}\left(\xi\right)\right]\dd\eta\\
   G_{1,1}\left(\cdot,0\right)&=\delta\left(\cdot-A+\kappa\right)\ast\varphi_{\kappa/3}\\
   \del_{t}G_{1,2}&=\frac{CA^{\beta}}{L^{\rho-b}}\int_{0}^{\infty}\frac{1}{\eta^{1+\omega_{2}}}\left[G_{1,2}\left(\xi+\eta\right)-G_{1,2}\left(\xi\right)\right]\dd\eta\\
   G_{1,2}\left(\cdot,0\right)&=\delta\left(\cdot\right)\ast\varphi_{\kappa/3}\\
   \del_{t}G_{2}&=C\int_{0}^{1}\frac{h_{\eps}\left(z\right)}{z}\left[\frac{A^{-\nu a}}{L^{a}}\left(z+\eps\right)^{b}+A^{\beta}L^{\beta}\left(z+\eps\right)^{-a}\right]\cdot\left[G_{2}\left(\xi+\frac{z}{L}\right)-G_{2}\left(\xi\right)\right]\dz\\
   G_{2}\left(\cdot,0\right)&=\delta\left(\cdot\right)\ast\varphi_{\kappa/3}.
  \end{split}
 \end{equation}
Then one has from Lemma~\ref{Lem:der:int:est} for any $\mu\in\left(0,1\right)$:
\begin{equation*}
 \begin{split}
  \int_{-\infty}^{-D}G_{1,2}\left(\xi,t\right)\dd\xi\leq C\left(\frac{\kappa}{D}\right)^{\mu}+\frac{C A^{\beta}t}{L^{\rho-b}D^{\omega_{2}}}
 \end{split}
\end{equation*}
and
\begin{equation*}
 \begin{split}
  \int_{-\infty}^{-D+A}G_{1,1}\left(\xi,t\right)\dd\xi&=\int_{-\infty}^{\left(A-\kappa\right)-\left(D-\kappa\right)}G_{1,1}\left(\xi,t\right)\dd\xi \leq C\left(\frac{\kappa}{D-\kappa}\right)^{\mu}+\frac{C A^{-\nu a}t}{L^{\rho+a-\max\left\{0,b\right\}}\left(D-\kappa\right)^{\omega_{1}}}\\
  &\leq C\left(\frac{\kappa}{D}\right)^{\mu}+\frac{C A^{-\nu a}t}{L^{\rho+a-\max\left\{0,b\right\}}D^{\omega_{1}}}.
 \end{split}
\end{equation*}
In the last step we used that for any $\delta\in\left(0,1\right)$ and $D\geq 1$, $\kappa\leq 1/2$ it holds $\left(D-\kappa\right)^{-\delta}\leq 2^{\delta} D^{-\delta}$. One thus needs an estimate for $G_{2}$. This will be quite similar to the proof of Lemma~\ref{Lem:der:int:est} but due to the different behaviour for $L_{\eps}\to 0$ and $L_{\eps}\not \to 0$ we sketch this here again. Defining $\tilde{G}_{2}\left(p,t\right):=\int_{\R}G_{2}\left(\xi,t\right)\ee^{p\left(\xi-\kappa/3\right)}\dd\xi$ and multiplying the equation for $G_{2}$ in \eqref{eq:subsolepsder} by $\ee^{p\left(\xi-\kappa/3\right)}$ and integrating one obtains
\begin{equation*}
 \begin{split}
  \del_{t}\tilde{G}_{2}\left(p,t\right)&=C\int_{0}^{1}\frac{h_{\eps}\left(z\right)}{z}\left[\left(z+\eps\right)^{-a}L^{b}A^{\beta}+\left(z+\eps\right)^{b}L^{-a}A^{-\nu a}\right]\cdot \left[\ee^{-\frac{pz}{L}}-1\right]\dz \tilde{G}_{2}\left(p,t\right)\\
  &=:M\left(p,L\right)\tilde{G}_{2}\left(p,t\right). 
 \end{split}
\end{equation*}
 Thus $\tilde{G}_{2}\left(p,t\right)=\int_{\R}\varphi_{\kappa/3}\left(\xi\right)\ee^{p\left(\xi-\kappa/3\right)}\dd\xi\exp\left(-t\abs{M\left(p,L\right)}\right)$ and one can estimate:
 \begin{equation*}
  \begin{split}
   \abs{M\left(p,L\right)}&\leq C\int_{0}^{1}\frac{h_{\eps}\left(z\right)}{z}\left[\left(z+\eps\right)^{-a}L^{b}A^{\beta}+\left(z+\eps\right)^{-a}L^{-a}A^{-\nu a}\right]\cdot \frac{pz}{L}\dz\\
   &=C p\left(L^{b-1}A^{\beta}\mu_{\eps}+L^{-a-1}A^{-\nu a}\lambda_{\eps}\right)\leq Cp\left(A^{\beta}+A^{-\nu a}\right).
  \end{split}
 \end{equation*}
For the last step note that due to our notation either $L=L_{\eps}$ (in the case $L_{\eps}\not\to 0$) and then the estimate is due to the definition of $L_{\eps}$. If $L=1$ (in the case $L_{\eps}\to 0$) one can assume without loss of generality that $\eps$ is such small that $L_{\eps}\leq 1$ (and thus by definition also $\lambda_{\eps},\mu_{\eps}\leq 1$). Using this and inserting $p:=\frac{1}{D}$ we obtain in the same way as in the proof of Lemma~\ref{Lem:der:int:est}:
\begin{equation*}
 \begin{split}
  \int_{-\infty}^{-D}G_{2}\left(\xi,t\right)\dd\xi\leq C\left(\left(\frac{\kappa}{D}\right)^{\mu}+\frac{t}{D}\left(A^{\beta}+A^{-\nu a}\right)\right).
 \end{split}
\end{equation*} 
Using these estimates on $G_{1,1}$, $G_{1,2}$ and $G_{2}$ one obtains from Lemma~\ref{Lem:int:est:conolution} (note also Remark~\ref{Rem:est:conv}):
\begin{equation*}
 \begin{split}
  1-\tilde{W}\left(A-D\right)&=\int_{-\infty}^{A-D}\left(G_{1,1}\ast G_{1,2}\ast G_{2}\right)\dd\xi\leq \int_{-\infty}^{A-\frac{D}{4}}G_{1,1}\dd\xi+\int_{-\infty}^{-\frac{D}{4}}G_{1,2}\dd\xi+\int_{-\infty}^{-\frac{D}{4}}G_{2}\dd\xi\\
  &\leq C\frac{\kappa^{\mu}}{D^{\mu}}+\frac{C A^{-\nu a}t}{L^{\rho+a-\max\left\{0,b\right\}}D^{\omega_{1}}}+\frac{C A^{\beta}t}{L^{\rho-b}D^{\omega_{2}}}+\frac{C t\left(A^{\beta}+A^{-\nu a}\right)}{D}.
 \end{split}
\end{equation*}
Choosing $D=A^{\sigma}$ (with $A\geq 1$) one has
\begin{equation*}
 1-\tilde{W}\left(A-A^{\sigma}\right)\leq C\kappa^{\mu}A^{-\mu\sigma}+\frac{Ct}{L^{a+\omega_{1}}}A^{-\nu a-\sigma \omega_{1}}+\frac{Ct}{L^{\rho-b}}A^{\beta -\sigma \omega_{2}}+C t\left(A^{\beta-\sigma}+A^{-\nu a -\sigma}\right).
\end{equation*}
In the case $L=1$ (i.e. $L_{\eps}\to 0$) it suffices to consider the exponents of $A$:
\begin{itemize}
 \item $-\mu\sigma<0$, as $\mu,\sigma>0$,
 \item $-\nu a-\sigma\omega_{1}<0$, as $\omega_{1},a>0$,
 \item $\beta-\sigma\omega_{2}=\beta-\sigma\rho=\begin{cases}
                                                 b-\sigma \rho & b\geq 0\\
                                                 \nu b-\sigma \rho & b<0
                                                \end{cases}
                                                <0$, independently of the sign of $b$ if we choose $\sigma$ sufficiently close to $1$ and $\sigma>\nu$ (as $b<\rho$).
 \item $\beta-\sigma=\begin{cases}
                      b-\sigma & b\geq 0\\
                      \nu b-\sigma & b<0
                     \end{cases}
                     <0$, independently of the sign of $b$ if we choose $\sigma>b$ as $\nu<1$ and $b<1$ (note that this choice of $\sigma$ does not collide with the choice made before)
 \item $-\nu a -\sigma <0$, as $a,\sigma>0$.
\end{itemize}
Thus, taking $-\theta$ to be the maximum of the (negative) exponents proves the claim in this case. 

If $L=L_{\eps}$ (i.e. $L_{\eps}\not\to 0$) one has to consider also the exponents of $L$:
\begin{itemize}
 \item $a+\omega_{1}>0$, as $a,\omega_{1}>0$,
 \item $\rho-b>0$, as by assumption $b<\rho$.
\end{itemize}
Thus either the two terms containing $L=L_{\eps}$ are bounded (if $L_{\eps}$ is bounded) or converge to zero (if $L_{\eps}\to \infty$) and so in both cases with the same $\theta>0$ as above the claim follows.
\end{proof}

\subsubsection{The iteration argument}\label{subsec:iterate}
In this section we will show Proposition~\ref{Prop:Hepslowerbound}. We therefore define 
\begin{equation*}
 F_{\eps}\left(X\right):=\int_{0}^{X}H_{\eps}\left(Y\right)\dY\quad \text{while for }L_{\eps}\to 0 \text{ this reduces to } \quad F_{\eps}\left(x\right)=\int_{0}^{x}h_{\eps}\left(y\right)\dy.
\end{equation*}
 We first show the following Lemma, that will be the key in the proof of Proposition~\ref{Prop:Hepslowerbound}.
\begin{lemma}\label{Lem:recursion}
 There exists $\theta>0$ such that
 \begin{equation*}
  \begin{split}
   F_{\eps}\left(A\right)&\geq -CA^{\nu\left(1-\rho\right)}+F_{\eps}\left(\left(A-A^{\sigma}\right)\ee^{T}\right)\ee^{-\left(1-\rho\right)T}\left(1-\frac{C}{A^{\theta}}\right) 
  \end{split}
 \end{equation*}
 for $A$ sufficiently large.
\end{lemma}

\begin{proof}
 From the choice of $\psi$ and $W$ respectively (using also the non-negativity and monotonicity properties of $W$)
 \begin{equation*}
  \begin{split}
   F_{\eps}\left(A\right)&=\int_{0}^{A}H_{\eps}\left(X\right)\dX\geq\int_{0}^{\infty}W\left(X,0\right)H_{\eps}\left(X\right)\dX\geq \ee^{-\left(1-\rho\right)T}\int_{0}^{\infty}H_{\eps}\left(X\right)W\left(\frac{X}{\ee^{T}},T\right)\dX\\
   &\geq\ee^{-\left(1-\rho\right)T}\int_{A^{\nu}}^{\infty}\del_{X}F_{\eps}\left(X\right)W\left(\frac{X}{\ee^{T}},T\right)\dX\\
   &=-\ee^{-\left(1-\rho\right)T}F_{\eps}\left(A^{\nu}\right)W\left(\frac{A^{\nu}}{\ee^{T}},T\right)-\int_{A^{\nu}}^{\infty}\ee^{-\left(1-\rho\right)T}\ee^{-T}F_{\eps}\left(X\right)\del_{\xi}W\left(\frac{X}{\ee^{T}},T\right)\dX\\
   &\geq -CA^{\nu\left(1-\rho\right)}+\ee^{-\left(1-\rho\right)T}\int_{A^{\nu}\ee^{-T}}^{\infty}F_{\eps}\left(X\ee^{T}\right)\left(G_{1,1}\ast G_{1,2}\ast G_{2}\right)\left(X,T\right)\dX
 \end{split}
 \end{equation*}
 where we changed variables in the last step and used that $W$ is bounded, $\del_{\xi}W=-G_{1,1}\ast G_{1,2}\ast G_{2}$ on $\left(A^{\nu},\infty\right)$ as well as $ F_{\eps}\left(A^{\nu}\right)\leq A^{\nu\left(1-\rho\right)}$. Noting that for $\sigma>\nu$ we have $A^{\nu}\ee^{-T}\leq A-A^{\sigma}$ for sufficiently large $A$ and using also the monotonicity of $F_{\eps}$ we can further estimate
 \begin{equation*}
  \begin{split}
   F_{\eps}\left(A\right)&\geq -CA^{\nu\left(1-\rho\right)}+\ee^{-\left(1-\rho\right)T}\int_{A-A^{\sigma}}^{\infty}F_{\eps}\left(X\ee^{T}\right)\left(G_{1,1}\ast G_{1,2}\ast G_{2}\right)\left(X,T\right)\dX\\
   &\geq -CA^{\nu\left(1-\rho\right)}+\ee^{-\left(1-\rho\right)T}F_{\eps}\left(\left(A-A^{\sigma}\right)\ee^{T}\right)\int_{A-A^{\sigma}}^{\infty}\left(G_{1,1}\ast G_{1,2}\ast G_{2}\right)\left(X,T\right)\dX\\
   &=-CA^{\nu\left(1-\rho\right)}+F_{\eps}\left(\left(A-A^{\sigma}\right)\ee^{T}\right)\ee^{-\left(1-\rho\right)T}W\left(A-A^{\sigma}\right)\\
   &\geq -CA^{\nu\left(1-\rho\right)}+F_{\eps}\left(\left(A-A^{\sigma}\right)\ee^{T}\right)\ee^{-\left(1-\rho\right)T}\left(1-\frac{C}{A^{\theta}}\right),
  \end{split}
 \end{equation*}
while in the last step Lemma~\ref{L.westimate} was applied.
\end{proof}
 We are now prepared to prove Proposition~\ref{Prop:Hepslowerbound}. This will be done by some iteration argument using recursively Lemma~\ref{Lem:recursion}.
\begin{proof}[Proof of Proposition~\ref{Prop:Hepslowerbound}]
 Let $\alpha:=\ee^{T}>1$. For any $\delta>0$ there exists $R_{\epsilon,\delta}>0$ such that $F_{\eps}\left(R\right)\geq R^{1-\rho}\left(1-\delta\right)$ for all $R\geq R_{\eps,\delta}$. For $A_{0}>\left(\frac{\alpha}{\alpha-1}\right)^{\frac{1}{1-\sigma}}$ we define a sequence $\left\{A_{k}\right\}_{k\in\N_{0}}$ by $A_{k+1}:=\alpha\left(A_{k}-A_{k}^{\sigma}\right)$. From the choice of $A_{0}$ one obtains that $A_{k}$ is strictly increasing and one has $A_{k}\to\infty$ as $k\to\infty$. Furthermore $\alpha A_{k}=A_{k+1}+\alpha A_{k}^{\sigma}$ and thus 
 \begin{equation*}
  A_{k}=\frac{A_{k+1}}{\alpha}\left(1+\alpha\frac{A_{k}^{\sigma}}{A_{k+1}}\right).
 \end{equation*}
 By iteration one obtains for any $N\in\N$:
 \begin{equation}\label{eq:reprR0}
  A_{0}=\frac{A_{N}}{\alpha^{N}}\prod_{k=0}^{N-1}\left(1+\alpha\frac{A_{k}^{\sigma}}{A_{k+1}}\right).
 \end{equation}
 For any $N\in\N$ and $0\leq k<N$ applying Lemma~\ref{Lem:recursion} one gets by induction:
 \begin{equation}\label{eq:reprFepsk}
  \begin{split}
   F_{\eps}\left(A_{k}\right)\geq F_{\eps}\left(A_{N}\right)\alpha^{-\left(N-k\right)\left(1-\rho\right)}\prod_{n=k}^{N-1}\left(1-\frac{C}{A_{n}^{\theta}}\right)-C\sum_{m=k}^{N-1}\alpha^{-\left(m-k\right)\left(1-\rho\right)}\left(\prod_{n=k}^{m-1}\left(1-\frac{C}{A_{n}^{\theta}}\right)\right)A_{m}^{\nu\left(1-\rho\right)},
  \end{split}
 \end{equation}
 where we use the convention $\sum_{k=l}^{u}a_{k}=0$ and $\prod_{k=l}^{u}a_{k}=1$ if $u<l$. Thus for $k=0$ one particularly obtains
 \begin{equation}\label{eq:lowerbditI0}
  \begin{split}
   F_{\eps}\left(A_{0}\right)&\geq F_{\eps}\left(A_{N}\right)\alpha^{-N\left(1-\rho\right)}\prod_{n=0}^{N-1}\left(1-\frac{C}{A_{n}^{\theta}}\right)-C\sum_{m=0}^{N-1}\alpha^{-m\left(1-\rho\right)}\left(\prod_{n=0}^{m-1}\left(1-\frac{C}{A_{n}^{\theta}}\right)\right)A_{m}^{\nu\left(1-\rho\right)}\\
   &=F_{\eps}\left(A_{N}\right)\alpha^{-N\left(1-\rho\right)}\prod_{n=0}^{N-1}\left(1-\frac{C}{A_{n}^{\theta}}\right)-C\sum_{m=0}^{N-1}\alpha^{\left(\nu-1\right)\left(1-\rho\right)m}\left(\prod_{n=0}^{m-1}\left(1-\frac{C}{A_{n}^{\theta}}\right)\right)\left(\alpha^{-m}A_{m}\right)^{\nu\left(1-\rho\right)}\\
   &=:\left(I\right)-\left(II\right).
  \end{split}
 \end{equation}
 We now estimate the two terms separately.

 Let $\delta_{*}:=\delta/2$. Choosing $N$ sufficiently large such that $A_{N}\geq R_{\epsilon,\delta_{*}}$ one has, using also~\eqref{eq:reprR0} 
 \begin{equation}\label{eq:lowerbditI1}
  \begin{split}
   \left(I\right)&\geq \left(1-\delta_{*}\right)A_{N}^{1-\rho}\alpha^{-N\left(1-\rho\right)}\prod_{n=0}^{N-1}\left(1-\frac{C}{A_{n}^{\theta}}\right)=\left(1-\delta_{*}\right)\left(\frac{A_{N}}{\alpha^{N}}\right)^{1-\rho}\prod_{n=0}^{N-1}\left(1-\frac{C}{A_{n}^{\theta}}\right)\\
   &=\left(1-\delta_{*}\right)A_{0}^{1-\rho}\frac{\prod_{n=0}^{N-1}\left(1-\frac{C}{A_{n}^{\theta}}\right)}{\left(\prod_{k=0}^{N-1}\left(1+\alpha\frac{A_{k}^{\sigma}}{A_{k+1}}\right)\right)^{1-\rho}}.
  \end{split}
 \end{equation}
 Let $0<D_{0}<D$ be parameters to be fixed later and assume $A_{0}>D$. One has $A_{k+1}=\alpha\left(A_{k}-A_{k}^{\sigma}\right)$. Thus using the monotonicity of $A_{k}$
 \begin{equation*}
  \begin{split}
   \frac{A_{k+1}}{A_{k}}=\alpha\left(1-A_{k}^{\sigma-1}\right)>\alpha\left(1-D_{0}^{\sigma-1}\right)=:\beta_{0}>1
  \end{split}
 \end{equation*}
 if we fix $D_{0}$ sufficiently large as $\alpha>1$. Using this, one has $A_{k+1}>\beta_{0}A_{k}$ and thus by iteration $A_{k+1}>\beta_{0}^{k+1}A_{0}$.\\
 We continue to estimate $(I)$ and thus consider first $\prod_{n=0}^{N-1}\left(1-\frac{C}{A_{n}^{\theta}}\right)$ while we assume that $D_{0}$ is sufficiently large such that $\frac{C}{D^{\theta}}<1$ and thus also $\frac{C}{A_{n}^{\theta}}<1$ by the monotonicity of $A_{n}$. Taking the logarithm of the product one has using the estimate $\log\left(1-x\right)\geq -\frac{x}{1-x}$:
 \begin{equation*}
  \begin{split}
   \sum_{n=0}^{N-1}\log\left(1-\frac{C}{A_{n}^{\theta}}\right)&\geq \sum_{n=0}^{N-1}-\frac{C}{A_{n}^{\theta}}\cdot \frac{1}{1-\frac{C}{A_{n}^{\theta}}}=-C\sum_{n=0}^{N-1}\frac{1}{A_{n}^{\theta}-C}\geq -C\sum_{n=0}^{N-1}\frac{1}{\beta_{0}^{n\theta}A_{0}^{\theta}-C}\\
   &\geq -C\sum_{n=0}^{N-1}\frac{1}{\beta_{0}^{\theta n}}\frac{1}{D^{\theta}-C}\geq -C\frac{\beta_{0}^{\theta}}{\beta_{0}^{\theta}-1}\frac{1}{D^{\theta}-C}=:-\frac{C_{\beta}}{D^{\theta}-C}.
  \end{split}
 \end{equation*}
 Thus one obtains using $\exp\left(-x\right)\geq 1-x$:
 \begin{equation}\label{eq:lowerbditI2}
  \begin{split}
   \prod_{n=0}^{N-1}\left(1-\frac{C}{A_{n}^{\theta}}\right)\geq \exp\left(-\frac{C_{\beta}}{D^{\theta}-C}\right)\geq 1-\frac{C_{\beta}}{D^{\theta}-C}.
  \end{split}
 \end{equation}
 Considering $\prod_{k=0}^{N-1}\left(1+\alpha\frac{A_{k}^{\sigma}}{A_{k+1}}\right)$ and applying again first the logarithm on the product and then using $\log\left(1+x\right)\leq x$ one obtains
 \begin{equation*}
  \begin{split}
   \sum_{k=0}^{N-1}\log\left(1+\alpha\frac{A_{k}^{\sigma}}{A_{k+1}}\right)&\leq \sum_{k=0}^{N-1}\alpha \frac{A_{k}^{\sigma}}{A_{k+1}}\leq \alpha \sum_{k=0}^{N-1}A_{k}^{\sigma-1}\leq \alpha A_{0}^{\sigma-1}\sum_{k=0}^{N-1}\left(\beta_{0}^{\sigma-1}\right)^{k}\leq \alpha D^{\sigma-1}\sum_{k=0}^{\infty}\left(\beta_{0}^{\sigma-1}\right)^{k}\\
   &=:C_{\gamma}D^{\sigma-1}.
  \end{split}
 \end{equation*}
 Thus one can estimate
 \begin{equation}\label{eq:lowerbditI3}
  \begin{split}
   \left(\prod_{k=0}^{N-1}\left(1+\alpha\frac{A_{k}^{\sigma}}{A_{k+1}}\right)\right)^{1-\rho}&\leq \exp\left[\left(C_{\gamma}D^{\sigma-1}\right)\left(1-\rho\right)\right]=1+\sum_{n=1}^{\infty}\frac{\left(1-\rho\right)^{n}C_{\gamma}^{n}D^{n\left(\sigma-1\right)}}{n!}\\
   &\leq 1+D^{\sigma-1}\sum_{n=1}^{\infty}\frac{C_{\gamma}^{n}\left(1-\rho\right)^{n}}{n!}\leq 1+\frac{\exp\left(C_{\gamma}\left(1-\rho\right)\right)}{D^{1-\sigma}}.
  \end{split}
 \end{equation}
 Putting the estimates of \eqref{eq:lowerbditI2} and \eqref{eq:lowerbditI3} together one obtains
 \begin{equation}\label{eq:lowerbditI4}
  \begin{split}
   \frac{\prod_{n=0}^{N-1}\left(1-\frac{C}{A_{n}^{\theta}}\right)}{\left(\prod_{k=0}^{N-1}\left(1+\alpha\frac{A_{k}^{\sigma}}{A_{k+1}}\right)\right)^{1-\rho}}&\geq \frac{1-\frac{C_{\beta}}{D^{\theta}-C}}{1+\frac{\exp\left(C_{\gamma}\left(1-\rho\right)\right)}{D^{1-\sigma}}}=1-\frac{\frac{C_{\beta}}{D^{\theta}-C}+\frac{\exp\left(C_{\gamma}\left(1-\rho\right)\right)}{D^{1-\sigma}}}{1+\frac{\exp\left(C_{\gamma}\left(1-\rho\right)\right)}{D^{1-\sigma}}}\\
   &\geq 1-\frac{C_{\beta}}{D^{\theta}-C}-\frac{\exp\left(C_{\gamma}\left(1-\rho\right)\right)}{D^{1-\sigma}}.
  \end{split}
 \end{equation}
 Together with \eqref{eq:lowerbditI1} this shows
 \begin{equation}\label{eq:lowerbditI5}
  \begin{split}
   (I)\geq \left(1-\delta_{*}\right)A_{0}^{1-\rho}\left(1-\frac{C_{\beta}}{D^{\theta}-C}-\frac{\exp\left(C_{\gamma}\left(1-\rho\right)\right)}{D^{1-\sigma}}\right).   
  \end{split}
 \end{equation}

 To estimate $(II)$ we note first that one has $\prod_{n=0}^{m-1}\left(1-\frac{C}{A_{n}^{\theta}}\right)\leq 1$ as well as $\prod_{k=0}^{m-1}\left(1+\alpha\frac{A_{k}^{\sigma}}{A_{k+1}}\right)\geq 1$ for any $m\in \N_{0}$. Using this as well as \eqref{eq:reprR0} for $N=m$ and $m=0,\ldots,N-1$ one obtains
 \begin{equation}\label{eq:lowerbditII1}
  \begin{split}
   (II)&=C\sum_{m=0}^{N-1}\alpha^{\left(\nu-1\right)\left(1-\rho\right)m}\left(\prod_{n=0}^{m-1}\left(1-\frac{C}{A_{n}^{\theta}}\right)\right)\left(\alpha^{-m}A_{m}\right)^{\nu\left(1-\rho\right)}\leq C\sum_{m=0}^{N-1}\alpha^{\left(\nu-1\right)\left(1-\rho\right)m}\left(\alpha^{-m}A_{m}\right)^{\nu\left(1-\rho\right)}\\
   &=C\sum_{m=0}^{N-1}\alpha^{\left(\nu-1\right)\left(1-\rho\right)m}\frac{A_{0}^{\nu\left(1-\rho\right)}}{\left(\prod_{k=0}^{m-1}\left(1+\alpha\frac{A_{k}^{\sigma}}{A_{k+1}}\right)\right)^{\nu\left(1-\rho\right)}}\leq C\sum_{m=0}^{N-1}\alpha^{\left(\nu-1\right)\left(1-\rho\right)m}A_{0}^{\nu\left(1-\rho\right)}\\
   &\leq A_{0}^{\nu\left(1-\rho\right)}C\sum_{m=0}^{\infty}\alpha^{\left(\nu-1\right)\left(1-\rho\right)m}=:C_{\nu}A_{0}^{\nu\left(1-\rho\right)}
  \end{split}
 \end{equation}
 Putting together the estimates \eqref{eq:lowerbditI5}, \eqref{eq:lowerbditII1} and \eqref{eq:lowerbditI0} yields
 \begin{equation}\label{eq:lowerbdit}
  \begin{split}
   F_{\eps}\left(A_{0}\right)&\geq \left(1-\delta_{*}\right)\left(1-\frac{C_{\beta}}{D^{\theta}-C}-\frac{\exp\left(C_{\gamma}\left(1-\rho\right)\right)}{D^{1-\sigma}}\right)A_{0}^{1-\rho}-C_{\nu}A_{0}^{\nu\left(1-\rho\right)}\\
   &\geq A_{0}^{1-\rho}\left(\left(1-\delta_{*}\right)\left(1-\frac{C_{\beta}}{D^{\theta}-C}-\frac{\exp\left(C_{\gamma}\left(1-\rho\right)\right)}{D^{1-\sigma}}\right)-\frac{C_{\nu}}{D^{\left(1-\nu\right)\left(1-\rho\right)}}\right).
  \end{split}
 \end{equation}
 We choose now $D$ sufficiently large such that one has
 \begin{itemize}
  \item $D\geq \left(3C_{\beta}\frac{2-\delta}{\delta}+C\right)^{1/\theta}$ which is equivalent to $\frac{C_{\beta}}{D^{\theta}-C}\leq \frac{1}{3}\frac{\delta}{2-\delta}$
  \item $D\geq \left(3\exp\left(C_{\gamma}\left(1-\rho\right)\right)\frac{2-\delta}{\delta}\right)^{\frac{1}{1-\sigma}}$ which is equivalent to $\frac{\exp\left(C_{\gamma}\left(1-\rho\right)\right)}{D^{1-\sigma}}\leq \frac{1}{3}\frac{\delta}{2-\delta}$
  \item $D\geq \left(\frac{3C_{\nu}}{1-\delta/2}\frac{2-\delta}{\delta}\right)^{\frac{1}{\left(1-\nu\right)\left(1-\rho\right)}}$ which is equivalent to $C_{\nu}\frac{1}{D^{\left(1-\nu\right)\left(1-\rho\right)}}\leq \left(1-\delta/2\right)\frac{1}{3}\frac{\delta}{2-\delta}$.
 \end{itemize}
 Inserting these estimates in \eqref{eq:lowerbdit} together with $\delta_{*}=\delta/2$ one gets
 \begin{equation*}
  \begin{split}
   F_{\eps}\left(A_{0}\right)&\geq A_{0}^{1-\rho}\left(\left(1-\frac{\delta}{2}\right)\left(1-\frac{2}{3}\frac{\delta}{2-\delta}\right)-\left(1-\frac{\delta}{2}\right)\frac{1}{3}\frac{\delta}{2-\delta}\right)\\
   &=A_{0}^{1-\rho}\left(1-\frac{\delta}{2}\right)\left(1-\frac{\delta}{2-\delta}\right)=\left(1-\delta\right)A_{0}^{1-\rho}.
  \end{split}
 \end{equation*}
 This proves the claim with $R_{\delta}=D$.
\end{proof}

\subsection{Excluding $L_{\eps}\to\infty$}

Before we can pass to the limit $\eps\to 0$ we have to exclude the case $L_{\eps}\to \infty$ as $\eps\to 0$ in order to obtain Proposition~\ref{Prop:Hepslowerbound} also for $h_{\eps}$ (instead of $H_{\eps}$). This will be done by some contradiction argument. We furthermore remark here that throughout this section we frequently use that we can bound
\begin{equation}\label{eq:bound:Leps}
 \begin{split}
  \int_{0}^{1}\left(x+\eps\right)^{-a}h_{\eps}\left(x\right)\dx\leq L_{\eps}^{1-b} \quad \text{and} \quad \int_{0}^{1}\left(x+\eps\right)^{b}h_{\eps}\left(x\right)\dx\leq L_{\eps}^{1+a}  
 \end{split}
\end{equation}
due to the definition of $L_{\eps}$ in \eqref{lepsdef}.

\subsubsection{Deriving a limit equation for $H_{\eps}$}

We first show the following lemma, stating the convergence of a certain integral occurring later.

\begin{lemma}\label{Lem:convergence:Q}
 Assume $L_{\eps}\to\infty$ as $\eps\to 0$. Let $Q_{\eps}$ be given by 
 \begin{equation*}
  Q_{\eps}\left(X\right):=\int_{0}^{1}\frac{h_{\eps}\left(y\right)}{L_{\eps}}K_{\eps}\left(y,L_{\eps}X\right)\dy.
 \end{equation*}
 Then there exists a (continuous) function $Q$ such that $Q_{\eps}\to Q$ locally uniformly up to a subsequence.
\end{lemma}

\begin{proof}
 It suffices to show that both $Q_{\eps}$ as well as $Q_{\eps}'$ are uniformly bounded on each fixed interval $\left[d,D\right]\subset \R_{+}$. One has using \eqref{eq:bound:Leps}
 \begin{equation}\label{eq:est:Qeps}
  \begin{split}
   \abs{Q_{\eps}\left(X\right)}&\leq C_{2}\int_{0}^{1}\frac{h_{\eps}\left(z\right)}{L_{\eps}}\left(\left(L_{\eps}X+\eps\right)^{-a}\left(z+\eps\right)^{b}+\left(L_{\eps}X+\eps\right)^{b}\left(z+\eps\right)^{-a}\right)\dz\\
   &\leq \frac{C_{2}}{L_{\eps}}\left(L_{\eps}X+\eps\right)^{-a}\int_{0}^{1}\left(z+\eps\right)^{b}h_{\eps}\left(z\right)\dz+\frac{C_{2}}{L_{\eps}}\left(L_{\eps}X+\eps\right)^{b}\int_{0}^{1}\left(z+\eps\right)^{-a}h_{\eps}\left(z\right)\dz\\
   &\leq C_{2}L_{\eps}^{a}\left(L_{\eps}X+\eps\right)^{-a}+C_{2}L_{\eps}^{-b}\left(L_{\eps}X+\eps\right)^{b}\\
   &= C_{2}\left(X+\frac{\eps}{L_{\eps}}\right)^{-a}+C_{2}\left(X+\frac{\eps}{L_{\eps}}\right)^{b}
  \end{split}
 \end{equation}
 while the right hand side is clearly locally uniformly bounded under the given assumptions. Rewriting $Q_{\eps}$ one obtains
 \begin{equation*}
  \begin{split}
   Q_{\eps}\left(X\right)=\int_{0}^{1}\frac{h_{\eps}\left(z\right)}{L_{\eps}}K_{\eps}\left(L_{\eps}X,z\right)\dz=L_{\eps}^{\gamma}\int_{0}^{1}\frac{h_{\eps}\left(z\right)}{L_{\eps}}K_{\frac{\eps}{L_{\eps}}}\left(X,\frac{z}{L_{\eps}}\right)\dz.
  \end{split}
 \end{equation*}
 Furthermore from \eqref{eq:Ass2} one has
 \begin{equation*}
  \abs{\del_{y}K_{\eps}\left(y,z\right)}\leq C\left(\left(z+\eps\right)^{-a}+\left(z+\eps\right)^{b}\right) \quad \text{for all } y\in\left[a,A\right]
 \end{equation*}
 and hence, similarly as before
 \begin{equation*}
  \begin{split}
   \abs{Q_{\eps}'\left(X\right)}&\leq CL_{\eps}^{\gamma}\int_{0}^{1}\frac{h_{\eps}\left(z\right)}{L_{\eps}}\left[\left(\frac{z+\eps}{L_{\eps}}\right)^{-a}+\left(\frac{z+\eps}{L_{\eps}}\right)^{b}\right]\dz\\
   &=CL_{\eps}^{\gamma-1+a}\int_{0}^{1}\left(z+\eps\right)^{-a}h_{\eps}\left(z\right)\dz+CL_{\eps}^{\gamma-1-b}\int_{0}^{1}\left(z+\eps\right)^{b}h_{\eps}\left(z\right)\dz\\
   &\leq C
  \end{split}
 \end{equation*}
 where we used also $\gamma=b-a$.
\end{proof}

\begin{lemma}
 Let $\rho\in\left(\max\left\{0,b\right\},1\right)$ and assume $L_{\eps}\to\infty$ as $\eps\to 0$. Then there exists a measure $H$ such that (up to a subsequence) $H_{\eps}\stackrel{*}{\rightharpoonup}H$ and $H$ satisfies
 \begin{equation}\label{eq:H:infty}
  \del_{X}\left(XH\right)+\left(\rho-1\right)H-\del_{X}\left(Q\left(X\right)H\left(X\right)\right)+\frac{H\left(X\right)Q\left(X\right)}{X}=0
 \end{equation}
 in the sense of distributions with $Q\left(X\right)=\lim_{\eps\to 0}\int_{0}^{1}\frac{h_{\eps}\left(y\right)}{L_{\eps}}K_{\eps}\left(y,L_{\eps}X\right)\dy$.
\end{lemma}

\begin{proof}
 Transforming the equation
 \begin{equation*}
  \del_{x}\left(\int_{0}^{x}\int_{x-y}^{\infty}\frac{K_{\eps}\left(y,z\right)}{z}h_{\eps}\left(y\right)h_{\eps}\left(z\right)\dz\dy\right)=\del_{x}\left(xh_{\eps}\left(x\right)\right)+\left(\rho-1\right)h_{\eps}\left(x\right)
 \end{equation*}
 to the rescaled variables $X=\frac{x}{L_{\eps}}$ one obtains
 \begin{equation*}
  \frac{1}{L_{\eps}}\del_{X}\left(\int_{0}^{L_{\eps}X}\int_{L_{\eps}X-y}^{\infty}\frac{K_{\eps}\left(y,z\right)}{z}h_{\eps}\left(y\right)h_{\eps}\left(z\right)\dz\dy\right)=\frac{1}{L_{\eps}}\del_{X}\left(L_{\eps}X\frac{H\left(X\right)}{L_{\eps}^{\rho}}\right)+\left(\rho-1\right)\frac{H_{\eps}\left(X\right)}{L_{\eps}^{\rho}}.
 \end{equation*}
 Testing with $\psi\in C_{c}^{\infty}\left(\R_{+}\right)$ (in the rescaled $X$-variable), splitting the integral and interchanging the order of integration we can rewrite this as
 \begin{equation*}
  \begin{split}
   &\quad\int_{0}^{\infty}\left(X\del_{X}\psi\left(X\right)-\left(\rho-1\right)\psi\left(X\right)\right)H_{\eps}\left(X\right)\dX\\
   &=\frac{1}{L_{\eps}^{\rho-\gamma}}\int_{\frac{1}{L_{\eps}}}^{\infty}\int_{\frac{1}{L_{\eps}}}^{\infty}\frac{K_{\frac{\eps}{L_{\eps}}}\left(Y,Z\right)}{Z}H_{\eps}\left(Y\right)H_{\eps}\left(Z\right)\left[\psi\left(Y+Z\right)-\psi\left(Y\right)\right]\dZ\dY\\
   &\quad + \int_{\frac{1}{L_{\eps}}}^{\infty}\int_{0}^{1}\frac{K_{\eps}\left(L_{\eps}Y,z\right)}{z}h_{\eps}\left(z\right)H_{\eps}\left(Y\right)\left[\psi\left(Y+\frac{z}{L_{\eps}}\right)-\psi\left(Y\right)\right]\dz\dY\\
   &\quad + \int_{0}^{1}\int_{0}^{\infty}\left(\int_{\frac{y}{L_{\eps}}}^{\frac{y}{L_{\eps}}+Z}\del_{X}\psi\left(X\right)\dX\right)\frac{K_{\eps}\left(y,L_{\eps}Z\right)}{L_{\eps}Z}h_{\eps}\left(y\right)H_{\eps}\left(Z\right)\dZ\dy\\
   &=\left(I\right)+\left(II\right)+\left(III\right).
  \end{split}
 \end{equation*}
In the following we assume that $\supp\psi\subset\left[d,D\right]$ with $d>0$ and $D>1$. Furthermore we can assume that $L_{\eps}>1$ is sufficiently large and that $\eps<1$ is sufficiently small (as we are assuming $L_{\eps}\to\infty$ for $\eps\to 0$). We first show that $\left(I\right)\to 0$ as $\eps\to 0$:
\begin{equation*}
 \begin{aligned}
  &\quad\abs{\left(I\right)}\\
   &\leq \frac{C_{2}}{L_{\eps}^{\rho-\gamma}}\int_{\frac{1}{L_{\eps}}}^{D}\int_{\frac{1}{L_{\eps}}}^{\infty}\frac{\left(Y+\frac{\eps}{L_{\eps}}\right)^{-a}\left(Z+\frac{\eps}{L_{\eps}}\right)^{b}}{Z}H_{\eps}\left(Y\right)H_{\eps}\left(Z\right)\abs{\psi\left(Y+Z\right)-\psi\left(Y\right)}\dZ\dY\\
  &\quad +\frac{C_{2}}{L_{\eps}^{\rho-\gamma}}\int_{\frac{1}{L_{\eps}}}^{D}\int_{\frac{1}{L_{\eps}}}^{\infty}\frac{\left(Y+\frac{\eps}{L_{\eps}}\right)^{b}\left(Z+\frac{\eps}{L_{\eps}}\right)^{-a}}{Z}H_{\eps}\left(Y\right)H_{\eps}\left(Z\right)\abs{\psi\left(Y+Z\right)-\psi\left(Y\right)}\dZ\dY\\
  &\leq \frac{2^{\tilde{b}}C}{L_{\eps}^{\rho-\gamma}}\int_{\frac{1}{L_{\eps}}}^{D}\int_{\frac{1}{L_{\eps}}}^{\infty}\frac{Y^{-a}Z^{b}+Y^{b}Z^{-a}}{Z}H_{\eps}\left(Y\right)H_{\eps}\left(Z\right)\abs{\psi\left(Y+Z\right)-\psi\left(Y\right)}\dZ\dY\\
  &\leq \frac{C\norm{\psi'}_{L^{\infty}}}{L_{\eps}^{\rho-\gamma-a}}\int_{\frac{1}{L_{\eps}}}^{D}\int_{\frac{1}{L_{\eps}}}^{1}\left(Z^{b}+Y^{b}\right)H_{\eps}\left(Y\right)H_{\eps}\left(Z\right)\dZ\dY+\frac{C\norm{\psi}_{L^{\infty}}}{L_{\eps}^{\rho-\gamma}}\left[L_{\eps}^{a}\int_{\frac{1}{L_{\eps}}}^{D}H_{\eps}\left(Y\right)\dY\int_{1}^{\infty}Z^{b-1}H_{\eps}\left(Z\right)\dZ\right.\\
  &\quad\left.+\max\left\{L_{\eps}^{-b},D^{b}\right\}\int_{\frac{1}{L_{\eps}}}^{D}H_{\eps}\left(Y\right)\dY\int_{1}^{\infty}Z^{-a-1}H_{\eps}\left(Z\right)\dZ\right]\\
  &\leq CL_{\eps}^{\gamma+a-\rho}\max\left\{1,L_{\eps}^{-b},D^{b}\right\}\int_{\frac{1}{L_{\eps}}}^{D}H_{\eps}\left(Y\right)\dY\int_{\frac{1}{L_{\eps}}}^{1}H_{\eps}\left(Z\right)\dZ+CL_{\eps}^{\gamma-\rho}\left[L_{\eps}^{a}D^{1-\rho}+\max\left\{L_{\eps}^{-b},D^{b}\right\}D^{1-\rho}\right]\\
  &=CD^{1-\rho}L_{\eps}^{b-\rho}\max\left\{D^{b},L_{\eps}^{-b}\right\}+CD^{1-\rho}L_{\eps}^{b-\rho}+CD^{1-\rho}L_{\eps}^{b-a-\rho}\max\left\{L_{\eps}^{-b},D^{b}\right\}\\
  &\to 0
 \end{aligned}
\end{equation*}
as $L_{\eps}\to\infty$ (i.e. for $\eps\to 0$ by assumption).

 Next we show $\left(II\right)\to \int_{0}^{\infty}\del_{X}\psi\left(X\right)H\left(X\right)Q\left(X\right)\dX$. As $H_{\eps}$ is a sequence of locally uniformly bounded (non-negative Radon) measures there exists a (non-negative Radon) measure $H$ such that $H_{\eps}\stackrel{*}{\rightharpoonup} H$ in the sense of measures. Using now Taylor's formula for $\psi$ one obtains
\begin{equation*}
 \psi\left(Y+\frac{z}{L_{\eps}}\right)-\psi\left(Y\right)=\psi'\left(Y\right)\cdot \frac{z}{L_{\eps}}+\frac{z^{2}}{L_{\eps}^{2}}\int_{0}^{z}\left(z-t\right)\psi''\left(Y+\frac{t}{L_{\eps}}\right)\dt.
\end{equation*}
Using this in $(II)$ one gets
\begin{equation*}
 \begin{split}
  \left(II\right)&=\int_{\frac{1}{L_{\eps}}}^{\infty}\int_{0}^{1}\frac{K_{\eps}\left(L_{\eps}Y,z\right)}{z}h_{\eps}\left(z\right)H_{\eps}\left(Y\right)\cdot\frac{z}{L_{\eps}}\psi'\left(Y\right)\dz\dY\\
  &\quad+\int_{\frac{1}{L_{\eps}}}^{\infty}\int_{0}^{1}\frac{K_{\eps}\left(L_{\eps}Y,z\right)}{z}h_{\eps}\left(z\right)H_{\eps}\left(Y\right)\cdot\frac{z^{2}}{L_{\eps}^{2}}\int_{0}^{z}\left(z-t\right)\psi''\left(Y+\frac{t}{L_{\eps}}\right)\dt\dz\dY\\
  &=\left(II\right)_{a}+\left(II\right)_{b}
 \end{split}
\end{equation*}
We consider terms separately beginning with $\left(II\right)_{b}$ (and assuming $L_{\eps}$ to be sufficiently large, i.e. $\frac{1}{L_{\eps}}<d$):
\begin{equation*}
 \begin{split}
  \abs{\left(II\right)_{b}}&\leq \frac{1}{L_{\eps}^{2}}\int_{\frac{1}{L_{\eps}}}^{\infty}\int_{0}^{1}z^{2}h_{\eps}\left(z\right)K_{\eps}\left(L_{\eps}Y,z\right)H_{\eps}\left(Y\right)\int_{0}^{z}\abs{\psi''\left(Y+\frac{t}{L_{\eps}}\right)}\dt\dz\dY\\
  &\leq \frac{C\norm{\psi''}_{\infty}}{L_{\eps}^{2}}\int_{d-\frac{1}{L_{\eps}}}^{D}\int_{0}^{1}h_{\eps}\left(z\right)H_{\eps}\left(Y\right)\left[\left(L_{\eps}Y+\eps\right)^{-a}\left(z+\eps\right)^{b}+\left(L_{\eps}Y+\eps\right)^{b}\left(z+\eps\right)^{-a}\right]\dz\dY\\
  &\leq \frac{C D^{1-\rho}}{L_{\eps}^{2}}\left[\left(L_{\eps}d-1+\eps\right)^{-a}L_{\eps}^{1+a}+\max\left\{\left(L_{\eps}D+\eps\right)^{b},\left(L_{\eps}d-1+\eps\right)^{b}\right\}L_{\eps}^{1-b}\right]\\
  &=CD^{1-\rho}\left[\frac{\left(d-\frac{1}{L_{\eps}}+\frac{\eps}{L_{\eps}}\right)^{-a}}{L_{\eps}}+\frac{\max\left\{\left(D+\frac{\eps}{L_{\eps}}\right)^{b},\left(d-\frac{1}{L_{\eps}}+\frac{\eps}{L_{\eps}}\right)^{b}\right\}}{L_{\eps}}\right]\\
  &\to 0
 \end{split}
\end{equation*}
as $\eps\to 0$ (and $L_{\eps}\to\infty$). On the other hand (using the symmetry of $K_{\eps}$)
\begin{equation*}
 \begin{split}
  \left(II\right)_{a}=\int_{\frac{1}{L_{\eps}}}^{\infty}H_{\eps}\left(Y\right)\psi'\left(Y\right)\int_{0}^{1}\frac{h_{\eps}\left(z\right)}{L_{\eps}}K_{\eps}\left(L_{\eps}Y,z\right)\dz\dY=\int_{\frac{1}{L_{\eps}}}^{\infty}H_{\eps}\left(Y\right)\psi'\left(Y\right)Q_{\eps}\left(Y\right)\dY.
 \end{split}
\end{equation*}
Thus one obtains $\left(II\right)_{a}\to \int_{0}^{\infty}H\left(Y\right)Q\left(Y\right)\psi'\left(Y\right)\dY$ directly from Lemma~\ref{Lem:convergence:Q}. It remains to show that $\left(III\right)\to \int_{0}^{\infty}H\left(Y\right)\frac{Q\left(Y\right)}{Y}\psi\left(Y\right)\dY$. We first rewrite $(III)$ as
\begin{equation*}
 \begin{split}
  (III)&=\int_{0}^{\infty}\int_{0}^{1}\int_{\frac{y}{L_{\eps}}}^{\frac{y}{L_{\eps}}+Z}\frac{K_{\eps}\left(y,L_{\eps}Z\right)}{L_{\eps}Z}h_{\eps}\left(y\right)H_{\eps}\left(Z\right)\del_{X}\psi\left(X\right)\dX\dy\dZ\\
  &=\int_{0}^{\infty}\int_{0}^{1}\frac{K_{\eps}\left(y,L_{\eps}Z\right)}{L_{\eps}Z}h_{\eps}\left(y\right)H_{\eps}\left(Z\right)\left[\psi\left(Z\right)+\psi\left(Z+\frac{y}{L_{\eps}}\right)-\psi\left(Z\right)-\psi\left(\frac{y}{L_{\eps}}\right)\right]\dy\dZ.
 \end{split}
\end{equation*}
Due to $\supp \psi\subset \left[d,D\right]$ we obtain for $L_{\eps}$ sufficiently large (i.e. $L_{\eps}\geq \frac{2}{d}$) that $\psi\left(\frac{y}{L_{\eps}}\right)=0$ for all $y\in\left[0,1\right]$.
Thus using also the definition of $Q_{\eps}$ we can rewrite $(III)$ as
\begin{equation*}
 \begin{split}
  (III)=\int_{0}^{\infty}\psi\left(Z\right)H_{\eps}\left(Z\right)\frac{Q_{\eps}\left(Z\right)}{Z}\dZ+\int_{0}^{\infty}\int_{0}^{1}\frac{K_{\eps}\left(y,L_{\eps}Z\right)}{L_{\eps}Z}h_{\eps}\left(y\right)H_{\eps}\left(Z\right)\left[\psi\left(Z+\frac{y}{L_{\eps}}\right)-\psi\left(Z\right)\right]\dy\dZ\\
  =:(III)_{a}+(III)_{b}.
 \end{split}
\end{equation*}
 The integral $(III)_{a}$ converges (up to a subsequence) to $\int_{0}^{\infty}\psi\left(Z\right)H\left(Z\right)\frac{Q\left(Z\right)}{Z}\dZ$ according to Lemma~\ref{Lem:convergence:Q}. It thus remains to show that $(III)_{b}$ converges to zero. To see this note that as $\psi$ is smooth and compactly supported we have for $y\in\left[0,1\right]$:
\begin{equation*}
 \begin{split}
  \abs{\psi\left(Z+\frac{y}{L_{\eps}}\right)-\psi\left(Z\right)}\leq C\left(\psi\right)\frac{y}{L_{\eps}}\chi_{\left[d-\frac{1}{L_{\eps}},\infty\right)}\left(Z\right)\leq\frac{C\left(\psi\right)}{L_{\eps}}\chi_{\left[d-\frac{1}{L_{\eps}},\infty\right)}\left(Z\right).
 \end{split}
\end{equation*}
Using this we can estimate
\begin{equation*}
 \begin{split}
  (III)_{b}\leq \frac{C\left(\psi\right)}{L_{\eps}}\int_{d-1/L_{\eps}}^{\infty}H_{\eps}\left(Z\right)\frac{Q_{\eps}\left(Z\right)}{Z}\dZ.
 \end{split}
\end{equation*}
From the estimates on $Q_{\eps}$ in \eqref{eq:est:Qeps} we obtain that the integral on the right hand side is bounded uniformly in $\eps$ and thus for $L_{\eps}\to\infty$ the right hand side converges to zero, concluding the proof.
\end{proof}

\subsubsection{Non-solvability of the limit equation}

\begin{lemma}\label{Lem:non:solv:limit}
 For $\rho\in\left(0,1\right)$ there exists no solution $H$ to \eqref{eq:H:infty} satisfying the lower bound \eqref{Hepslowerbound} and $\int_{0}^{R}H\left(X\right)\dX\leq R^{1-\rho}$ for each $R\geq 0$.
\end{lemma}

\begin{proof}
 Assuming such a solution exists and rewriting \eqref{eq:H:infty} one has
 \begin{equation*}
  \del_{X}\left(\left(X-Q\left(X\right)\right)H\left(X\right)\right)=\left(\left(1-\rho\right)-\frac{Q\left(X\right)}{X}\right)H\left(X\right).
 \end{equation*}
 Defining $F\left(X\right):=\left(X-Q\left(X\right)\right)H\left(X\right)$ this is equivalent to
 \begin{equation*}
  \del_{X}F\left(X\right)=\frac{\left(1-\rho\right)-\frac{Q\left(X\right)}{X}}{X-Q\left(X\right)}F\left(X\right)\quad \text{and thus} \quad F\left(X\right)=C\cdot \exp\left(\int_{A}^{X}\frac{\left(1-\rho\right)-\frac{Q\left(Y\right)}{Y}}{Y-Q\left(Y\right)}\dY\right)
 \end{equation*}
 for some constant $C$. Considering $Q$ one can assume (up to passing to a subsequence of $\eps$ again denoted by $\eps$) that either $\lambda_{\eps}\geq \mu_{\eps}$ or $\mu_{\eps}\geq\lambda_{\eps}$ for all $\eps$, while both cases will be denoted as $\lambda\geq \mu$ or $\mu\geq \lambda$ either. One can then estimate (using the definition of $\lambda_{\eps}$ and $\mu_{\eps}$):
 \begin{equation*}
  \begin{split}
   0\leq Q\left(X\right)&\leq \lim_{\eps\to 0}C_{2}\int_{0}^{1}\frac{h_{\eps}\left(y\right)}{L_{\eps}}\left(\left(y+\eps\right)^{-a}\left(L_{\eps}X+\eps\right)^{b}+\left(y+\eps\right)^{b}\left(L_{\eps}X+\eps\right)^{-a}\right)\dy\\
   &\leq \lim_{\eps\to 0}\frac{C_{2}}{L_{\eps}}\left(L_{\eps}^{1-b}\left(L_{\eps}X+\eps\right)^{b}+L_{\eps}^{a+1}\left(L_{\eps}X+\eps\right)^{-a}\right)=\lim_{\eps\to 0}C_{2}\left(\left(X+\frac{\eps}{L_{\eps}}\right)^{b}+\left(X+\frac{\eps}{L_{\eps}}\right)^{-a}\right)\\
   &=C_{2}\left(X^{b}+X^{-a}\right).
  \end{split}
 \end{equation*}
 On the other hand 
 \begin{equation*}
  \begin{split}
   Q\left(X\right)&\geq \lim_{\eps\to 0}C_{1}\int_{0}^{1}\frac{h_{\eps}\left(y\right)}{L_{\eps}}\left(\left(y+\eps\right)^{-a}\left(L_{\eps}X+\eps\right)^{b}+\left(y+\eps\right)^{b}\left(L_{\eps}X+\eps\right)^{-a}\right)\dy\\
   &\geq \lim_{\eps\to 0} \frac{C_{1}}{L_{\eps}}\begin{cases}
                                                 L_{\eps}^{1-b}\left(L_{\eps}X+\eps\right)^{b} & \mu\geq \lambda\\
                                                 L_{\eps}^{1+a}\left(L_{\eps}X+\eps\right)^{-a} & \lambda \geq \mu
                                                \end{cases}\\
   &=C_{1}\begin{cases}
           X^{b} & \mu\geq\lambda\\
           X^{-a} & \lambda\geq\mu.
          \end{cases}
  \end{split}
 \end{equation*}
 This shows in particular that for sufficiently large $A$ one has $Q\left(X\right)<X$ for all $X\geq A$ and $F$ is well defined for $X\geq A$. We claim now $C>0$. To see this assume $C\leq 0$. Then $F\leq 0$ and as just shown $X-Q\left(X\right)\geq 0$ for $X\geq A$. As $H\left(X\right)=\frac{F\left(X\right)}{X-Q\left(X\right)}$ one has (using $\int_{0}^{R}H\dX\geq \frac{R^{1-\rho}}{2}$ for sufficiently large $R$ due to Proposition~\ref{Prop:Hepslowerbound} and $\int_{0}^{A}H\dX\leq A^{1-\rho}$):
 \begin{equation*}
  \begin{split}
   0\geq \int_{A}^{R}H\left(X\right)\dX=\int_{0}^{R}H\left(X\right)\dX-\int_{0}^{A}H\left(X\right)\dX\geq \frac{1}{2}R^{1-\rho}-A^{1-\rho}=R^{1-\rho}\left(\frac{1}{2}-\left(\frac{A}{R}\right)^{1-\rho}\right)>0
  \end{split}
 \end{equation*}
 for sufficiently large $R$ and thus a contradiction. Therefore we have $C>0$.

 We choose now $X_{0}<A$ such that $Q\left(X_{0}\right)=X_{0}$ and $Q\left(X\right)<X$ for all $X>X_{0}$ which is possible due to the lower and upper bound for $Q$. We get that $\frac{X-Q\left(X\right)}{X-X_{0}}$ is bounded on $\left[X_{0},\infty\right)$ and thus we have $X-Q\left(X\right)\leq K\left(X-X_{0}\right)$ for some $K>0$. Furthermore as $Q\left(Y\right)\sim Y$ for $Y\to X_{0}$ we obtain
 \begin{equation*}
  \begin{split}
   -\left(1-\rho\right)+\frac{Q\left(Y\right)}{Y}=\rho-1+\frac{Q\left(Y\right)}{Y}\geq \rho-\delta>0
  \end{split}
 \end{equation*}
 on $\left[X_{0},\overline{X}\right]$ for some $\overline{X}>X_{0}$ close to $X_{0}$ and $\delta>0$. We then get
 \begin{equation*}
  \int_{\overline{X}}^{A}\frac{\left(1-\rho\right)-\frac{Q\left(Y\right)}{Y}}{Y-Q\left(Y\right)}\dY:=C\left(\overline{X},A\right)<\infty.
 \end{equation*}
 Using the definition of $F$ we can then write $H$ as
 \begin{equation*}
  \begin{split}
   H&=\frac{C}{X-Q\left(X\right)}\exp\left(-\int_{X}^{A}\frac{\left(1-\rho\right)-\frac{Q\left(Y\right)}{Y}}{Y-Q\left(Y\right)}\dY\right)\\
   &\geq \frac{C}{K\left(X-X_{0}\right)}\exp\left(-\int_{X}^{\overline{X}}\frac{\left(1-\rho\right)-\frac{Q\left(Y\right)}{Y}}{Y-Q\left(Y\right)}\dY-\int_{\overline{X}}^{A}\frac{\left(1-\rho\right)-\frac{Q\left(Y\right)}{Y}}{Y-Q\left(Y\right)}\dY\right)\\
   &\geq \frac{C}{K\left(X-X_{0}\right)}\exp\left(-C\left(\overline{X},A\right)\right)\exp\left(\left(\rho-\delta\right)\int_{X}^{\overline{X}}\frac{1}{K\left(Y-X_{0}\right)}\dY\right)\\
   &=\frac{C}{K}\exp\left(-C\left(\overline{X},A\right)\right)\left(\overline{X}-X_{0}\right)^{\frac{\rho-\delta}{K}}\frac{1}{\left(X-X_{0}\right)^{1+\frac{\rho-\delta}{K}}}=C\left(A,\overline{X},X_{0},K\right)\frac{1}{\left(X-X_{0}\right)^{1+\alpha}}
  \end{split}
 \end{equation*}
 with $\alpha=\frac{\rho-\delta}{K}>0$, contradicting the local integrability of $H$.
\end{proof}

This shows that $L_{\eps}$ has to be bounded and thus by scale invariance we obtain from Proposition~\ref{Prop:Hepslowerbound} also the lower bound for $h_{\eps}$, i.e. we have
\begin{proposition}\label{P.hepslowerbound}
 For any $\delta>0$ there exists $R_{\delta}>0$ such that 
 \begin{equation}\label{hepslowerbound}
  \int_0^R h_{\eps}(x)\,dx \geq (1-\delta) R^{1-\rho} \qquad \mbox{ for all } R \geq R_{\delta}.
 \end{equation}
\end{proposition}

\subsection{Exponential decay at the origin}\label{subsec:exp:decay}

We will show in this section that $h_{\eps}$ decays exponentially near zero in an averaged sense, a property that will be crucial when passing to the limit $\eps\to 0$.

\begin{lemma}\label{Lem:expdecay}
 There exist constants $C$ and $c$ independent of $\eps$ such that 
 \begin{equation*}
  \begin{split}
   \int_{0}^{D}h_{\eps}\left(y\right)\dy\leq CD^{1-\rho}\exp\left(-c \left(D+\eps\right)^{-a}\right)
  \end{split}
 \end{equation*}
 for any $D\in\left(0,1\right]$ and all $\eps>0$.
\end{lemma}

\begin{proof}
  Let $\delta=\frac{1}{2}$, then due to Proposition~\ref{P.hepslowerbound} there exists $R_{*}>0$ such that $\int_{0}^{B_2 R_{*}}h_{\eps}\left(z\right)\dz\geq \frac{\left(B_{2} R_{*}\right)^{1-\rho}}{2}$ for any $B_{2}\geq 1$. On the other hand one has $\int_{0}^{B_{1}R_{*}}h_{\eps}\left(z\right)\dz\leq \left(B_{1}R_{*}\right)^{1-\rho}$ for any $B_{1}\geq 0$. Thus one has
  \begin{equation*}
   \begin{split}
    \int_{B_{1}R_{*}}^{B_{2}R_{*}}h_{\eps}\left(z\right)\dz=\int_{0}^{B_{2}R_{*}}h_{\eps}\left(z\right)\dz-\int_{0}^{B_{1}R_{*}}h_{\eps}\left(z\right)\dz\geq \frac{\left(B_{2}R_{*}\right)^{1-\rho}}{2}-\left(B_{1}R_{*}\right)^{1-\rho}\geq 1
   \end{split}
  \end{equation*}
  for sufficiently large $B_{2}$ (depending on $B_{1}$). Thus one can estimate
  \begin{equation*}
   \begin{split}
    &\quad\int_{0}^{R}\int_{R-y}^{\infty}\frac{K_{\eps}\left(y,z\right)}{z}h_{\eps}\left(y\right)h_{\eps}\left(z\right)\dz\dy\\
    &\geq C_{1}\int_{0}^{R}\int_{B_{1}R_{*}}^{B_{2}R_{*}}\frac{\left(y+\eps\right)^{-a}\left(z+\eps\right)^{b}+\left(y+\eps\right)^{b}\left(z+\eps\right)^{-a}}{z}h_{\eps}\left(y\right)h_{\eps}\left(z\right)\dz\dy\\
    &\geq \frac{C_{1}}{\left(R+\eps\right)^{a}}\int_{0}^{R}h_{\eps}\left(y\right)\dy\int_{B_{1}R_{*}}^{B_{2}R_{*}}\frac{\left(z+\eps\right)^{b}}{z}h_{\eps}\left(z\right)\dz\\
    &\geq \frac{C}{\left(R+\eps\right)^{a}}\int_{0}^{R}h_{\eps}\left(y\right)\dy\left(B_{2}R_{*}\right)^{b-1}\int_{B_{1}R_{*}}^{B_{2}R_{*}}h_{\eps}\left(z\right)\dz\geq \frac{C}{\left(R+\eps\right)^{a}}\left(B_{2}R_{*}\right)^{b-1}\int_{0}^{R}h_{\eps}\left(y\right)\dy.
   \end{split}
  \end{equation*}
  Using this and taking $\chi_{\left(-\infty,R\right]}$ (restricted to $\left[0,\infty\right)$) by some approximation argument as test function in the equation $\left(1-\rho\right)h_{\eps}\left(x\right)-\del_{x}\left(xh_{\eps}\left(x\right)\right)+\del_{x}I_{\eps}\left[h_{\eps}\right]\left(x\right)=0$ we obtain
  \begin{equation*}
   \begin{split}
    0&=\left(1-\rho\right)\int_{0}^{R}h_{\eps}\left(x\right)\dx-Rh_{\eps}\left(R\right)+\int_{0}^{R}\int_{R-y}^{\infty}\frac{K_{\eps}\left(y,z\right)}{z}h_{\eps}\left(y\right)h_{\eps}\left(z\right)\dz\dy\\
    &\geq \left(1-\rho\right)\int_{0}^{R}h_{\eps}\left(x\right)\dx-Rh_{\eps}\left(R\right)+\frac{C}{\left(R+\eps\right)^{a}}\left(B_{2}R_{*}\right)^{b-1}\int_{0}^{R}h_{\eps}\left(x\right)\dx.    
   \end{split}
  \end{equation*}
  Thus one has
  \begin{equation*}
   \begin{split}
    \left(1-\rho\right)\int_{0}^{R}h_{\eps}\left(x\right)\dx+\frac{C\left(B_{2}R_{*}\right)^{b-1}}{\left(R+\eps\right)^{a}}\int_{0}^{R}h_{\eps}\left(x\right)\dx\leq Rh_{\eps}\left(R\right)=R\del_{R}\int_{0}^{R}h_{\eps}\left(x\right)\dx
   \end{split}
  \end{equation*}
  or equivalently
  \begin{equation*}
   \begin{split}
    \del_{R}\left(\int_{0}^{R}h_{\eps}\left(x\right)\dx\right)\geq \left(\frac{1-\rho}{R}+\frac{C\left(B_{2}R_{*}\right)^{b-1}}{R\left(R+\eps\right)^{a}}\right)\int_{0}^{R}h_{\eps}\left(x\right)\dx&\geq \left(\frac{1-\rho}{R}+\frac{C\left(B_{2}R_{*}\right)^{b-1}}{\left(R+\eps\right)^{a}}\right)\int_{0}^{R}h_{\eps}\left(x\right)\dx,
   \end{split}
  \end{equation*}
 where we used $\frac{1}{R\left(R+\eps\right)^{a}}\geq \frac{1}{\left(R+\eps\right)^{a}}$ for $R\in\left[D,1\right]$. Integrating this inequality over $\left[D,1\right]$ and using $\int_{0}^{1}h_{\eps}\dx\leq 1$ as well as $\left(1+\eps\right)^{-a}\leq 1$ gives
  \begin{equation*}
   \begin{split}
    \int_{0}^{D}h_{\eps}\left(x\right)\dx
    &\leq \exp\left(\frac{C\left(B_{2}R_{*}\right)^{b-1}}{a}\right) D^{1-\rho}\exp\left(-\frac{C\left(B_{2}R_{*}\right)^{b-1}}{a}\left(D+\eps\right)^{-a}\right).
   \end{split}
  \end{equation*}
\end{proof}

\begin{lemma}\label{Lem:moment:est:eps}
 For $D\leq 1$ and any $\alpha\in \R$ one has the following estimate
 \begin{equation*}
  \begin{split}
   \int_{0}^{D}\left(x+\eps\right)^{\alpha}h_{\eps}\left(x\right)\dx\leq CD^{1-\rho}\left(D+\eps\right)^{\alpha}\exp\left(-c\left(D+\eps\right)^{-a}\right) \quad \text{if } \alpha\geq 0\\
   \int_{0}^{D}\left(x+\eps\right)^{\alpha}h_{\eps}\left(x\right)\dx\leq \tilde{C}D^{1-\rho}\exp\left(-\frac{c}{2}\left(D+\eps\right)^{-a}\right)\quad \text{if } \alpha<0
  \end{split}
 \end{equation*}
\end{lemma}

\begin{proof}
The case $\alpha\geq 0$ follows immediately from Lemma~\ref{Lem:expdecay}. The case $\alpha<0$ follows from Lemma~\ref{Lem:expdecay} using a dyadic decomposition as in Lemma~\ref{Lem:moment:est:stand}.
\end{proof}

As $\left\{h_{\eps}\right\}_{\eps>0}$ is a locally bounded sequence of non-negative Radon measures one can extract a subsequence (again denoted by $\eps$) such that $h_{\eps}\stackrel{*}{\rightharpoonup} h$ in the sense of measures. As a direct consequence of Lemma~\ref{Lem:moment:est:eps} one obtains:
\begin{lemma}\label{Lem:moment:exp}
 For $D\leq 1$ and $\alpha\in\R$ one has
 \begin{equation*}
  \begin{split}
   \int_{0}^{D}x^{\alpha}h\left(x\right)\dx&\leq \tilde{C}D^{1-\rho}\exp\left(-\frac{c}{2}D^{-a}\right) \quad\text{if } \alpha<0\\
   \int_{0}^{D}x^{\alpha}h\left(x\right)\dx&\leq C D^{1+\alpha-\rho}\exp\left(-cD^{-a}\right)\quad \text{if } \alpha\geq 0.
  \end{split}
 \end{equation*}
 \end{lemma}

\begin{proof}
 This follows from Lemma~\ref{Lem:moment:est:eps}.
\end{proof}

As a consequence of Lemma~\ref{Lem:moment:exp} together with Lemma~\ref{Lem:moment:est:stand} we obtain

\begin{corollary}\label{Cor:moment:est}
For any $\alpha\in\R$ and $D>0$ each limit $h$ satisfies
\begin{enumerate}
 \item $\int_{0}^{\infty}x^{\alpha}h\left(x\right)\dx<\infty$ if $\alpha<\rho-1$,
 \item $\int_{0}^{D}x^{\alpha}h\left(x\right)\dx<C\left(D\right)$.
\end{enumerate}
 \end{corollary}

\begin{remark}
 We obtain corresponding results for $h_{\eps}$ and $h$ with $x^{\alpha}$ replaced by $\left(x+\eps\right)^{\alpha}$.
\end{remark}

\subsection{Passing to the limit $\eps\to 0$}\label{sec:limit:eps}

In this section we will finally conclude the proof of Theorem~\ref{T.main} by passing to the limit $\eps\to 0$ in \eqref{eq:self:sim:eps}. Before doing this we first show that $I\left[h\right]$ is locally integrable:

\begin{lemma}
 For $h$ as given above one has $I\left[h\right]\in L_{\mathrm{loc}}^{1}\left(\left[0,\infty\right)\right)$.
\end{lemma}

\begin{proof}
 Let $D>0$. Then one has
 \begin{equation*}
  \begin{split}
   \int_{0}^{D}I\left[h\right]\left(x\right)\dx&=\int_{0}^{D}\int_{0}^{x}\int_{x-y}^{\infty}\frac{K\left(y,z\right)}{z}h\left(y\right)h\left(z\right)\dz\dy\dx\\
   &\leq C\int_{0}^{D}\int_{0}^{D}\int_{0}^{\infty}\left(y^{-a}z^{b-1}+y^{b}z^{-a-1}\right)h\left(y\right)h\left(z\right)\dz\dy\dx\\
   &\leq C\int_{0}^{D}\int_{0}^{D}\left(y^{-a}+y^{b}\right)h\left(y\right)\dy\dx\leq C\left(D\right)
  \end{split}
 \end{equation*}
 where Corollary~\ref{Cor:moment:est} was used. One similarly gets $\int_{N}I\left[h\right]\dx=0$ for bounded null sets $N\subset \left[0,\infty\right)$.
\end{proof}

To show that $h$ is a (weak) self-similar solution it only remains to pass to the limit in the weak form of the equation
\begin{equation*}
 \begin{split}
  \partial_x I_{\eps}[h_{\eps}] = \partial_x \left( x h_{\eps}\right) + (\rho-1) h_{\eps}.
 \end{split}
\end{equation*}
Thus let $\varphi\in C_{c}^{\infty}\left(\left[0,\infty\right)\right)$. Then the weak form reads as
\begin{equation*}
 \begin{split}
  \int_{0}^{\infty}\del_{x}\varphi\left(x\right)\int_{0}^{x}\int_{x-y}^{\infty}\frac{K_{\eps}\left(y,z\right)}{z}h_{\eps}\left(y\right)h_{\eps}\left(z\right)\dz\dy\dx=\int_{0}^{\infty}\del_{x}\varphi\left(x\right)xh_{\eps}\left(x\right)\dx+\left(1-\rho\right)\int_{0}^{\infty}\varphi\left(x\right)h_{\eps}\left(x\right)\dx.
 \end{split}
\end{equation*}
One can easily pass to the limit in the right hand side. To prove Theorem~\ref{T.main} it thus remains to show that one can also take the limit in the left hand side of this equation. This will be done in the following Proposition.

\begin{proposition}\label{Prop:limit:eps}
 For any $\varphi\in C_{c}^{\infty}\left(\left[0,\infty\right)\right)$ one has 
 \begin{equation*}
  \begin{split}
   \int_{0}^{\infty}\del_{x}\varphi\left(x\right)\int_{0}^{x}\int_{x-y}^{\infty}\frac{K_{\eps}\left(y,z\right)}{z}h_{\eps}\left(z\right)h_{\eps}\left(y\right)\dz\dy\dx\longrightarrow \int_{0}^{\infty}\del_{x}\varphi\left(x\right)\int_{0}^{x}\int_{x-y}^{\infty}\frac{K\left(y,z\right)}{z}h\left(z\right)h\left(y\right)\dz\dy\dx
  \end{split}
 \end{equation*}
as $\eps\to 0$.
\end{proposition}

\begin{proof}
 Taking the difference of the two integrals and rewriting one obtains
 \begin{equation*}
  \begin{split}
   &\quad \abs{\int_{0}^{\infty}\del_{x}\varphi\left(x\right)\left(\int_{0}^{x}\int_{x-y}^{\infty}\frac{K\left(y,z\right)}{z}h\left(y\right)h\left(z\right)-\frac{K_{\eps}\left(y,z\right)}{z}h_{\eps}\left(y\right)h_{\eps}\left(z\right)\dy\dz\right)\dx}\\
   &\leq \abs{\int_{0}^{\infty}\del_{x}\varphi\left(x\right)\left(\int_{0}^{x}\int_{x-y}^{\infty}\frac{K\left(y,z\right)-K_{\eps}\left(y,z\right)}{z}h\left(y\right)h\left(z\right)\dz\dy\right)\dx}\\
   &\quad+\abs{\int_{0}^{\infty}\del_{x}\varphi\left(x\right)\left(\int_{0}^{x}\int_{x-y}^{\infty}\frac{K_{\eps}\left(y,z\right)}{z}h\left(y\right)\left(h\left(z\right)-h_{\eps}\left(z\right)\right)\dz\dy\right)\dx}\\
   &\quad +\abs{\int_{0}^{\infty}\del_{x}\varphi\left(x\right)\left(\int_{0}^{x}\int_{x-y}^{\infty}\frac{K_{\eps}\left(y,z\right)}{z}h_{\eps}\left(z\right)\left(h\left(y\right)-h_{\eps}\left(y\right)\right)\dz\dy\right)\dx}=:\left(I\right)+\left(II\right)+\left(III\right)
  \end{split}
 \end{equation*} 
 We estimate the three terms separately and take $D>0$ such that $\supp \varphi\subset \left[0,D\right]$. Then due to Lebesgue's Theorem (using also Corollary~\ref{Cor:moment:est} and Lemma~\ref{Lem:moment:est:stand}) we obtain
 \begin{equation*}
  \begin{split}
   \left(I\right)&\leq \int_{0}^{\infty}\abs{\del_{x}\varphi\left(x\right)}\left(\int_{0}^{x}\int_{x-y}^{\infty}\frac{\abs{K\left(y,z\right)-K_{\eps}\left(y,z\right)}}{z}h\left(y\right)h\left(z\right)\dz\dy\right)\dx\to 0\quad \text{as } \eps\to 0. 
  \end{split}
 \end{equation*}
 To estimate the other two terms we will need some cutoff functions. Let $M,N\in\N$ and $\zeta_{1}^{N},\zeta_{2}^{N},\xi_{1}^{M},\xi_{2}^{M}\in C^{\infty}\left(\left[0,\infty\right)\right)$ such that 
\begin{equation*}
 \begin{split}
   \zeta_{1}^{N}=0 \text{ on } \left[0,\frac{1}{N}\right]\cup\left[N+1,\infty\right),\quad  \zeta_{1}^{N}=1 \text{ on } \left[\frac{2}{N},N\right], \quad  0\leq \zeta_{1}^{N}\leq 1, \quad  \zeta_{2}^{N}:=1-\zeta_{1}^{N},\\
   \xi_{1}^{M}=0 \text{ on } \left[0,\frac{1}{M}\right], \quad \xi_{1}^{M}=1 \text{ on } \left[\frac{2}{M},\infty\right), \quad 0\leq \xi_{1}^{M}\leq 1, \quad  \xi_{2}^{M}:=1-\xi_{1}^{M}. 
 \end{split}
\end{equation*}
Defining $K_{\eps}^{i,N}\left(y,z\right):=K_{\eps}\left(y,z\right)\cdot \zeta_{i}^{N}\left(z\right)$ for $i=1,2$ one obtains using also Fubini's Theorem:
\begin{equation*}
 \begin{split}
  \left(II\right)
  &\leq \abs{\int_{0}^{\infty}\int_{0}^{\infty}\frac{K_{\eps}^{1,N}\left(y,z\right)}{z}h\left(y\right)\left(h\left(z\right)-h_{\eps}\left(z\right)\right)\int_{y}^{y+z}\del_{x}\varphi\left(x\right)\dx\dy\dz}\\
  &\quad +\abs{\int_{0}^{\infty}\int_{0}^{\infty}\frac{K_{\eps}^{2,N}\left(y,z\right)}{z}h\left(y\right)\left(h\left(z\right)-h_{\eps}\left(z\right)\right)\int_{y}^{y+z}\del_{x}\varphi\left(x\right)\dx\dy\dz}=:\left(II\right)_{a}+\left(II\right)_{b}
 \end{split}
\end{equation*}
We consider again terms separately and without loss of generality we assume $\eps<1$. Then using Corollary~\ref{Cor:moment:est} and Lemma~\ref{Lem:moment:est:stand} we obtain
\begin{equation}\label{eq:weak:strong:0}
 \begin{split}
  \left(II\right)_{b}
  &\leq C\norm{\del_{x}\varphi}_{\infty}\int_{0}^{\frac{2}{N}}\int_{0}^{D}\left[\left(y+\eps\right)^{-a}\left(z+\eps\right)^{b}+\left(y+\eps\right)^{b}\left(z+\eps\right)^{-a}\right]h\left(y\right)\left(h\left(z\right)+h_{\eps}\left(z\right)\right)\dy\dz\\
  &\quad +C\norm{\varphi}_{\infty}\int_{N}^{\infty}\int_{0}^{D}\frac{\left(y+\eps\right)^{-a}\left(z+\eps\right)^{b}+\left(y+\eps\right)^{b}\left(z+\eps\right)^{-a}}{z}h\left(y\right)\left(h\left(z\right)+h_{\eps}\left(z\right)\right)\dy\dz\\
  &\leq \norm{\del_{x}\varphi}_{\infty}C\left(D\right)\int_{0}^{\frac{2}{N}}\left(\left(z+\eps\right)^{b}+\left(z+\eps\right)^{-a}\right)\left(h\left(z\right)+h_{\eps}\left(z\right)\right)\dz\\
  &\qquad +C\left(D\right)\norm{\varphi}_{\infty}\int_{N}^{\infty}\left(2^{\tilde{b}}z^{b-1}+z^{-a-1}\right)\left(h\left(z\right)+h_{\eps}\left(z\right)\right)\dz\\
  &\leq \norm{\del_{x}\varphi}_{\infty}C\left(D\right)\left[\frac{1}{N^{1-\rho}}\left(\frac{2}{N}+\eps\right)^{\tilde{b}}+\frac{1}{N^{1-\rho}}\right]+C\left(D\right)\norm{\varphi}_{\infty}\left[N^{b-\rho}+N^{-a-\rho}\right]\longrightarrow 0,
 \end{split}
\end{equation}
for $N\to\infty$. Furthermore one has
\begin{equation}\label{eq:weak:strong:1}
 \begin{split}
  \left(II\right)_{a}
  &=\abs{\int_{0}^{\infty}\left(h\left(z\right)-h_{\eps}\left(z\right)\right)\psi_{\eps}^{N}\left(z\right)\dz}
 \end{split}
\end{equation}
with $\psi_{\eps}^{N}\left(z\right):=\int_{0}^{D}\frac{K_{\eps}^{1,N}\left(y,z\right)}{z}h\left(y\right)\left[\varphi\left(y+z\right)-\varphi\left(y\right)\right]\dy$. We claim that $\psi_{\eps}^{N}\to \psi^{N}$ strongly in $C\left(\left[0,\infty\right)\right)$ with $\psi^{N}\left(z\right):=\int_{0}^{D}\frac{K\left(y,z\right)}{z}h\left(y\right)\zeta_{1}^{N}\left(z\right)\left[\varphi\left(y+z\right)-\varphi\left(y\right)\right]\dy$. Note that by construction we have $\supp\psi_{\eps}^{N}\subset \left[\frac{1}{N},N+1\right]$ for all $\eps>0$. To show (uniform) convergence we have to use a cutoff also in $y$, i.e. one can estimate
\begin{equation*}
 \begin{split}
  \abs{\psi^{N}\left(z\right)-\psi_{\eps}^{N}\left(z\right)}&\leq \abs{\int_{0}^{D}\frac{K\left(y,z\right)-K_{\eps}\left(y,z\right)}{z}\zeta_{1}^{N}\left(z\right)\xi_{1}^{M}\left(y\right)h\left(y\right)\left[\varphi\left(y+z\right)-\varphi\left(y\right)\right]\dy}\\
  &\quad +\abs{\int_{0}^{D}\frac{K\left(y,z\right)-K_{\eps}\left(y,z\right)}{z}\zeta_{1}^{N}\left(z\right)\xi_{2}^{M}\left(y\right)h\left(y\right)\left[\varphi\left(y+z\right)-\varphi\left(y\right)\right]\dy}\\
  &=:\left(II\right)_{a,1}+\left(II\right)_{a,2}.
 \end{split}
\end{equation*}
Using similar arguments as in \eqref{eq:weak:strong:0} we get
\begin{equation}\label{eq:strong:conv:1:1}
 \begin{split}
  \left(II\right)_{a,2}
  &\leq C\left(N,\varphi\right)\left[\frac{1}{M^{1+\tilde{b}-\rho}}+\frac{1}{M^{1-\rho}}\right]\longrightarrow 0
 \end{split}
\end{equation}
for $M\to \infty$ and $N$ fixed. As $K$ is continuous on $\left[\frac{1}{M},D\right]\times\left[\frac{1}{N},N+1\right]$ for $M,N\in\N$ fixed, one has $K_{\eps}\to K$ uniformly on $\left[\frac{1}{M},D\right]\times\left[\frac{1}{N},N+1\right]$ for $\eps\to 0$. Thus we get $(II)_{a,1}\to 0$ for $\eps\to 0$ (with $M,N$ fixed). Together with \eqref{eq:strong:conv:1:1} this shows that $\psi_{\eps}^{N}\to \psi^{N}$ strongly. Thus one can pass to the limit in \eqref{eq:weak:strong:1} to obtain together with \eqref{eq:weak:strong:0}: $(II)\to 0$ as $\eps\to 0$.

In a similar way we can show that $(III)\to 0$ for $\eps\to 0$.

 \end{proof}

\section*{Acknowledgements} 
The authors acknowledge support through the CRC 1060 \textit{The mathematics of emergent effects} at the University of Bonn that is funded through the German Science Foundation (DFG).

\appendix

\section{Moment estimates}

\begin{lemma}\label{Lem:moment:est:stand}
 Let $h\in \mathcal{X}_{\rho}$ and $\alpha\in\R$. Then one has the following estimates
 \begin{enumerate}
  \item $\int_{0}^{D}x^{\alpha}h\left(x\right)\dx\leq C \norm{h}D^{1-\rho+\alpha}$ for all $D>0$ if $\rho-1<\alpha$,
  \item $\int_{D}^{\infty}x^{\alpha}h\left(x\right)\dx\leq C\norm{h}D^{1-\rho+\alpha}$ for all $D>0$ if $\alpha<\rho-1$,
 \end{enumerate}
where $\norm{h}$ is defined in~\eqref{eq:S1E3}.
\end{lemma}

\begin{proof}
 ~\begin{enumerate}
   \item The case $\alpha\geq 0$ is clear by definition of $\mathcal{X}_{\rho}$. For $\alpha\in\left(\rho-1,0\right)$ one has, using a dyadic decomposition, that
         \begin{equation*}
          \begin{split}
           &\quad\int_{0}^{D}x^{\alpha}h\left(x\right)\dx=\sum_{n=0}^{\infty}\int_{2^{-\left(n+1\right)}D}^{2^{-n}D}x^{\alpha}h\left(x\right)\dx\leq \sum_{n=0}^{\infty}2^{-\alpha\left(n+1\right)}D^{\alpha}\int_{2^{-\left(n+1\right)}D}^{2^{-n}D}h\left(x\right)\dx\\
           &\leq \norm{h}\sum_{n=0}^{\infty}2^{-\alpha\left(n+1\right)}D^{\alpha}\left(2^{-n}D\right)^{1-\rho}=2^{-\alpha}\norm{h}D^{1+\alpha-\rho}\sum_{n=0}^{\infty}\left(2^{1+\alpha-\rho}\right)^{-n}=C\left(\alpha,\rho\right)\norm{h}D^{1+\alpha-\rho}.
          \end{split}
         \end{equation*}
   \item This follows similarly using again a dyadic decomposition.
  \end{enumerate}
\end{proof}

\section{Dual problems}

\subsection{Existence results}\label{Sec:existence:results}

In this section we show the existence of solutions to some dual problems arising in the proof of the lower bounds. Throughout this section we will use the following notation: $\Mfin$ will denote the space of finite measures, $\Mfin_{+}$ is the space on non-negative finite measures. Furthermore $C_{b}^{n}$ denotes the space of bounded $n$-times differentiable functions with bounded derivatives. Let $\omega\in\left(0,1\right)$, $A\in\R$ and consider the equation 
\begin{equation}\label{eq:standard:dual:problem}
 \del_{t}f\left(x,t\right)-\const\int_{0}^{\infty}\frac{1}{y^{1+\omega}}\left[f\left(x+y\right)-f\left(x\right)\right]\dy=0
\end{equation}
together with initial value $f\left(x,0\right)=\delta\left(\cdot-A\right)$.

\begin{proposition}\label{Prop:ex:stand:dual:distr}
 There exists a (weak) solution $f\in C\left(\left[0,T\right],\Mfin_{+}\right)$ of \eqref{eq:standard:dual:problem} with initial value $f_{0}=\delta\left(\cdot-A\right)$. Furthermore this $f$ satisfies $\supp f\left(\cdot,t\right)\subset \left(-\infty,A\right]$ and $\int_{\R}f\left(\cdot,t\right)\dx=1$ for all $t\in\left[0,T\right]$.
\end{proposition}

\begin{proof}[Proof (Sketch)]
 First we consider the regularized equation
 \begin{equation}\label{eq:reg:weak:sol}
  \begin{split}
   \del_{t}f\left(x,t\right)&=\const\int_{0}^{\infty}\frac{1}{y^{1+\omega}+\nu}\left[f\left(x+y,t\right)-f\left(x,t\right)\right]\dy\\
   f\left(\cdot,0\right)&=\delta\left(\cdot-A\right)
  \end{split}
 \end{equation}
with $\nu>0$. In the second step we will pass to the limit $\nu\to 0$. We can reformulate \eqref{eq:reg:weak:sol} as the following fixed-point problem:
\begin{equation}\label{eq:reg:fix}
 \begin{split}
  f^{\nu}\left(x,t\right)=\delta\left(x-A\right)\ee^{-\const\int_{0}^{\infty}\frac{1}{y^{1+\omega}+\nu}\dy}+\int_{0}^{t}\ee^{-\left(t-s\right)\int_{0}^{\infty}\frac{1}{y^{1+\omega}+\nu}\dy}\int_{0}^{\infty}\frac{1}{y^{1+\omega}+\nu}f\left(x+y\right)\dy\ds.
 \end{split}
\end{equation}
It is straightforward, applying the contraction mapping theorem, to obtain a solution $f\in C\left(\left[0,T\right],\Mfin_{+}\right)$ for any $T>0$. Furthermore, one obtains $\int_{\R}f^{\nu}\left(x,t\right)\dx=1$ for all $t>0$ and $\nu>0$ (by integrating the equation, see below). In addition $f^{\nu}$ satisfies equation \eqref{eq:reg:weak:sol} in weak form, i.e.
\begin{equation}\label{eq:reg:weak:form}
 \begin{split}
  \int_{\R}f^{\nu}\left(x,t\right)\psi\left(x\right)\dx=\psi\left(A\right)+\int_{0}^{t}\int_{\R}\int_{0}^{\infty}\frac{1}{y^{1+\omega}+\nu}f^{\nu}\left(x,s\right)\left[\psi\left(x-y\right)-\psi\left(x\right)\right]\dy\dx\ds
 \end{split}
\end{equation}
for all $\psi\in C_{b}\left(\R\right)$ and for $0<\tilde{\omega}<\omega$ taking $\psi\left(x\right)=\abs{x}^{\tilde{\omega}}$ and using $\abs{\abs{x-y}^{\tilde{\omega}}-\abs{x}^{\tilde{\omega}}}\leq \abs{y}^{\tilde{\omega}}$ we obtain (by approximation)
\begin{equation*}
 \begin{split}
  \int_{\R}f^{\nu}\left(x,t\right)\abs{x}^{\tilde{\omega}}\dx&\leq \abs{A}^{\tilde{\omega}}+\int_{0}^{t}\int_{\R}\int_{0}^{\infty}\frac{\abs{\abs{x-y}^{\tilde{\omega}}-\abs{x}^{\tilde{\omega}}}}{y^{1+\omega}+\nu}f^{\nu}\left(x,s\right)\dy\dx\ds\\
  &\leq \abs{A}^{\tilde{\omega}}+\int_{0}^{t}\int_{\R}\int_{0}^{\infty}\frac{\abs{y}^{\tilde{\omega}}}{y^{1+\omega}+\nu}f^{\nu}\left(x,s\right)\dy\dx\ds\leq C\left(T,\omega,\tilde{\omega},A\right).
 \end{split}
\end{equation*}
Thus $\int_{\R}\abs{x}^{\tilde{\omega}}f^{\nu}\left(x,t\right)\dx$ is uniformly bounded (i.e. independent of $\nu$ and $t$).

Using this and that $\left\{f^{\nu}\right\}_{\nu>0}$ is uniformly bounded by $1$, we can extract a subsequence $\left\{f^{\nu_{n}}\right\}_{n\in\N}$ (denoted in the following as $\left\{f^{n}\right\}_{n\in\N}$) such that $f^{n}\left(\cdot,t_{k}\right)$ converges in the sense of measures to some $f\left(\cdot,t_{k}\right)$ for all $k\in\N$, where $\left\{t_{k}\right\}_{k\in\N}=\left[0,T\right]\cap \Q$.

We next show that $f^{n}$ is equicontinuous in $t$ as a distribution, i.e. from \eqref{eq:reg:weak:form} we obtain for any $\psi\in C_{c}^{1}\left(\R\right)$:
\begin{equation*}
 \begin{split}
  &\quad \abs{\int_{\R}\left(f^{n}\left(x,t\right)-f^{n}\left(x,s\right)\right)\psi\left(x\right)\dx}=\abs{\int_{s}^{t}\int_{\R}f^{n}\left(x,r\right)\int_{0}^{\infty}\frac{1}{y^{1+\omega}+\nu}\left[\psi\left(x-y\right)-\psi\left(x\right)\right]\dy\dx\dd r}\\
  &\leq \int_{s}^{t}\int_{\R}f^{n}\left(x,r\right)\left[\int_{0}^{1}\frac{\norm{\psi'}_{L^{\infty}}y}{y^{1+\omega}+\nu}\dy+\int_{1}^{\infty}\frac{2\norm{\psi}_{L^{\infty}}}{y^{1+\omega}+\nu}\dy\right]\dx\dd r\leq C\left(\psi\right)\abs{t-s},
 \end{split}
\end{equation*}
where $C\left(\psi\right)$ is a constant independent of $\nu$ but depending on $\psi$ and $\psi'$. Using the equicontinuity of $f^{n}$ (as a distribution) one can show that $f^{n}$ converges to some limit $f$ (in the sense of distributions) for all $t\in\left[0,T\right]$. 

Using furthermore the uniform boundedness of $\int_{\R}\abs{x}^{\tilde{\omega}}f^{n}\left(x,t\right)\dx$ one can show that $f^{n}$ converges already in the sense of measures by approximating and cutting the test function for large values of $\abs{x}$.

Using similar arguments we can also show that for the limit $f^{n}\rightharpoonup f$ we have $f\in C\left(\left[0,T\right],\Mfin_{+}\right)$ and taking the limit $n\to \infty$ in \eqref{eq:reg:weak:form}, $f$ satisfies
\begin{equation}\label{eq:weak:sol:limit}
 \begin{split}
   \int_{\R}f\left(x,t\right)\psi\left(x\right)\dx=\psi\left(A\right)+\int_{0}^{t}\int_{\R}\int_{0}^{\infty}\frac{1}{y^{1+\omega}}f\left(x,s\right)\left[\psi\left(x-y\right)-\psi\left(x\right)\right]\dy\dx\ds
 \end{split}
\end{equation}
for each $\psi\in C_{b}^{1}\left(\R\right)$ and all $t\in\left[0,T\right]$.

From the construction of $f$ using the contraction mapping principle we immediately get $\supp f\left(\cdot,t\right)\subset \left(-\infty,A\right]$ for all $t\in\left[0,T\right]$. To see $\int_{\R}f\left(\cdot,t\right)\dx=1$ for all $t\in\left[0,T\right]$ we integrate equation~\eqref{eq:standard:dual:problem} over $\R$ and use Fubini's theorem to obtain $\del_{t}\int_{\R}G\left(\cdot,t\right)\dx=0$. Thus together with the initial condition the claim follows. 
\end{proof}

\begin{remark}
 The analogous result holds true if $f_{0}=-\delta\left(\cdot-A\right)$.
\end{remark}

As a direct consequence of Proposition~\ref{Prop:ex:stand:dual:distr} we also obtain smooth solutions for smoothed initial data. Therefore for $\kappa>0$ we denote in the following by $\varphi_{\kappa}$ a non-negative, symmetric standard mollifier with $\supp \varphi_{\kappa}\subset\left[-\kappa, \kappa\right]$.

\begin{proposition}\label{Prop:ex:stand:dual:smooth:meas}
 Let $f_{0}:=\delta\left(\cdot-A\right)$. Then there exists a solution $f\in C^{1}\left(\left[0,T\right],C^{\infty}\left(\R\right)\right)$ to \eqref{eq:standard:dual:problem} with initial datum $f_{0}\ast\varphi_{\kappa}=\varphi_{\kappa}\left(\cdot-A\right)$. 
\end{proposition}

\begin{proof}
 This follows directly by convolution in $x$ from Proposition~\ref{Prop:ex:stand:dual:distr}.
\end{proof}

\begin{proposition}\label{Prop:ex:stand:dual:smooth:func}
 There exists a strong solution $f\in C^{1}\left(\left[0,T\right],C^{\infty}\left(\R\right)\right)$ to \eqref{eq:standard:dual:problem} with initial datum $f_{0}:=\chi_{\left(-\infty,A\right]}\ast \varphi_{\kappa}$.
\end{proposition}

\begin{proof}
 Let $G$ be the solution given by Proposition~\ref{Prop:ex:stand:dual:smooth:meas} for $G_{0}:=\delta\left(\cdot-A\right)\ast \varphi_{\kappa}$. Then $f\left(x,t\right):=\int_{x}^{\infty}G\left(y,t\right)\dy$ solves \eqref{eq:standard:dual:problem} with the desired initial condition.
\end{proof}

In the same way as in the proofs of Proposition~\ref{Prop:ex:stand:dual:smooth:func} and Proposition~\ref{Prop:ex:stand:dual:smooth:func} we obtain the following existence result:

\begin{proposition}\label{Prop:ex:dual:eps}
 Let $\eps>0$, $L>0$ and $\lambda_{1},\lambda_{2}>0$ be two constants (depending on some parameters). Then there exists a weak solution $G\in C\left(\left[0,T\right],\Mfin_{+}\right)$ and a strong solution $W\in C\left(\left[0,T\right],C^{\infty}\right)$ of the equation
 \begin{equation}\label{eq:dual:eps}
  \del_{t} W\left(\xi,t\right)-\int_{0}^{1}\frac{h_{\eps}\left(z\right)}{z}\left[\lambda_{1}\left(z+\eps\right)^{-a}+\lambda_{2}\left(z+\eps\right)^{b}\right]\left[W\left(\xi+\frac{z}{L},t\right)-W\left(\xi,t\right)\right]\dz=0
 \end{equation}
together with initial condition $G\left(\cdot,0\right)=\delta\left(\cdot-A\right)$ and $W\left(\cdot,0\right)=\chi_{\left(-\infty,A\right]}\ast\varphi_{\kappa}$.
\end{proposition}

\begin{remark}
 The measure $G$ has the same properties as the measure $f$ in Proposition~\ref{Prop:ex:stand:dual:distr}.
\end{remark}

\begin{remark}
 By convolution we also obtain a strong solution $G\in C\left(\left[0,T\right],C^{\infty}\right)$ of \eqref{eq:dual:eps} with initial condition $G\left(\cdot,0\right)=\delta\left(\cdot-A\right)\ast\varphi_{\kappa}$.
\end{remark}

For further use we denote the integral kernels occurring in Proposition~\ref{Prop:ex:stand:dual:distr} and Proposition~\ref{Prop:ex:dual:eps} by
\begin{equation}\label{eq:kernel:Neps}
 N_{\omega}\left(z\right):=z^{-1-\omega}\quad \text{and} \quad  N_{\eps}\left(z\right):=\frac{h_{\eps}\left(z\right)}{z}\left[\lambda_{1}\left(z+\eps\right)^{-a}+\lambda_{2}\left(z+\eps\right)^{b}\right].
\end{equation}

\begin{proposition}\label{Prop:ex:dual:sum}
 Let $n\in\N$, $R\in \R$ and $N_{i}\colon \left(0,\infty\right)\to\R_{\geq 0}$ either of the form $N_{\omega_{i}}$ for some $\omega_{i}\in \left(0,1\right)$ or $N_{\eps}$ given by \eqref{eq:kernel:Neps} (and then continued by $0$ to $\left(0,\infty\right)$) for $i=1,\ldots n$. Let $N:=\sum_{i=1}^{n} N_{i}$. Then there exists a solution $f\in C^{1}\left(\left[0,T\right],C^{\infty}\left(\R\right)\right)$ to the equation
  \begin{equation}\label{eq:dual:convolution}
   \begin{split}
    \del_{t}f\left(x,t\right)=\int_{0}^{\infty}N\left(z\right)\left[f\left(x+z\right)-f\left(x\right)\right]\dz
   \end{split}
  \end{equation}
 either with initial datum $f_{0}=\chi_{\left(-\infty,R\right]}\ast^{n} \varphi_{\kappa}$ or $f_{0}=\delta\left(\cdot-R\right)\ast^{n}\varphi_{\kappa}$, where $\ast^{n}$ denotes the $n$-fold convolution with $\varphi_{\kappa}$.
\end{proposition}

\begin{proof}
 It suffices to consider the case $n=2$ (otherwise argue by induction). Then by Proposition~\ref{Prop:ex:stand:dual:smooth:meas} and Proposition~\ref{Prop:ex:stand:dual:smooth:func} there exist solutions $f^{i}$ to equation~\eqref{eq:dual:convolution} with $N$ replaced by $N_{i}$ and initial datum $f^{1}_{0}=\delta\left(\cdot\right)\ast\varphi_{\kappa}$ and $f^{2}_{0}=\chi_{\left(-\infty,R\right]}\ast\varphi_{\kappa}$ (or $f^{2}_{0}=\delta\left(\cdot-R\right)\ast\varphi_{\kappa}$). A straightforward computation shows that the convolution $f:=f^{1}\ast f^{2}$ satisfies~\eqref{eq:dual:convolution} together with the correct initial condition.
\end{proof}

\begin{remark}\label{Rem:properies}
 Let $G_{\kappa}$ and $f_{\kappa}$ be the solutions given by Proposition~\ref{Prop:ex:dual:sum} with initial condition $G_{\kappa}\left(\cdot,0\right)=\delta\left(\cdot-A\right)\ast^{n} \varphi_{\kappa}$ and $f\left(\cdot,0\right)=\chi_{\left(-\infty,A\right]}\ast^{n}\varphi_{\kappa}$. Then from the construction in the proof of Proposition~\ref{Prop:ex:dual:sum} and Proposition~\ref{Prop:ex:stand:dual:distr} we obtain:
 \begin{enumerate}
  \item $G_{\kappa}\geq 0$ on $\R$ (in the sense of measures) and $0\leq f_{\kappa}\leq 1$ for all $t\in\left[0,T\right]$,
  \item $\supp G_{\kappa}\left(\cdot,t\right), \supp f_{\kappa}\left(\cdot,t\right)\subset \left(-\infty,A+n\kappa\right]$ for all $t\in\left[0,T\right]$,
  \item $\int_{\R}G\left(\cdot,t\right)\dx=1$ for all $t\in\left[0,T\right]$,
  \item $f_{\kappa}$ is non-increasing.
 \end{enumerate}
\end{remark}

\subsection{Integral estimates for subsolutions}

In this section we will always assume that the integral kernel $N$ is given as the sum of kernels of the form $N_{\omega_{i}}$ or $N_{\eps}$ and we will prove several properties and estimates that are frequently used. We now prove some integral estimates.

\begin{lemma}\label{Lem:der:int:est}
 Let $\omega\in \left(0,1\right)$ and $G$ the solution of 
 \begin{equation}\label{eq:der:int:est}
  \begin{split}
   \del_{t}G\left(x,t\right)&=P\int_{0}^{\infty}N_{\omega}\left(z\right)\left[G\left(x+z\right)-G\left(x\right)\right]\dz\\
   G\left(\cdot,0\right)&=\delta\left(\cdot-A\right)\ast\varphi_{\kappa}=\varphi_{\kappa}\left(x-A\right)
  \end{split}
 \end{equation}
 given by Proposition~\ref{Prop:ex:stand:dual:smooth:meas}, where $P$ is a constant. Then for any $\mu \in \left(0,1\right)$ one has
 \begin{enumerate}
  \item $\int_{-\infty}^{A-D}G\left(x,t\right)\dx\leq C\left(\frac{\kappa}{D}\right)^{\mu}+ C\frac{P t}{D^{\omega}}$ for all $D>0$ and
  \item $\int_{A-2}^{A}\abs{x-A}G\left(x,t\right)\dx\leq C_{\mu}\kappa^{\mu}+C_{\omega} Pt$.
 \end{enumerate}
\end{lemma}

\begin{proof}
 By shifting with $A$ we can assume $A=0$. Let $Z>0$. Then testing equation~\eqref{eq:der:int:est} with $\ee^{Z\left(x-\kappa\right)}$ (note that this is possible as $\supp G\subset \left(-\infty,\kappa\right]$) one obtains
 \begin{equation*}
  \begin{split}
   \del_{t}\int_{\R}G\left(x,t\right)\ee^{Z\left(x-\kappa\right)}\dx&=P\int_{\R}\int_{0}^{\infty}N_{\omega}\left(y\right)\left[G\left(x+y\right)-G\left(x\right)\right]\ee^{Z\left(x-\kappa\right)}\dy\dx\\
   &=P\int_{0}^{\infty}N_{\omega}\left(y\right)\left(\ee^{-Zy}-1\right)\dy\int_{\R}G\left(x,t\right)\ee^{Z\left(x-\kappa\right)}\dy=:M_{\omega}\left(Z\right)\int_{\R}G\left(x,t\right)\ee^{Z\left(x-\kappa\right)}\dx.
  \end{split}
 \end{equation*}
 Furthermore
 \begin{equation*}
  \begin{split}
   \int_{\R}G_{\kappa}\left(x,0\right)\ee^{Z\left(x-\kappa\right)}\dx=\int_{\R}\varphi_{\kappa}\left(x\right)\ee^{Z\left(x-\kappa\right)}\dx.
  \end{split}
 \end{equation*}
 Thus we obtain $\int_{\R}G\left(x,t\right)\ee^{Z\left(x-\kappa\right)}\dx=\int_{\R}\varphi_{\kappa}\left(x\right)\ee^{Z\left(x-\kappa\right)}\dx\exp\left(-t\abs{M_{\omega}\left(Z\right)}\right)$. Estimating $M_{\omega}\left(Z\right)$ we obtain
 \begin{equation*}
  \begin{split}
   \abs{M_{\omega}\left(Z\right)}&\leq P\int_{0}^{\infty}\frac{1-\ee^{-Zy}}{y^{1+\omega}}\dy=-\frac{P}{\omega}\int_{0}^{\infty}\left(1-\ee^{-Zy}\right)\frac{\del}{\del y}\left(y^{-\omega}\right)\dy\\
   &=\frac{PZ}{\omega}\int_{0}^{\infty}\frac{\ee^{-Zy}}{y^{\omega}}\dy=\frac{PZ^{\omega}}{\omega}\int_{0}^{\infty}y^{-\omega}\ee^{-y}\dy=P\frac{\Gamma\left(1-\omega\right)}{\omega}Z^{\omega}\\
   &=CPZ^{\omega}.
  \end{split}
 \end{equation*}
 Using that $G=0$ on $\left(\kappa,\infty\right)$ we get 
 \begin{equation*}
  \begin{split}
   &\quad\int_{-\infty}^{\kappa}G\left(x,t\right)\left(1-\ee^{Z\left(x-\kappa\right)}\right)\dx\\
   &=\int_{\R}G\left(x,t\right)\dx-\int_{\R}G\left(x,t\right)\ee^{Z\left(x-\kappa\right)}\dx=1-\int_{\R}\varphi_{\kappa}\left(x\right)\ee^{Z\left(x-\kappa\right)}\dx\exp\left(-t\abs{M_{\omega}\left(Z\right)}\right)\\
   &\leq \left[\left(1-\int_{\R}\varphi_{\kappa}\left(x\right)\ee^{Z\left(x-\kappa\right)}\dx\right)+\int_{\R}\varphi_{\kappa}\left(x\right)\ee^{Z\left(x-\kappa\right)}\dx\abs{M_{\omega}\left(Z\right)}t\right].
  \end{split}
 \end{equation*}
 As $\supp \varphi\subset\left[-\kappa,\kappa\right]$ we can estimate $\ee^{-2Z\kappa}\leq \int_{\R}\varphi_{\kappa}\left(x\right)\ee^{Z\left(x-\kappa\right)}\dx\leq 1$. Then choosing $Z=\frac{1}{D}$ and using also the estimate for $M_{\omega}$ we obtain
 \begin{equation*}
  \begin{split}
     &\quad \int_{-\infty}^{-D}G\left(x,t\right)\dx\leq \int_{-\infty}^{-D}G\left(x,t\right)\frac{1-\ee^{\frac{x-\kappa}{D}}}{1-\ee^{-1-\frac{\kappa}{D}}}\dx\leq \int_{-\infty}^{\kappa}\left(\cdots\right)\dx\\
     &\leq\frac{1}{1-\ee^{-1-\frac{\kappa}{D}}} \left[\left(1-\ee^{-\frac{2\kappa}{D}}\right)+CPt\frac{1}{D^{\omega}}\right]\leq C\left(\frac{\kappa}{D}\right)^{\mu}+C\frac{Pt}{D^{\omega}}.
  \end{split}
 \end{equation*} 
 To prove the second part we use a dyadic decomposition and the estimate from the first part to obtain
 \begin{equation*}
  \begin{split}
   \int_{-2}^{0}\abs{x}G_{\kappa}\left(x,t\right)\dx&=\sum_{n=-1}^{\infty}\int_{-2^{-n}}^{-2^{-\left(n+1\right)}}\abs{x}G_{\kappa}\left(x,t\right)\dx\leq C\sum_{n=-1}^{\infty}2^{-n}\left[\left(\frac{\kappa}{2^{n+1}}\right)^{\mu}+Pt2^{\omega\left(n+1\right)}\right]\\
   &=C\sum_{n=-1}^{\infty}2^{\mu}\kappa^{\mu}\left(2^{\mu-1}\right)^{n}+2^{\omega}Pt\left(2^{\omega-1}\right)^{n}\leq C_{\mu}\kappa^{\mu}+C_{\omega}Pt.
  \end{split}
 \end{equation*}
\end{proof}

We now consider the situation of Proposition~\ref{Prop:ex:dual:sum} where the integral kernel is given as the sum of different kernels

\begin{lemma}\label{Lem:int:est:conolution}
 In the situation of Proposition~\ref{Prop:ex:dual:sum} with $n=2$ one has
 \begin{enumerate}
  \item $\int_{-\infty}^{A-D}G\left(x,t\right)\dx\leq \int_{-\infty}^{A-D/2}G_{1}\left(x,t\right)\dx+\int_{-\infty}^{-D/2}G_{2}\left(x,t\right)\dx$
  \item $\int_{A-1}^{A}\abs{x-A}G\left(x,t\right)\dx\leq \int_{A-1-\kappa}^{A+\kappa}\abs{x-A}G_{1}\left(x\right)\dx+\int_{-1-\kappa}^{\kappa}\abs{x}G_{2}\left(x\right)\dx$.
 \end{enumerate}
\end{lemma}

\begin{proof}
 We consider again only the case $A=0$, while the general result follows by shifting. 
\begin{enumerate}
   \item One has
         \begin{equation*}
          \begin{split}
           \int_{-\infty}^{-D}G\left(x,t\right)\dx&=\int_{\R}\int_{\R}\chi_{\left(-\infty,-D\right]}\left(x+y\right)G_{1}\left(x,t\right)G_{2}\left(y,t\right)\dx\dy\\
           &=\int_{\R}\int_{-\infty}^{-D-y}G_{1}\left(x,t\right)G_{2}\left(y,t\right)\dx\dy\\
           &=\int_{-\frac{D}{2}}^{\infty}\int_{-\infty}^{-D-y}G_{1}\left(x,t\right)G_{2}\left(y,t\right)\dx\dy+\int_{-\infty}^{-\frac{D}{2}}\int_{-\infty}^{-D-y}G_{1}\left(x,t\right)G_{2}\left(y,t\right)\dx\dy\\
           &\leq \int_{-\infty}^{-\frac{D}{2}}G_{1}\left(x,t\right)\dx\int_{\R}G_{2}\left(y,t\right)\dy+\int_{-\infty}^{-\frac{D}{2}}G_{2}\left(y,t\right)\dy\int_{\R}G_{1}\left(x,t\right)\dx\\
           &\leq \int_{-\infty}^{-D/2}G_{1}\left(x,t\right)\dx+\int_{-\infty}^{-D/2}G_{2}\left(x,t\right)\dx
          \end{split}
         \end{equation*}
         where in the last step we used that $G_{i}$ is normalized for $i=1,2$.
   \item To prove the second estimate we first rewrite
         \begin{equation*}
          \begin{split}
           &\quad \int_{-1}^{0}\abs{x}G\left(x,t\right)\dx=\int_{\R}\int_{\R}\chi_{\left[-1,0\right]}\left(x+y\right)\abs{x+y}G_{1}\left(x,t\right)G_{2}\left(y,t\right)\dx\dy\\
           &=\int_{\R}\int_{-1-y}^{-y}\abs{x+y}G_{1}\left(x,t\right)G_{2}\left(y,t\right)\dx\dy\\
           &=\int_{-\infty}^{\kappa}\int_{-1-y}^{-y}\abs{x+y}G_{1}\left(x,t\right)G_{2}\left(y,t\right)\dx\dy
          \end{split}
         \end{equation*}
         where we used that $G_{2}=0$ on $\left\{y>\kappa\right\}$. Using also $G_{1}=0$ on $\left\{x>\kappa\right\}$ we have furthermore
         \begin{equation*}
          \begin{split}
           \int_{-1}^{0}\abs{x}G\left(x,t\right)\dx\leq \int_{-\infty}^{\kappa}\int_{-1-y}^{\kappa}\abs{x+y}G_{1}\left(x,t\right)G_{2}\left(y,t\right)\dx\dy.
          \end{split}
         \end{equation*}
         Noting that for $y<-1-\kappa$ we have $-1-y>\kappa$ and thus the $x$-integral equal zero as $G_{1}=0$ on $\left\{x>\kappa\right\}$ we obtain (also using $-1-\kappa\leq -1-y$ for $y\in \left[-1-\kappa,\kappa\right]$)
         \begin{equation*}
          \begin{split}
           \int_{-1}^{0}\abs{x}G\left(x,t\right)\dx&\leq \int_{-1-\kappa}^{\kappa}\int_{-1-\kappa}^{\kappa}\abs{x+y}G_{1}\left(x,t\right)G_{2}\left(y,t\right)\dx\dy\\
           &\leq \int_{-1-\kappa}^{\kappa}\int_{-1-\kappa}^{\kappa}\left(\abs{x}+\abs{y}\right)G_{1}\left(x,t\right)G_{2}\left(y,t\right)\dx\dy\\
           &\leq \int_{-1-\kappa}^{\kappa}\abs{x}G_{1}\left(x,t\right)\dx+\int_{-1-\kappa}^{\kappa}\abs{y}G_{2}\left(y,t\right)\dy,
          \end{split}
         \end{equation*}
         where in the last step we used $\int_{-1-\kappa}^{\kappa}G_{i}\left(x,t\right)\dx\leq \int_{\R}G_{i}\left(x,t\right)\dx=1$.
  \end{enumerate}
\end{proof}

\begin{remark}\label{Rem:est:conv}
 By induction we can prove the corresponding estimates for $n>2$ with $D/2$ replaced by $D/2^{n-1}$ and $\kappa$ replaced by $n\kappa$ (and of course summing over all $G_{i}$, $i=1,\ldots, n$ on the right hand side).
\end{remark}

\bibliographystyle{amsplain}
\bibliography{coagulation}

\end{document}